\documentclass{amsart}

\usepackage{macros}
\standardsettings
\colorcommentstrue

\usepackage{setspace}
\renewcommand{\zero}{\ee_0}

\newcommand{\Grass}{\mathcal G}
\newcommand{\Proj}{{\mathrm{proj}}}
\newcommand{\Euc}{{\mathrm{euc}}}
\newcommand{\para}{\mathrm{para}}
\newcommand{\token}{\infty}

\newcommand{\LS}{\mathbf{K}}
\newcommand{\RLS}{\mathbf{L}}
\newcommand{\SK}{\mathbf{S}}
\newcommand{\dimH}{d}

\DeclareMathOperator{\Rad}{Rad}

\draftfalse

\begin{document}
\title{Dimension rigidity in conformal structures}

\authortushar\authordavid\authormariusz


\begin{Abstract}
Let $\Lambda$ be the limit set of a conformal dynamical system, i.e. a Kleinian group acting on either finite- or infinite-dimensional real Hilbert space, a conformal iterated function system, or a rational function. We give an easily expressible sufficient condition, requiring that the limit set is not too much bigger than the radial limit set, for the following dichotomy: $\Lambda$ is either a real-analytic manifold or a fractal in the sense of Mandelbrot (i.e. its Hausdorff dimension is strictly greater than its topological dimension).

Our primary focus is on the infinite-dimensional case. An important component of the strategy of our proof comes from the rectifiability techniques of Mayer and Urba\'nski ('03), who obtained a dimension rigidity result for conformal iterated function systems (including those with infinite alphabets). In order to handle the infinite dimensional case, both for Kleinian groups and for iterated function systems, we introduce the notion of pseudorectifiability, a variant of rectifiability, and develop a theory around this notion similar to the theory of rectifiable sets.

Our approach also extends existing results in the finite-dimensional case, where it unifies the realms of Kleinian groups, conformal iterated function systems, and rational functions. For Kleinian groups, we improve on the rigidity result of Kapovich ('09) by substantially weakening its hypothesis of geometrical finiteness. Moreover, our proof, based on rectifiability, is entirely different than that of Kapovich, which depends on homological algebra.

Another advantage of our approach is that it allows us to use the ``demension'' of \v Stan$'$ko ('69) as a substitute for topological dimension. For example, we prove that any dynamically defined version of Antoine's necklace must have Hausdorff dimension strictly greater than 1 (i.e. the demension of Antoine's necklace).
\end{Abstract}
\maketitle



\tableofcontents

\draftnewpage
\section{Introduction}

Beginning with historical perspective, our problem has roots in Poincar\'e's 1883 \emph{M\'emoire sur les groupes klein\'eens}, in which Poincar\'e famously described his investigations into the surprising geometric intricacy of the ``limit curves'' of a certain family of discrete groups \cite[\68]{Poincare1}. These groups, generated by inversions in a simple closed chain of circles externally tangent to each other, had been presented to Poincar\'e by Klein in the course of their 1881-1882 correspondence \cite[p.102]{PoincareKlein}, and their limit sets are notable for being perhaps the first examples of ``naturally occurring'' fractals. For groups in this class such that the limit curve is not a circle, Poincar\'e argued that at every parabolic point, the curve had a tangent line but no osculating circle.\Footnote{In modern notation, an \emph{osculating circle} for a set $K$ at a point $\pp\in K$ is a circle $C$ such that for all $\xx\in K$, $\dist(\xx,C) = o(\|\xx - \pp\|^2)$.} He went on to conjecture that no point of the limit curve which was not a parabolic limit point could have a tangent line.\Footnote{Poincar\'e writes, ``De plus j'ai tout lieu de croire qu'il n'y \`a pas de tangente aux points de $L$ qui ne font pas partie de $P$'' \cite[p.79]{Poincare1} (``Also I have every reason to believe that there are no tangent points of $L$ that are not part of $P$''). Here $L$ is the limit curve and $P$ the set of parabolic points. In the case of loxodromic fixed points, the conjecture was proven by Fricke in 1894 \cite{Fricke} (see also \cite[pp.399-445]{FrickeKlein}). Although it seems that there has been no interest in the conjecture for the last 120 years, we give the full solution below (Remark \ref{remarkpoincareconjecture}).} Both of these facts exemplify strong senses in which the limit curve is not analytic, which were deduced from the assumption that the limit curve is not a circle. Thus they are the prototype of dynamical rigidity theorems: a dichotomy between circles and fractals.

 
 

Another historically important source of fractals was the theory of rational functions, whose Julia sets are in many ways analogous to the limit sets of Kleinian groups. In 1920, Fatou proved the following dichotomy regarding these sets \cite[p.250]{Fatou1}: if some relatively open subset of the Julia set of a rational function is a simple curve that has a tangent at every point, then the Julia set must either be a generalized circle (i.e. a geometric circle or line) or an arc of a generalized circle. Recently, the hypothesis of Fatou's theorem was weakened as follows:
\begin{theorem}[{\cite[Corollary 1 and Theorem 2]{EremenkoVanstrien}}; see also {\cite[Theorem 2]{BergweilerEremenko}}]
\label{theoremeremenkovanstrien}
If a relatively open subset of the Julia set $J$ of a rational function is contained in a smooth curve (i.e. one that that has a tangent at every point), then $J$ is contained in a generalized circle $C$. Moreover, in this case either $J = C$, $J$ is an arc in $C$, or $J$ is homeomorphic to the Cantor set.
\end{theorem}
Again this theorem can be viewed as a dichotomy: every Julia set is either contained in a generalized circle or is not contained in any smooth curve. It is interesting to ask for a strengthening of the latter case of the dichotomy by describing other fractal properties that the Julia set (or limit set in the case of Kleinian groups) must have in that case. The first result in this direction was given in 1979 by R.\,Bowen \cite{Bowen}, who showed that the limit set of a convex-cocompact quasi-Fuchsian group\Footnote{A convex-cocompact Kleinian group $G \leq \Mob(\what\C)$ is \emph{quasi-Fuchsian} if it is conjugate to some cocompact Fuchsian group (i.e. a uniform lattice) in $\Mob(S^1)$.} is either a generalized circle or has Hausdorff dimension strictly greater than 1. Bowen's result was improved by D.\,P.\,Sullivan \cite{Sullivan_discrete_conformal_groups}, P.\,J.\,Braam \cite{Braam}, R.\,D.\,Canary and E.\,C.\,Taylor \cite{CanaryTaylor}, and finally by C.\,J.\,Bishop and P.\,W.\,Jones, who proved the following:
\begin{theorem}[{\cite[Corollary 1.8]{BishopJones}}]
\label{theorembishopjones}
Let $G\leq\Mob(\what\C)$ be a finitely generated Kleinian group, and let $\Lambda$ denote the limit set of $G$. Then either $\Lambda$ is totally disconnected, $\Lambda$ is a generalized circle, or the Hausdorff dimension of $\Lambda$ is $>1$.
\end{theorem}
Since $\Lambda$ is compact, an equivalent way of saying that it is totally disconnected is to say that its topological dimension is equal to zero. So another way to phrase Bishop and Jones' result is to say that if $\TD(\Lambda)\geq 1$, then either $\Lambda$ is a generalized circle or $\HD(\Lambda) > 1$.\Footnote{Here and in the sequel, $\HD(S)$ denotes the Hausdorff dimension of a set $S$ and $\TD(S)$ denotes its topological dimension. The unsubscripted notation $\dim(S)$ will only be used when unambiguous, e.g. when $S$ is a vector space.} This way of looking at things suggests a natural generalization, namely that the ``second case'' of the dichotomy may be described by the criterion that the Hausdorff dimension is strictly greater than the topological dimension. This point of view was introduced in \cite{MayerUrbanski}, where the following theorem was proven:

\begin{theorem}[{\cite[Theorem 1.2]{MayerUrbanski}}; generalized by Theorem \ref{theorem2} below]
\label{theoremmayerurbanski}
Let $(u_a)_{a\in E}$ be a (finite or infinite, hyperbolic) conformal iterated function system (CIFS) on $\R^d$ $(3\leq d < \infty)$ which satisfies the open set condition, and let $J$ denote the limit set of $(u_a)_{a\in E}$. Suppose that $\HD(\cl J\butnot J) < \HD(J)$.\Footnote{An easy-to-check equivalent condition for this inequality is described in \cite{MayerUrbanski}.} Then either $\HD(J) > k := \TD(\cl J)$, or $\cl J$ is contained in a generalized $k$-sphere.
\end{theorem}

\begin{remark*}
If $(u_a)_{a\in E}$ is a finite parabolic CIFS (see \cite{MauldinUrbanski3} for the definition) with limit set $J$, then there is an associated infinite hyperbolic CIFS $(\w u_a)_{a\in \w E}$ whose limit set $\w J$ is equal to $J$ minus a countable set of points \cite[\65]{MauldinUrbanski3}. So by applying Theorem \ref{theoremmayerurbanski} to the CIFS $(\w u_a)_{a\in \w E}$, one can see that either $\HD(J) > k := \TD(J)$, or $J$ is contained in a generalized $k$-sphere.
\end{remark*}

The proof of Theorem \ref{theoremmayerurbanski} uses the tools of rectifiability theory. By contrast, the following rigidity theorem due to M.\,Kapovich was proven using homological algebra:

\begin{theorem}[{\cite[Theorem 1.3]{Kapovich2}}; generalized by Theorem \ref{theorem1} below]
\label{theoremkapovich}
Let $G \leq \Mob(\R^d)$ be a nonelementary geometrically finite Kleinian group $(d < \infty)$, and let $\Lambda$ denote the limit set of $G$. Then either $\HD(\Lambda) > k := \TD(\Lambda)$, or $\Lambda$ is a generalized $k$-sphere,\Footnote{A subset of $\R^d$ is called a \emph{generalized $k$-sphere} if it is either a $k$-dimensional plane or a geometric sphere contained in a $(k + 1)$-dimensional affine subspace of $\R^d$.}
\end{theorem}
\noindent Kapovich's paper extended a number of earlier results. For more on the history of this result and for similar results concerning isometry groups of negatively curved spaces, see \cite[p.2]{Kapovich2}.

Turning back to Julia sets, the first Hausdorff dimension result was proven by D.\,P.\,Sullivan in 1982, as the result of an effort to establish a ``dictionary'' between Kleinian groups and rational functions \cite[p.405]{Sullivan_wandering}, about which we shall say more below. Sullivan showed that if the Julia set of a hyperbolic rational map is a Jordan curve, then it is either a generalized circle or has Hausdorff dimension $>1$. Sullivan's result was improved by D.\,H.\,Hamilton, who proved the following:

\begin{theorem}[{\cite[Theorem 1]{Hamilton2}}]
\label{theoremhamilton}
If the Julia set of a rational map is a Jordan curve, then it is either a generalized circle or has Hausdorff dimension $>1$.
\end{theorem}

\noindent For polynomials, the assumption that the Julia set is a Jordan curve can be weakened to merely assume that the Julia set is connected; see \cite[p.168]{Urbanski7}, where this is proven by combining the results of \cite{DouadyHubbard, Przytycki, Zdunik}.

The parallels between Theorems \ref{theoremmayerurbanski}-\ref{theoremhamilton} harmonize well with many other similarities between the fields of Kleinian groups, iterated function systems, and rational functions. The idea of systematically studying the analogies between the limit sets of Kleinian groups and the Julia sets of rational functions is known as Sullivan's dictionary, since it was introduced by D.\,P.\,Sullivan \cite[p.405]{Sullivan_wandering} (see \cite{McMullen_classification,McMullen_renormalization,Sullivan_discrete_conformal_groups,Sullivan_seminar, Sullivan_conformal_dynamical,Urbanski4} for more discussion of Sullivan's dictionary). Sullivan's dictionary has been the inspiration for many theorems both in the theory of Kleinian groups and the theory of rational functions, including the result of Sullivan mentioned above. In many cases, similar proofs work to demonstrate theorems on both sides of Sullivan's dictionary.

\begin{table}
\begin{tabular}{|c|c|c|}
\hline
\spc{\textbf{Kleinian groups and} \\ \textbf{their generalizations}}
&
\spc{\textbf{Conformal IFSes and} \\ \textbf{GDMSes\footnotemark}}
&
\spc{\textbf{Rational and} \\ \textbf{meromorphic functions}}
\\ \hline
\hline
Limit set
&
Closure of limit set
&
Julia set
\\ \hline
Radial limit set
&
Limit set
&
Radial Julia set \cite{Rempe}
\\ \hline
Poincar\'e exponent
&
Unique zero of pressure
&
\spc{Poincar\'e exponent \cite{Przytycki2} \\ Unique zero of pressure \cite{PRS2} \\ Hyperbolic dimension \cite[p.320]{PrzytyckiUrbanski} \\ Dynamical dimension \cite[p.320]{PrzytyckiUrbanski} \\ (these are all equal\footnotemark)}
\\ \hline
\spc{Bishop--Jones theorem \\ \cite[Theorem 1]{BishopJones}}
&
\spc{Bowen's formula \\ \cite[Theorem 3.15]{MauldinUrbanski1} \\ \cite[Theorem 10.2]{MSU}}
&
\spc{Bowen-type formula \\ \cite{BKZ, Przytycki2, PRS2, Rempe}}
\\ \hline
Discreteness assumption
&
Open set condition
&
\\ \hline
Convex-cobounded group
&
Finite IFS/GDMS \cite{MauldinUrbanski1, MauldinUrbanski2}
&
\spc{NCP map \cite{Urbanski6} with no parabolic points \\ Special case: \\ Hyperbolic rational map \cite[\62]{Urbanski4}}
\\ \hline
\spc{Geometrically finite group \\ \cite{Bowditch_geometrical_finiteness}}
&
Finite parabolic IFS \cite{MauldinUrbanski3}
&
\spc{NCP map \cite[Definition 4.1]{Urbanski4}, \cite{Urbanski6} \\ Special cases: \\ Parabolic rational map \cite[\63]{Urbanski4} \\ Geometrically finite rational map \cite{McMullen_conformal_2}}
\\ \hline
\spc{Patterson--Sullivan measure \\ \cite{Patterson2, Sullivan_density_at_infinity}}
&
\spc{$\delta$-conformal measure \\ \cite[p.10]{MauldinUrbanski1}}
&
\spc{$\delta$-conformal measure \\ \cite[Definition 1.4]{Urbanski4}, \cite[Theorem 3]{Sullivan_conformal_dynamical}}
\\ \hline
\spc{Patterson density of \\ a Gibbs cocycle \cite[p.3]{PPS}}
&
\spc{Equilibrium/Gibbs state of \\ a potential function \cite[\62.2]{MauldinUrbanski2}}
&
\spc{Equilibrium/Gibbs state of \\ a potential function \cite{KotusUrbanski2, MayerUrbanski2}} 
\\ \hline
\end{tabular}
\vspace{.2 in}
\caption{A three-way dictionary between the theories of Kleinian groups, dynamics of rational functions, and conformal iterated function systems, extending Sullivan's Dictionary, which consists of the first and last columns. We put the IFS column in the middle because in some sense the theory of IFSes ``interpolates'' between the theory of rational functions and the theory of Kleinian groups; for example, the fact that it extends to higher dimensions is shared with Kleinian groups, while in the theory of rational functions it is often useful to consider an infinitely generated IFS generated by inverse branches of iterates of a rational function.}
\label{tabledictionary}
\end{table}

\addtocounter{footnote}{-1}

\footnotetext{GDMS stands for ``graph directed Markov system'', a concept introduced and studied extensively in \cite{MauldinUrbanski2}.}

\addtocounter{footnote}{1}

\footnotetext{see \cite[Theorem 12.3.11]{PrzytyckiUrbanski} and \cite{Przytycki2, PRS2, Rempe}}

Although the theory of iterated function systems is not traditionally considered to be part of Sullivan's dictionary, the similarities between it and the other two fields seem significant enough that we include it as a column in our version of Sullivan's dictionary, Table \ref{tabledictionary}. Thus in our framework, Theorems \ref{theoremmayerurbanski}-\ref{theoremhamilton} can all be regarded as variations of a more general ``meta-theorem'' which applies to all conformal dynamical systems. In fact, in this paper we will prove such a meta-theorem (namely Theorem \ref{theoremrigidity}) and show that it implies a sequence of theorems (namely Theorems \ref{theorem1}-\ref{theorem3}) which are similar in spirit to Theorems \ref{theoremmayerurbanski}-\ref{theoremhamilton}.

The main purpose of this paper is not just to unify these results, but also to extend them to infinite dimensions. Such an extension can be made both in the Kleinian groups and IFS settings, although not in the rational function setting (since holomorphic maps in several complex variables are not conformal). Some motivation for studying the theory of infinite-dimensional Kleinian groups can be found in \cite[\61.3]{DSU}. The strategy of our proof is to extend the rectifiability argument of \cite{MayerUrbanski} to infinite dimensions by introducing the notion of \emph{pseudorectifiability}, a notion which agrees with the notion of rectifiability for finite-dimensional sets but not for infinite-dimensional ones. We develop the theory of pseudorectifiable sets to the extent of its applicability to our rigidity proof. As a side note, we also show that in infinite dimensions, the notions of pseudorectifiablility and rectifiability already disagree in the realm of limit sets of conformal iterated function systems satisfying the strong open set condition. It would be interesting to study the pseudorectifiability condition further, in the context of nonlinear functional analysis and geometric analysis in infinite dimensions.

Another important goal of this paper is to improve the finite-dimensional versions of these results by replacing the topological dimension with the ``demension'' of M.\,A.\,\v Stan$'$ko \cite{Stanko}. The demension of a set in Euclidean space is an integer which is always at least the topological dimension of that set and is sometimes strictly greater. A standard example is Antoine's necklace, which has demension 1 but topological dimension 0. For more on demension see \6\ref{sectiondem}, in which we state a theorem which implies that any ``dynamically defined'' version of Antoine's necklace has Hausdorff dimension $>1$ (Corollary \ref{corollaryantoine}).

{\bf Acknowledgements.} The authors thank R.\,J.\,Daverman and G.\,A.\,Venema for permission to re-use a figure from their book. The research of the first-named author was supported in part by a 2014-2015 Faculty Research Grant from the University of Wisconsin--La Crosse. The research of the second-named author was supported in part by the EPSRC Programme Grant EP/J018260/1. The research of the third-named author was supported in part by the NSF grant DMS-1361677. The authors thank the anonymous referee for valuable comments.

\section{Statement of results}
\label{sectionstatements}

\begin{convention}
Throughout this paper,
\begin{itemize}
\item All measures and sets are assumed to be Borel.
\item For each $\delta\geq 0$, $\scrH^\delta$ denotes the $\delta$-dimensional Hausdorff measure.
\item $\HH$ denotes a real separable Hilbert space (either finite- or infinite-dimensional).
\item $\dimH = \dim(\HH)\in\N\cup\{\infty\}$, and $1\leq k\leq d$ is fixed ($k\neq\infty$).
\item If $\mu$ is a measure on $X$ and $f:X\to Y$, then $f(\mu) = \mu\circ f^{-1}$ denotes the image measure.
\end{itemize}
\end{convention}

\begin{convention}
\label{conventionimplied}
The symbols $\lesssim_\times$, $\gtrsim_\times$, and $\asymp_\times$ will denote coarse multiplicative asymptotics. For example, $A\lesssim_{\times,K} B$ means that there exists a constant $C > 0$ (the \emph{implied constant}), depending only on $K$, such that $A\leq C B$. In general, dependence of the implied constant(s) on universal objects such as those given in the hypotheses of the main theorems will be omitted from the notation.

The notation $A \asymp_\times B$ should not be confused with the notation $\mu \asymp \nu$, which as usual means that the measures $\mu$ and $\nu$ are equivalent, i.e. each is absolutely continuous to the other.
\end{convention}

\begin{convention}
$\NN(A,\epsilon)$ denotes the open $\epsilon$-neighborhood of $A$, i.e. $\NN(A,\epsilon) = \{x : \dist(x,A) < \epsilon\}$.
\end{convention}

In the following theorems, we use somewhat unconventional notation: we use $\LS$ to denote the first row of Table \ref{tabledictionary}, and we use $\RLS$ to denote the second row. The reason for this is that if we used a more standard notation, such as $\Lambda$ or $L$ for the first row and $\Lr$ or $L_r$ for the second row, then this notation would look somewhat awkward when applied to the second column of Table \ref{tabledictionary}: $\Lambda$ or $L$ would denote the closure of the limit set and $\Lr$ or $L_r$ would denote the limit set. So instead, we use the letters $\LS$ and $\RLS$, which hopefully have less connotative baggage. (The fact that $\LS$ suggests a compact set is a connotation we want to keep, since in the hypotheses of our theorems we assume that $\LS$ is compact.)

For definitions of the terms used in Theorems \ref{theorem1}-\ref{theorem3}, see the proof of Lemma \ref{lemmaradial} below.

\begin{theorem}
\label{theorem1}
Let $G$ be a group of M\"obius transformations of $\HH$, and let $\LS$ and $\RLS$ denote the limit set and the radial limit set of $G$, respectively. If $\dimH = \infty$, assume additionally that $\LS$ is compact. Let $\delta = \HD(\RLS)$ and $k = \TD(\LS)$, and assume that
\begin{equation}
\label{mainassumption}
\scrH^\delta(\LS\butnot\RLS) = 0
\end{equation}
(which holds, for example, if $\HD(\RLS) < \HD(\LS)$). Then the following dichotomy holds: either $\delta > k$, or $\LS$ is a generalized $k$-sphere.
\end{theorem}

\begin{theorem}
\label{theorem2}
Let $(u_a)_{a\in E}$ be a (finite or infinite, hyperbolic) conformal iterated function system on $\HH$ which satisfies the open set condition. Let $\RLS$ denote the limit set of $(u_a)_{a\in E}$, and let $\LS$ denote the closure of $\RLS$. If $\dimH = \infty$, assume additionally that $\LS$ is compact and that $(u_a)_{a\in E}$ satisfies the strong open set condition. Let $\delta = \HD(\RLS)$ and $k = \TD(\LS)$, and assume that \eqref{mainassumption} holds. Then the following dichotomy holds: either $\delta > k$, or $\LS$ is the closure of a relatively open and relatively compact subset of a real-analytic $k$-dimensional manifold $M\subset\HH$. If $\dimH\geq 3$ (or more generally if $(u_a)_{a\in E}$ consists of M\"obius transformations), then in the latter case $M$ is a generalized $k$-sphere. If $(u_a)_{a\in E}$ consists of similarities, then in the latter case $M$ is a $k$-plane.
\end{theorem}

\begin{theorem}
\label{theorem3}
Let $T:\what\C\to\what\C$ be a rational function, and let $\LS$ and $\RLS$ denote the Julia set and the radial Julia set of $T$, respectively. Let $\delta = \HD(\RLS)$ and $k = \TD(\LS)$, and assume that \eqref{mainassumption} holds. Then the following dichotomy holds: either $\delta > k$, or $\LS$ is either a generalized circle or a segment of a generalized circle (if $k = 1$) or the entire Riemann sphere (if $k = 2$).
\end{theorem}

\begin{remark*}
Theorem \ref{theorem1} implies Theorem \ref{theoremkapovich}, which in turn implies Theorem \ref{theorembishopjones} modulo \cite[Theorem 1.2]{BishopJones}. Theorem \ref{theorem2} implies Theorem \ref{theoremmayerurbanski}. Theorem \ref{theorem3} doesn't quite imply Theorem \ref{theoremhamilton}, due to the additional hypothesis \eqref{mainassumption}, about which we will say more below.
\end{remark*}


\begin{remark*}
It is worth noting that in Theorems \ref{theorem1}-\ref{theorem3}, the weaker inequality $\delta\geq k$ follows directly from \eqref{mainassumption} and the classical Szpilrajn theorem:
\begin{theorem}[Szpilrajn's theorem, {\cite[Theorem VII.2]{HurewiczWallman}}]
\label{theoremsz}
Let $X$ be a metric space and let $k = \TD(X)$. Then $\scrH^k(X) > 0$. In particular, $\HD(X)\geq \TD(X)$.
\end{theorem}
\noindent Letting $X = \LS$, we get $\scrH^k(\LS) > 0$. On the other hand, if $\delta < k$ then \eqref{mainassumption} implies that $\scrH^k(\LS\butnot \RLS) = \scrH^k(\RLS) = 0$, a contradiction. So $\delta\geq k$.
\end{remark*}

\begin{remark*}
When $k = d$, Theorems \ref{theorem1}-\ref{theorem3} can be deduced as corollaries from the following classical result:
\begin{theorem}[{\cite[Theorem 1.8.10]{Engelking}}]
\label{theoreminterior}
Let $M$ be a $k$-dimensional topological manifold and suppose that $S\subset M$ satisfies $\TD(S) = k$. Then $S$ has nonempty interior relative to $M$.
\end{theorem}
\end{remark*}

Theorems \ref{theoremsz} and \ref{theoreminterior} will both be used in the proof of Theorems \ref{theorem1}-\ref{theorem3}.

{\bf Possible weakenings.} We now discuss to what degree it is possible to weaken the hypotheses of Theorems \ref{theorem1}-\ref{theorem3}.

{\bf The hypothesis \eqref{mainassumption}.}
For many examples, this hypothesis is satisfied trivially, since the set $\LS\butnot\RLS$ has Hausdorff dimension zero. In particular, this is true for geometrically finite Kleinian groups (e.g. \cite[Theorem 12.4.5]{DSU}), finitely generated conformal IFSes (trivially since $\LS = \RLS$), topological Collet--Eckmann rational functions \cite[p.139, para.3]{PrzytyckiRivera}, rational functions with no recurrent critical points \cite[Theorem 6.1]{Urbanski6}, and certain more general classes of rational functions \cite[Corollary 6.3]{RiveraShen}. However, the hypothesis cannot be removed entirely, as the following examples illustrate:
\begin{itemize}
\item By modifying the construction of \cite{Patterson3}, one can show that for any $\epsilon > 0$ and for any compact nowhere dense $F\subset\HH$, there exists $G$ a discrete group of M\"obius transformations such that $\HD(\RLS) \leq\epsilon$ but $\LS\butnot\RLS = G(F)$. By letting $F$ be a rectifiable set which is not a generalized sphere, one gets a counterexample to a hypothetical generalization of Theorem \ref{theorem1}. A similar construction would give a counterexample to a hypothetical generalization of Theorem \ref{theorem2}.
\item In \cite{AstalaZinsmeister}, a quasi-Fuchsian group is constructed such that $\LS$ is rectifiable but not a generalized circle, and in \cite{Bishop2}, another quasi-Fuchsian group is constructed such that $\HD(\LS) = \TD(\LS) = 1$ but $\LS$ is not rectifiable.
\end{itemize}
It is an open question whether the hypothesis \eqref{mainassumption} can be replaced by the hypothesis that $G$ is finitely generated. The question is solved for $\dimH = 2$ by Theorem \ref{theorembishopjones}, but higher dimensions are still open (cf. \cite[Conjecture 1.4]{Kapovich2} for a related conjecture).

It is also open whether Theorem \ref{theorem3} holds if the hypothesis \eqref{mainassumption} is removed; cf. \cite[p.207]{Bishop3}, where the same question is asked but with the assumption $\TD(\LS) = 1$ replaced by the slightly stronger assumption that $\LS$ is connected. If this hypothesis is strengthened even further to the assumption that $\LS$ is a Jordan curve, then the question is solved by Theorem \ref{theoremhamilton}. In the other direction, if ``rational function'' is replaced by ``transcendental meromorphic function'' or ``transcendental entire function'', then the answer is known to be negative \cite{Hamilton,Bishop_non_rigidity}.

{\bf The open set condition and strong open set condition.}
The only reason that these conditions are needed in Theorem \ref{theorem2} is to ensure that Bowen's Formula holds (cf. \cite[Theorem 3.15]{MauldinUrbanski1}, \cite[Theorem 10.2]{MSU}). So if $(u_a)_{a\in E}$ is an iterated function system for which Bowen's Formula holds, then the open set hypothesis is not necessary. For example, this holds when $(u_a)_{a\in E}$ is chosen at random from a family of iterated function systems satisfying the transversality condition \cite{SSU}. Note that if $\dimH = \infty$, then a counterexample to Bowen's formula which satisfies the open set condition but not the strong open set condition is given in \cite[Theorem 11.1]{MSU}.

{\bf The assumption that $T$ is a rational function.}
It is natural to weaken this assumption to the hypothesis that $T:\C\to\what\C$ is a meromorphic function. As it turns out, our proof is valid assuming that $T$ is a meromorphic function which admits a $\delta$-conformal measure. (We leave the details to the reader.) Sufficient conditions for the existence of a $\delta$-conformal measure may be found in \cite{MayerUrbanski2}.

\section{Demension rigidity}
\label{sectiondem}
This section concerns an improvement of Theorems \ref{theorem1}-\ref{theorem3} (in the finite-dimensional case) in which the topological dimension is replaced by the ``demension'' (short for ``dimension of embedding'') of M.\,A.\,\v Stan$'$ko \cite{Stanko}. It can be skipped on a first reading, since it will not be used in the proofs of Theorem \ref{theorem1}-\ref{theorem3}.

\subsection{Demension}
Our rigidity theorems \ref{theorem1}-\ref{theorem3} are in a sense motivated by Theorem \ref{theoremsz}, which shows a pre-existing relation (independent of dynamics) between the Hausdorff dimension of a set $\SK\subset\HH$ and its topological structure. A natural question is whether Theorem \ref{theoremsz} is a complete description of the relation between Hausdorff dimension and topology. The answer is in a sense yes, and in a sense no. The sense in which the answer is yes is that any set $\SK\subset\HH$ is homeomorphic to a set $\SK'\subset\HH$ such that $\HD(\SK') = \TD(\SK') = \TD(\SK)$. This is the converse to Theorem \ref{theoremsz}, i.e. the other direction of Szpilrajn's theorem \cite[Theorem VII.5]{HurewiczWallman}. The sense in which the answer is no is that the homeomorphism between $\SK$ and $\SK'$ may not be extendable to all of $\HH$.

The best way to illustrate this is with an example. Consider Antoine's necklace construction \cite{Antoine} (cf. Figure \ref{figureantoine}), in which a torus in $\R^3$ is replaced by a chain of linked tori, each of which is then replaced by a smaller chain of linked tori, and so on infinitely until the intersection is a compact set $\SK$ (a ``necklace'') homeomorphic to the Cantor set. If $\gamma$ is a loop through the hole of the first torus, then $\gamma$ is not contractible in $\R^3\butnot\SK$ \cite[\62]{Blankinship}. On the other hand, if $\SK'\subset\R^3$ is the standard Cantor set, then $\R^3\butnot\SK'$ is simply connected. Thus, the homeomorphism between $\SK$ and $\SK'$ cannot be extended to all of $\R^3$. For this reason, $\SK$ is sometimes called a ``wild'' Cantor set.

\begin{figure}
\centerline{\mbox{\includegraphics[scale=.63]{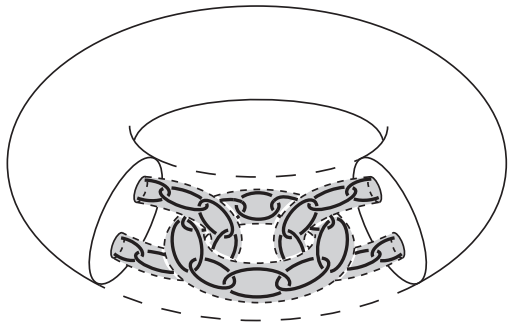}}}
\caption{The first- and second-level iterates of Antoine's necklace construction. In the IFS version of this construction, the generators for the IFS are similarities that send the large torus into the smaller tori. Since the IFS satsifies the strong separation condition, its limit set $\LS$ will be totally disconnected. On the other hand, if $\gamma$ is a loop through the ``hole'' of the large torus, then $\gamma$ is not homotopic to zero in $\HH\butnot\LS$, and thus $\pi_1(\HH\butnot\LS)\neq 0$ (see e.g. \cite[\62]{Blankinship}). Figure used with permission from \cite[Figure 2.5]{DavermanVenema}.}
\label{figureantoine}
\end{figure}

Given a compact set $\SK\subset\HH$, what is the smallest possible Hausdorff dimension among sets homeomorphic to $\SK$ such that there is a homeomorphism extendible to all of $\HH$? This question was answered by J.\,V\"ais\"al\"a \cite{Vaisala3}: it is the \emph{demension} of the set $\SK$:

\begin{definition}
\label{definitiondem}
The \emph{demension} of a set $\SK\subset\HH$, denoted $\dem(\SK)$, is the largest integer $k$ such that there exists a piecewise linear polyhedron $\PP$ of dimension $(d - k)$ and $\epsilon > 0$ such that for every continuous map $F:\HH\times[0,1]\to\HH$ satisfying the following conditions:
\begin{itemize}
\item[(I)] (Isotopy) For each $t\in [0,1]$, the map $F_t(\xx) = F(\xx,t)$ is a homeomorphism, and $F_0$ is the identity;
\item[(II)] For all $t\in [0,1]$,
\[
\|F_t\| := \sup_{\xx\in\HH} \dist(\xx,F_t(\xx)) \leq \epsilon;
\]
\item[(III)] $F_t(\xx) = \xx$ for all $\xx\in\HH\butnot\thickvar{\PP\cap\SK}\epsilon$ and $t\in [0,1]$;
\end{itemize}
we have $F_1(\PP)\cap \SK \neq \emptyset$. Loosely speaking, $\dem(\SK)$ is the largest integer $k$ such that there is a piecewise linear polyhedron $\PP$ of dimension $(d - k)$ whose sufficiently small isotopic perturbations all intersect $\SK$.
\end{definition}

We remark that this definition of demension is not the usual one (which is less relevant for our purposes), but is equivalent to the usual one by \cite[Proposition 1.2]{Edwards}.


Let us first address the question of how to compute the demension. The following observation is good enough to compute the demension of Antoine's necklace and its higher-dimensional generalizations (see \cite{Blankinship}):

\begin{observation}
\label{observationdemension}
If $\SK\subset\HH$ is a compact set such that
\begin{equation}
\label{homotopy}
\pi_1(\HH\butnot\SK) \neq 0,
\end{equation}
then $\dem(\SK)\geq d - 2$.
\end{observation}
\begin{proof}
Let $\gamma$ be a piecewise linear representation of a noncontractible loop in $\HH\butnot\SK$, and let $\PP$ be a piecewise linear (possibly self-intersecting) $2$-simplex whose boundary is $\gamma$. Fix $0 < \epsilon < \dist(\gamma,\SK)$. Then if $F$ is an isotopy satisfying (I)-(III), then $F_1(\PP)$ is a $2$-simplex whose boundary is $\gamma$. Since $\gamma$ is noncontractible, it follows that $F_1(\PP)\cap\SK\neq\emptyset$.
\end{proof}

In fact, a nearly complete characterization of the demension is known, but its proof requires deep topological results.

\begin{theorem}[{\cite[Theorem 1.4]{Edwards}}, {\cite[p.598]{Rushing}}]
Let $\SK\subset\HH$ be a compact set.
\begin{itemize}
\item If $\SK$ is locally homotopically 1-coconnected, then $\dem(\SK) = \TD(\SK)$.
\item If $\SK$ is not locally homotopically 1-coconnected then $\dem(\SK) = d - 2 \geq \TD(\SK)$ unless the following case holds: $d = 3$, $\dem(\SK) = 2 > \TD(\SK) = 1$.
\end{itemize}
\end{theorem}

Here, the set $\SK$ is said to be \emph{locally homotopically 1-coconnected} if for every $x\in \HH$ and neighborhood $U$ of $x$, there exists a neighborhood $V$ of $x$ such that every loop in $V\butnot\SK$ is contractible in $U\butnot\SK$.

The ``exceptional'' case $d = 3$, $\dem(\SK) = 2 > \TD(\SK) = 1$ can in fact occur, as demonstrated by H.\,G.\,Bothe \cite{Bothe} and independently by D.\,R.\,McMillan, Jr. and W.\,H.\,Row, Jr. \cite{McMillanRow}.

\subsection{Demension and Hausdorff dimension}
As mentioned earlier, one characterization of the demension of a compact set $\SK\subset\HH$ is that it is the infimum of the Hausdorff dimension of the image of $\SK$ under a homeomorphism of $\HH$ \cite{Vaisala3}. In particular, this implies that $\HD(\SK)\geq \dem(\SK)$. For completeness we provide an elementary proof of a slightly stronger version of this assertion:

\begin{theorem}[{\cite[Theorem 6.15]{LuukkainenVaisala}}]
\label{theoremszprime}
Let $\SK\subset\HH$ be any set, and let $k = \dem(\SK)$. Then $\scrH^k(\SK) > 0$. In particular, $\HD(\SK) \geq \dem(\SK)$.
\end{theorem}
\begin{proof}
Let $\PP\subset\HH$ and $\epsilon > 0$ be as in Definition \ref{definitiondem}, let $\Isom(\HH)$ denote the isometry group of $\HH$, and fix $g\in\Isom(\HH)$. If $g$ is sufficiently small, then there exists an isotopy $(F_t)_t$ satisfying the requirements of Definition \ref{definitiondem} such that $F_1 = g$ on $U = \thickvar{\PP\cap\SK}{\epsilon/2}$ and $\|F_1\| < \dist(\PP\butnot U,\SK\butnot U)$. Since $k = \dem(\SK)$, we have $F_1(\PP)\cap\SK\neq\emptyset$ and thus $g(\PP)\cap\SK\neq\emptyset$.

Let $A_1,\ldots,A_m$ be a collection of $(d - k)$-dimensional affine subspaces of $\HH$ such that $\PP\subset \bigcup_1^m A_i$. Then for all sufficiently small $g\in\Isom(\HH)$, there exists $i = 1,\ldots,m$ such that $g(A_i)\cap\SK \neq\emptyset$. By Fubini's theorem, for some $i = 1,\ldots,m$ and for a positive measure set of $h\in\SO(\HH)$, the set of $\vv\in\HH$ such that $(h(A_i) + \vv)\cap\SK\neq\emptyset$ is of positive measure. Equivalently, $\scrH^k(\pi_{V(h)}(\SK)) > 0$, where $V(h)$ is the orthogonal complement of the linear part of $h(A_i)$ (which satisfies $\dim(V(h)) = k$). It follows that $\scrH^k(\SK) > 0$.
\end{proof}

Now we come to the main theorem of this section, which says that strict inequality holds in Theorem \ref{theoremszprime} for ``dynamically defined'' fractals:

\begin{theorem}
\label{theoremprime}
The results of \62 are still true if $\TD$ is replaced by $\dem$.
\end{theorem}

\begin{corollary}
\label{corollaryantoine}
For a ``dynamically defined'' Antoine's necklace $\LS$ (cf. Remark \ref{remarkdemdynamical} below), we have $\HD(\LS) > 1$.
\end{corollary}

The proof of Theorem \ref{theoremprime} is given in the next section, following the statement of Theorem \ref{theoremrigidity}. The idea: where Theorems \ref{theoremsz} and \ref{theoreminterior} are used in the proof of Theorems \ref{theorem1}-\ref{theorem3}, instead use Theorem \ref{theoremszprime} and Theorem \ref{theoreminteriorprime} below. Actually, the proof requires a slight strengthening of Theorem \ref{theoremszprime}, specifically in the proof of Lemma \ref{lemmafederermattila}, which is used to prove Theorem \ref{theoremrigidity}.

\begin{theorem}
\label{theoreminteriorprime}
Let $M\subset\HH$ be a $k$-dimensional compact smooth manifold-with-boundary, and suppose that $\SK\subset M$ satisfies $\dem(\SK) = k$. Then $\SK$ has nonempty interior relative to $M$.
\end{theorem}
\begin{proof}
Let $\PP\subset\HH$ and $\epsilon > 0$ be as in Definition \ref{definitiondem}. Without loss of generality, we can assume that $\PP$ intersects $M$ only finitely many times, and that each intersection is transversal. (This is due to the dimension relation $\dim(\PP) + \dim(M) = \dimH$.) Write $\PP\cap M = \{\xx_1,\ldots,\xx_m\}$, and without loss of generality suppose $\dist(\xx_i,\xx_j) \geq \epsilon$ for all $i\neq j$. Fix $\vv_1,\ldots,\vv_m\in\HH$ small; then there exists an isotopy $(F_t)_t$ satisfying the requirements of Definition \ref{definitiondem} such that $F_1(\xx) = \xx + \vv_i$ for all $\xx\in B_i := B(\xx_i,\epsilon/3)$ and such that $\dist(\xx,F_1(\xx)) \leq \dist(\PP\butnot \bigcup_i B_i,\SK\butnot\bigcup_i B_i)$. Since $k = \dem(\SK)$, we have $F_1(\PP)\cap\SK\neq\emptyset$.

If $\vv_1,\ldots,\vv_m$ are sufficiently small, then we can write $F_1(\PP)\cap M = \{\gg_1(\vv_1),\ldots,\gg_m(\vv_m)\}$, where $\gg_1,\ldots,\gg_m:U\to M$ are smooth nonsingular maps from some neighborhood $\0\in U\subset\HH$. So for some $i$, we have $\gg_i(\vv_i)\in\SK$. Since $\vv_1,\ldots,\vv_m$ were arbitrary, there exists $i$ such that for all $\vv\in U$, we have $\gg_i(\vv)\in\SK$. But then $\SK$ contains the relatively open set $\gg_i(U)\subset M$.
\end{proof}

\begin{remark}
\label{remarkdemdynamical}
It is not hard to see that in the construction of Antoine's necklace, the small tori in the first step may be chosen to be similar copies of the large torus. (The large torus must be chosen to satisfy a certain inequality in order for this to work.) In this case, Antoine's necklace is actually the limit set of a similarity IFS, giving a nontrivial example of Theorem \ref{theoremprime}. In fact, the optimality of Theorem \ref{theoremprime} (for the case of necklaces) was proven by T.\,B.\,Rushing \cite{Rushing}, who showed that for every $\delta\in (1,3)$ there exists a similarity-generated necklace whose Hausdorff dimension is $\delta$. (He also showed that when $\delta\in\{1,3\}$, there exists a necklace of Hausdorff dimension $\delta$ which is not the limit set of a similarity IFS.)

One might wonder what other dynamically defined fractals have demension strictly larger than their topological dimension. Examples of geometrically finite Kleinian groups acting on $\R^3$ whose limit sets satisfy \eqref{homotopy} but have topological dimension zero are given in \cite{BestvinaCooper, Gusevskii, Matsumoto}. For such examples, $\dem(\SK) \geq 1$ by Observation \ref{observationdemension}, and it seems likely that equality holds.

It is interesting to ask whether the ``exceptional case'' $d = 3$, $\dem(\SK) = 2 > \TD(\SK) = 1$ can occur for $\SK$ dynamically defined. In fact, the original construction of McMillan and Row \cite{McMillanRow} is essentially an IFS construction, although their paper does not make that very clear (which is not surprising as IFSes had not been invented yet). We hope to clarify this in future work.

\end{remark}

\draftnewpage
\section{A general rigidity theorem}
In this section, we state a general rigidity theorem and show that the four theorems of the previous section all reduce to it. In this general theorem, there is no dynamical system mentioned explicitly in the hypotheses, but the hypotheses concern a measure $\mu$ on $\HH$ which should be interpreted as the conformal measure of some conformal dynamical system. To make this interpretation explicit, we include a hypothesis that a certain set $\Rad(\mu)\subset\Supp(\mu)$ has full $\mu$-measure, where $\Supp(\mu)$ denotes the topological support of $\mu$. The set $\Rad(\mu)$ is supposed to represent the radial limit set of the dynamical system which generates $\mu$, in a way made precise by Lemma \ref{lemmaradial} below.

\begin{definition}
\label{definitionradial}
Let $\AA$ be one of the following classes:
\begin{itemize}
\item $\Conf(\HH)$, the class of conformal homeomorphisms between open subsets of $\HH$,
\item $\Mob(\HH)$, the class of M\"obius transformations of $\HH$,
\item $\Sim(\HH)$, the class of similarities of $\HH$.
\end{itemize}
Note that by Liouville's theorem, if $\dimH\geq 3$ then $\Conf(\HH) = \Mob(\HH)$.\Footnote{The classical proof of Liouville's theorem by Nevanlinna \cite{Nevanlinna} is valid in the infinite-dimensional setting. We note in passing that the proof in \cite{Nevanlinna} is incomplete as it stands, as the fourth displayed equation on \cite[p.4]{Nevanlinna} is only valid under the assumption that $\alpha\neq 0$: if $\alpha = 0$, then this equation should be replaced by the equation $\rho(x) = (v,x) + \beta$, where $v$ is a constant vector and $\beta$ is a constant scalar. However, this case can be analyzed in a straightforward way, assuming that the inner product $(\cdot,\cdot)$ is positive definite.} 
Let $\mu$ be a finite measure on $\HH$, let $U\subset \HH$ be a bounded open set, and fix $\delta > 0$. A point $\pp\in\Supp(\mu)$ will be called \emph{$(\delta,\mu)$-radial} (or more properly \emph{$(\AA,U,\delta,\mu)$-radial}) if there exists a sequence of maps $g_n:U\to\HH$ of class $\AA$ such that:
\begin{itemize}
\item[(a)] For all $n$,
\begin{itemize}
\item[(1)] $\|g_n'\| := \sup_{\xx\in U} |g_n'(\xx)| \asymp_\times \inf_{\xx\in U} |g_n'(\xx)|$;\Footnote{We use the notation $|g_n'(\xx)|$ rather than $\|g_n'(\xx)\|$ because $g_n'(\xx)$ is a similarity, so $|g_n'(\xx)|$ denotes its dilatation constant.}
\item[(2)] For all $\xx,\yy\in U$,
\[
\|g_n(\yy) - g_n(\xx)\| \asymp_\times \|g_n'\|\cdot\|\yy - \xx\|;
\]
\item[(3)] For all $A\subset U$,
\begin{equation}
\label{muconformal}
\mu(g_n(A)) \geq \int_A |g_n'|^\delta \; \dee\mu;
\end{equation}
\item[(4)] $\dist(\pp,g_n(U)) \lesssim_\times \|g_n'\|$.
\end{itemize}
\item[(b)] $\|g_n'\| \to 0$.
\end{itemize}
The set of $(\delta,\mu)$-radial points is denoted $\Rad_\delta(\mu)$. If $\pp\in\Rad_\delta(\mu)$, then $C_\pp$ denotes the implied constant of (a4) plus $\diam(U)$ times the implied constant of (a2). The significance of this is that for all $n$, we have
\begin{equation}
\label{Cpdef}
g_n(U) \subset B(\pp,C_\pp\|g_n'\|).
\end{equation}
\end{definition}

\begin{lemma}
\label{lemmaradial}
Every radial limit point of a conformal dynamical system is in $\Rad_\delta(\mu)$, where $\delta$ is the ``natural'' dimension of the conformal dynamical system and $\mu$ is the $\delta$-conformal measure. More precisely, if we are in any of the following scenarios:
\begin{itemize}
\item[(1)] $G\leq\AA = \Mob(\HH)$ is a group whose limit set is compact, $\delta$ is the Poincar\'e exponent of $G$, $\mu$ is the Patterson--Sullivan measure of $G$, and $\RLS$ is the radial limit set;
\item[(2)] $(u_a)_{a\in E}$ is a regular conformal iterated function system on $\HH$ satisfying the open set condition whose generators are of class $\AA$, $\delta$ is the Bowen parameter, $\mu$ is the $\delta$-conformal measure, and $\RLS$ is the limit set;
\item[(3)] $\HH = \C$, $T:\what\C\to\what\C$ is a rational function, $\AA = \Conf(\HH)$, $\delta$ is the hyperbolic dimension, $\mu$ is the $\delta$-conformal measure, and $\RLS$ is the radial Julia set;
\end{itemize}
then (possibly after conjugating in case (3)) there exists an open set $U\subset\HH$ such that $\mu(U) > 0$ and $\RLS \subset \Rad(\AA,U,\delta,\mu)$. Moreover,
\begin{itemize}
\item[(i)] in case \text{(1)}, any bounded open set $U$ for which $\Supp(\mu)\nsubset \cl U$ may be chosen;
\item[(ii)] in case \text{(2)}, the set $U$ may be chosen so that $\Supp(\mu)\subset U$.
\end{itemize}
\end{lemma} 
The definitions of the terms used in this lemma will be given in the proof.
\begin{remark}
Lemma \ref{lemmaradial} informally says that ``the radial limit set is always contained in the $\mu$-radial set'', i.e. the inclusion $\RLS\subset\Rad_\delta(\mu)$ always holds. On the other hand, the reverse inclusion $\Rad_\delta(\mu) \subset\RLS$ may fail. For example, suppose that $\dimH < \infty$, and let $G\leq\Mob(\HH)$ be a Kleinian lattice. Then $\mu$ is Lebesgue measure, and $\Rad(\mu) = \HH$ contains the parabolic points of $G$ as well as the radial limit points of $G$. This is because it is impossible to distinguish parabolic points from radial points if one only knows what $\mu$ is and not what the group $G$ is.
\end{remark}
\NPC{Proof}
\begin{proof}[Proof in case \text{(1)}]
Let $G\leq\Mob(\HH)$ be a group of M\"obius transformations of $\HH$. We recall that for each $g\in\Mob(\HH)$, the \emph{Poincar\'e extension} of $g$ is the unique M\"obius transformation $\what g\in\Mob(\R\oplus\HH)$ such that $\what g\given\HH = g$ and $\what g(\H) = \H$, where $\H = (0,\infty)\times\HH$ is the upper half-space model of hyperbolic space. The \emph{limit set} of $G$ is the set
\[
\LS = \{\xx\in\what\HH : \exists g_n\in G \;\; \what g_n(\zero) \to \xx\},
\]
where $\zero = (1,\0) \in\H$. The \emph{Poincar\'e exponent} of $G$ is the infimum of all $s\geq 0$ such that the series
\[
\Sigma_s(G) = \sum_{g\in G} e^{-s\dist_\H(\zero,\what g(\zero))}
\]
converges, where $\dist_\H$ denotes the hyperbolic metric on $\H$. Since, in the assumptions of this lemma, the limit set $\LS$ is compact, there exists a \emph{Patterson--Sullivan measure} $\mu$ (cf. \cite[Th\'eor\`eme 5.4]{Coornaert} or \cite[Theorem 15.4.6]{DSU}), which is a measure on $\LS$ satisfying the transformation equation
\begin{equation}
\label{pattersonsullivan}
\mu(g(A)) = \int_A |g'|^\delta \;\dee\mu \;\; (g\in G, A\subset\what\HH).
\end{equation}
Finally, the \emph{radial limit set} $\RLS\subset\LS$ is the set of all $\pp\in\LS$ with the following property: There exists a sequence $(g_n)_1^\infty$ in $G$ such that if $H:\H\to(0,\infty)$ denotes projection onto the first coordinate, then $\what g_n(\zero)\to\pp$ and
\begin{equation}
\label{radial1}
\|\what g_n(\zero) - \pp\| \lesssim_\times H(\what g_n(\zero)) \all n.
\end{equation}
Now let $U\subset\HH$ be a bounded open set such that $\LS = \Supp(\mu)\nsubset \cl U$. Let $V\subset\HH$ be a bounded open set intersecting $\LS$ such that $\dist(U,V) > 0$; such a set exists by our assumption on $U$. Since $V\cap\LS\neq\emptyset$, there exists a loxodromic isometry $h\in G$ whose fixed points are both in $V$ \cite[Proposition 7.4.7]{DSU}, and by replacing $h$ by a sufficiently large iterate, we can assume that $V\cup h(V) = \what\HH$ (cf. \cite[Theorem 6.1.10]{DSU}).

Fix $\pp\in \RLS$, and let $\what g_n(\zero)\to\pp$ be a sequence satisfying \eqref{radial1}. For each $n$, choose $h_n\in \{\id,h\}$ such that $h_n^{-1}g_n^{-1}(\infty) \in V$. It is not hard to see that if $g_n$ is replaced by $g_n h_n$ in \eqref{radial1}, then the inequality remains valid after changing the constant $C$ appropriately. Thus, we may without loss of generality assume that $g_n^{-1}(\infty)\in V$. We now proceed to demonstrate the conditions of Definition \ref{definitionradial}:
\begin{itemize}
\item[(a1)] Since $U,V$ are bounded sets and $\dist(U,V) > 0$, the inclusion $g_n^{-1}(\infty)\in V$ implies that
\[
\sup_{\xx\in U} \|\xx - g_n^{-1}(\infty)\| \asymp_\times \inf_{\xx\in U} \|\xx - g_n^{-1}(\infty)\|.
\]
Applying the formula
\[
|g_n'(\xx)| = H(\what g_n(\zero))\frac{\|\zero - g_n^{-1}(\infty)\|^2}{\|\xx - g_n^{-1}(\infty)\|^2}
\]
demonstrates that
\begin{equation}
\label{gnprimeasymp}
|g_n'(\xx)| \asymp_\times H(\what g_n(\zero)) \all \xx\in U.
\end{equation}
\item[(a2)] This follows from \eqref{gnprimeasymp} and the geometric mean value theorem (e.g. \cite[Proposition 4.2.4]{DSU} or \cite[(15) on p.19]{Ahlfors}).
\item[(a3)] This follows from \eqref{pattersonsullivan}.
\item[(a4)] This follows from \eqref{gnprimeasymp} and \eqref{radial1}.
\item[(b)] This follows from \eqref{gnprimeasymp} and the convergence $\what g_n(\zero)\to\pp$.
\qedhere\end{itemize}
%
\end{proof}
\begin{proof}[Proof in case \text{(2)}]
Let $(u_a)_{a\in E}$ be a (hyperbolic) conformal iterated function system (CIFS) on $\HH$ satisfying the open set condition (OSC). We recall (cf. \cite[p.6-7]{MauldinUrbanski1} for the case $\dimH < \infty$ and \cite[p.17]{MSU} for the case $3\leq \dimH \leq \infty$) that this means that
\begin{enumerate}[1.]
\item $E$ is a countable (finite or infinite) index set;
\item $X\subset\HH$ is a closed bounded set which is equal to the closure of its interior;
\item (Quasiconvexity) There exists $Q\geq 1$ such that for all $\xx,\yy\in X$, there exists a polygonal line $\gamma\subset X$ connecting $\xx$ and $\yy$ such that $\length(\gamma)\leq Q\|\yy - \xx\|$;
\item (Cone condition) If $\dimH < \infty$, then
\[
\inf_{\xx\in X, r\in (0,1)} \frac{\lambda_\HH(X\cap B(\xx,r))}{r^{\dimH}} > 0,
\]
where $\lambda_\HH$ denotes Lebesgue measure on $\HH$;
\item $V\subset\HH$ is an open connected bounded set such that $\dist(X,\HH\butnot V) > 0$;
\item For each $a\in E$, $u_a$ is a conformal homeomorphism from $V$ to an open subset of $V$;
\item (Open set condition) For all $a\in E$, $u_a(X) \subset X$, and the collection $(u_a(\Int(X)))_{a\in E}$ is disjoint;\Footnote{The \emph{strong open set condition} referred to in Theorem \ref{theorem2} is the additional hypothesis that $\Int(X)\cap\LS\neq\emptyset$.}
\item (Uniform contraction) $\sup_{a\in E} \sup |u_a'| < 1$, and if $E$ is infinite, $\lim_{a\in E} \sup |u_a'| = 0$;
\item (Bounded distortion property) For all $n\in\N$, $\omega\in E^n$, and $\xx,\yy\in V$,
\begin{equation}
\label{BD2}
|u_\omega'(\xx)| \asymp_\times |u_\omega'(\yy)|,
\end{equation}
where
\[
u_\omega = u_{\omega_1}\circ\cdots\circ u_{\omega_n};
\]
\end{enumerate}
For this proof, we assume in addition that the CIFS $(u_a)_{a\in E}$ is \emph{regular} (cf. \cite[p.21]{MauldinUrbanski1} ($\dimH < \infty$) and \cite[p.472]{MSU} ($3\leq\dimH\leq\infty$)), i.e.
\begin{enumerate}[10.]
\item (Regularity) There exists $\delta > 0$ such that
\begin{equation}
\label{bowen}
P(\delta) := \lim_{n\to\infty} \frac1n\log\sum_{\omega\in E^n} \|u_\omega'\|^\delta = 0.
\end{equation}
\end{enumerate}
The $\delta$ which satisfies \eqref{bowen} is called the \emph{Bowen parameter}. The \emph{limit set} of $(u_a)_{a\in E}$ is the set $\RLS = \pi(E^\N)$, where $\pi:E^\N\to X$ is the coding map
\[
\pi(\omega) = \lim_{n\to\infty} u_{\omega_1^n}(\pp_*).
\]
Here $\omega_1^n$ denotes the restriction of $\omega\in E^\N$ to $\{1,\ldots,n\}$, and $\pp_*\in X$ is an arbitrary point. The regularity of the CIFS $(u_a)_{a\in E}$ implies the existence of a \emph{$\delta$-conformal measure} $\mu$ (cf. \cite[Lemma 3.13]{MauldinUrbanski1} for the case $\dimH < \infty$ and \cite[Theorem 7.1]{MSU} for the case $3\leq \dimH \leq \infty$), which is a measure on $\RLS$ which satisfies the transformation equation
\begin{equation}
\label{deltaconformal}
\mu(A) = \sum_{a\in E} \int_{u_a^{-1}(A)} |u_a'|^\delta \;\dee\mu \;\; (a\in E, \; A\subset\RLS).
\end{equation}
Now let $\epsilon = \dist(X,\HH\butnot V) > 0$ and $U = \thickvar X{\epsilon/3}$, so that $\LS := \cl{\RLS} \subset U\subset \cl U\subset V$. Fix $\pp\in \RLS$, and find $\omega\in E^\N$ such that $\pp = \pi(\omega)$. For each $n$, let $g_n = u_{\omega_1^n}$. We now proceed to demonstrate the conditions of Definition \ref{definitionradial}:
\begin{itemize}
\item[(a1)] This is a restatement of \eqref{BD2}.
\item[(a2)] Note that if $\dimH\geq 3$, then this is a consequence of the geometric mean value theorem, but if $\dimH \leq 2$ an additional argument is needed. Fix $\xx,\yy\in U$, and let $\gamma:[0,t_0]\to\HH$ be the length parameterization the line segment connecting $g_n(\xx)$ and $g_n(\yy)$. Let $t_1$ be the largest element of $[0,t_0]$ such that $\gamma(t_1)\in g_n(\cl{\thickvar X{2\epsilon/3}})$. By the mean value theorem,
\[
\|g_n^{-1}\circ\gamma(t_1) - \xx\| \leq t_1 \max_{t\in [0,t_1]} |(g_n^{-1})'(\gamma(t))| \asymp_\times t_1 \|g_n'\|^{-1} \leq \|g_n'\|^{-1} \cdot \|g_n(\yy) - g_n(\xx)\|.
\]
If $t_1 = t_0$, then we get $\|g_n(\yy) - g_n(\xx)\| \gtrsim_\times \|g_n'\|\cdot \|\yy - \xx\|$. Otherwise, we have $g_n^{-1}\circ\gamma(t_1)\in\del\thickvar X{2\epsilon/3}$, which implies $\|g_n^{-1}\circ\gamma(t_1) - \xx\| \geq \epsilon/3$ and thus $\|g_n(\yy) - g_n(\xx)\| \gtrsim_\times \|g_n'\| \gtrsim_\times \|g_n'\| \cdot \|\yy - \xx\|$. So the asymptotic $\|g_n(\yy) - g_n(\xx)\| \gtrsim_\times \|g_n'\|\cdot \|\yy - \xx\|$ holds in either case. The proof of the other asymptotic $\|g_n(\yy) - g_n(\xx)\| \lesssim_\times \|g_n'\|\cdot \|\yy - \xx\|$ is similar.
\item[(a3)] This is immediate from \eqref{deltaconformal}.
\item[(a4)] In fact, $\pp\in g_n(\RLS) \subset g_n(U)$.
\item[(b)] This is immediate from \eqref{BD2} and the uniform contraction hypothesis.
\qedhere\end{itemize}
\end{proof}

\begin{proof}[Proof in case \text{(3)}]
Let $T:\what\C\to\what\C$ be a rational function. The \emph{Julia set} of $T$ is the set
\[
\LS = \what\C\butnot\bigcup\big\{U\subset\what\C : \text{$(T^n\given_U)_0^\infty$ is a normal family}\big\}.
\]
The \emph{hyperbolic dimension} $\delta$ of $T$ is the supremum of the Hausdorff dimensions of closed sets $F\subset\LS$ such that $T\given F$ is a conformal expanding repeller (cf. \cite[p.320]{PrzytyckiUrbanski}). 
By \cite[Theorem 12.3.11]{PrzytyckiUrbanski}, there exists a \emph{$\delta$-conformal measure} $\mu$, which is a measure on $\LS$ which satisfies the transformation equation
\begin{equation}
\label{deltaconformal3}
\int \#\big(A\cap T^{-1}(\xx)\big)\dee\mu(\xx) = \int_A |T'|^k \; \dee\mu \;\; (A\subset\LS).
\end{equation}
Finally, let $\RLS$ denote the \emph{radial Julia set} of $T$, i.e. the set of points $\pp\in\what\C$ such that there exist $\epsilon > 0$ and a sequence $n_j\to\infty$ with the following properties (cf. \cite[Definition 2.5]{Rempe}):
\begin{itemize}
\item[(I)] For each $j$, there exists $\phi_j:B_\sph(T^{n_j}(\pp),\epsilon)\to\what\C$, an inverse branch of $T^{n_j}$, such that $\phi_j(T^{n_j}(\pp)) = \pp$. Here the notation $B_\sph$ means that the ball is taken with respect to the spherical metric.
\item[(II)] $\diam(\phi_j(B_\sph(T^{n_j}(\pp),\epsilon)))\to 0$.
\end{itemize}
Now fix $\pp\in \RLS$, and let $\epsilon > 0$ be as above. By extracting a subsequence from $(n_j)_1^\infty$, we may without loss of generality suppose that $T^{n_j}(\pp)\to \qq\in\LS$. By conjugating if necessary, we can assume $\infty\notin\LS$ and in particular $\qq\neq\infty$. Let $U$ be a bounded open set containing $\qq$ such that $\cl U\subset B_\sph(\qq,\epsilon/2)$, and for each $j$ let $g_j = \phi_j\given U$. We now proceed to demonstrate the conditions of Definition \ref{definitionradial}:
\begin{itemize}
\item[(a1)] This is a consequence of the K\"oebe distortion theorem \cite[Theorem 1.4]{CarlesonGamelin}.
\item[(a2)] This is proven in the same manner as for case (2).
\item[(a3)] This is immediate from \eqref{deltaconformal3}.
\item[(a4)] In fact, $\pp\in g_j(U)$ for all sufficiently large $j$.
\item[(b)] This is immediate from (a1) and part (II) of the definition of the radial Julia set.
\qedhere\end{itemize}
\end{proof}


\noindent We are now ready to state our general rigidity theorem, and prove that it implies the results of \6\ref{sectionstatements}-\ref{sectiondem}:

\begin{theorem}
\label{theoremrigidity}
Let $\HH$ be a separable Hilbert space, and let $\mu$ be a finite measure on $\HH$ whose topological support $\SK$ is compact. Let $k = \TD(\SK)$ or $k = \dem(\SK)$, and suppose that $\scrH^k$-a.e. point of $\SK$ is $(\AA,U,k,\mu)$-radial, where $\AA$ is as in Definition \ref{definitionradial} and $U\subset\HH$ is a bounded open set such that $\mu(U) > 0$. If $\dimH = \infty$, we additionally assume that $\TD(\SK\cap U) = k$. Also suppose that $\Rad_k(\mu)$ is dense in $\SK$. Then there exists a map $h:U\to\HH$ of class $\AA$ such that $h(\SK\cap U)$ is contained in a $k$-plane.
\end{theorem}
\begin{proof}[Proof of the results of \6\ref{sectionstatements} and of Theorem \ref{theoremprime} assuming Theorem \ref{theoremrigidity}]
Let $\AA$, $\delta$, $\mu$, and $\RLS$ be as in Lemma \ref{lemmaradial}. In the case of CIFSes, this requires some justification, since the CIFS $(u_a)_{a\in E}$ is not assumed to be regular in Theorem \ref{theorem2}, and regularity is necessary to ensure the existence of a $\delta$-conformal measure $\mu$. This justification comes from \cite[Theorem 4.16]{MauldinUrbanski1} in the case $\dimH < \infty$, and from \cite[Proposition 9.2]{MSU} in the case $3\leq \dimH \leq \infty$, since by Theorem \ref{theoremsz} (or Theorem \ref{theoremszprime} if $k = \dem(\LS)$) $\scrH^k(\LS) > 0$, and by hypothesis $\HD(\LS) = k$. Thus by Lemma \ref{lemmaradial}, we get $\RLS\subset\Rad_\delta(\mu)$ after making an appropriate choice for $U$.

We claim that $\delta = \HD(\RLS)$. Indeed, this has been proven for Kleinian groups in \cite[Theorem 1.2.1]{DSU}, for CIFSes in \cite[Theorem 3.15]{MauldinUrbanski1} ($\dimH < \infty$) and \cite[Theorem 10.2]{MSU} ($3\leq\dimH\leq\infty$ and strong OSC), and for rational functions in \cite[Theorem 1.1]{Rempe} (see also \cite[Theorem 2.1]{McMullen_conformal_2}).

Now by Theorem \ref{theoremsz} (or Theorem \ref{theoremszprime} if $k = \dem(\LS)$) we have $\delta\geq k$, and if $\delta > k$ then we are done, so suppose that $\delta = k$. Then $\RLS\subset\Rad_k(\mu)$. Since Theorems \ref{theorem1}-\ref{theorem3} all include the hypothesis $\scrH^\delta(\LS\butnot\RLS) = 0$, we get $\scrH^k(\LS\butnot\Rad_k(\mu)) = 0$. Moreover, it is obvious that $\LS = \SK = \Supp(\mu)$.

If $\dimH = \infty$, then we additionally need to verify $\TD(\LS\cap U) = k$. In the case of CIFSes this follows from (ii) of Lemma \ref{lemmaradial}. In the case of Kleinian groups, it follows from the sum theorem for topological dimension \cite[Theorem 1.5.3]{Engelking} that there exists a closed set $F\propersubset \LS$ such that $\TD(F) = k$, so we get $\TD(\LS\cap U) = k$ for all $U\supset F$.

So by Theorem \ref{theoremrigidity}, there exists a map $h:U\to\HH$ of class $\AA$ such that $h(\LS\cap U)$ is contained in a $k$-plane. In particular, $\LS\cap U$ is contained in:
\begin{itemize}
\item a real-analytic $k$-dimensional manifold if $\AA = \Conf(\HH)$;
\item a generalized $k$-sphere if $\AA = \Mob(\HH)$;
\item a $k$-plane if $\AA = \Sim(\HH)$.
\end{itemize}
We complete the proof by breaking into cases:
\begin{itemize}
\item To prove Theorem \ref{theorem1}, let $(U_n)_1^\infty$ be an increasing sequence of open sets such that $\LS\nsubset\cl{U_n}$ for all $n$ but $\LS\butnot\{\pp\}\subset\bigcup_n U_n$ for some $\pp\in\LS$ (and $U_n\supset F$ if $\dimH = \infty$). For each $n$, by part (i) of Lemma \ref{lemmaradial}, there exists a generalized $k$-sphere $S_n$ which contains $\LS\cap U_n$. Since the set $S = \bigcup_N \bigcap_{n\geq N} S_n$ is the increasing union of generalized spheres of dimension $\leq k$, it is also a generalized sphere of dimension $\leq k$. But $\LS\butnot\{\pp\}\subset S$, and so since $\LS$ is perfect, $\LS\subset S$.

By Theorem \ref{theoreminterior} (or Theorem \ref{theoreminteriorprime} if $k = \dem(\LS)$), $\LS$ has nonempty interior relative to $S$. But then the boundary of $\LS$ relative to $A$ is a closed $G$-invariant set which does not contain $\LS$, hence it is empty. Thus $\LS = S$.
\item To prove Theorem \ref{theorem2}, let $U$ be as in (ii) of Lemma \ref{lemmaradial}. Then $\LS = \LS\cap U \subset M$ for some real-analytic $k$-dimensional manifold $M$.

By Theorem \ref{theoreminterior} (or Theorem \ref{theoreminteriorprime} if $k = \dem(\LS)$), $\LS$ has nonempty interior relative to $M$. Let $I$ be the interior of $\LS$ relative to $M$. Then $I$ is invariant under $(u_a)_{a\in E}$ and is therefore dense in $\LS$. So $I$ is a relatively open relatively compact subset of $M$ whose closure is equal to $\LS$.
\item To prove Theorem \ref{theorem3}, we divide into cases. If $k = 1$, then the conclusion follows from Theorem \ref{theoremeremenkovanstrien}. If $k = 2$, then by Theorem \ref{theoreminterior}, $\LS$ has nonempty interior.  But then the boundary of $\LS$ is a closed $T$-invariant set which does not contain $\LS$, hence it is empty. Thus $\LS = \what\C$.
\qedhere\end{itemize}
\end{proof}

The proof of Theorem \ref{theoremrigidity} will occupy the remainder of this paper. We divide it into three cases: $\dimH < \infty$; $\dimH = \infty$ and $k = 1$; $\dimH = \infty$ and $k > 1$. We refer to these cases as Theorem \ref{theoremrigidity}($\dimH < \infty$), Theorem \ref{theoremrigidity}($\dimH = \infty$, $k = 1$), and Theorem \ref{theoremrigidity}($\dimH < \infty$, $k > 1$), respectively, and we prove them in Sections \ref{sectionF}, \ref{sectionk1}, and \ref{sectionI}, respectively. Section \ref{sectionpseudo} contains preliminaries needed for the proof of Theorem \ref{theoremrigidity}($\dimH = \infty$, $k > 1$), and Section \ref{sectionlemmaTPconstant} contains the proof of a technical lemma needed in the proof of Theorem \ref{theoremrigidity}($\dimH = \infty$, $k > 1$).

\draftnewpage
\section{Proof of Meta-rigidity Theorem \ref{theoremrigidity}($\dimH < \infty$)}
\label{sectionF}
The basic structure of the proof may be summarized as follows:
\begin{itemize}
\item[Step 1:] Show that we can without loss of generality assume $\mu = \scrH^k\given_\SK$. In particular, this implies that $\scrH^k(\SK) < \infty$.
\item[Step 2:] Use the rigidity assumption $k = \TD(\SK)$ or $k = \dem(\SK)$ to prove that $\mu$-a.e. point has an \emph{approximate tangent $k$-plane} (Definition \ref{definitiontangentplane} below), using results from geometric measure theory. This is in some sense the key step, and generalizing it appropriately to infinite dimensions will be one of the main themes in the proof of Theorem \ref{theoremrigidity}($d = \infty$).
\item[Step 3:] Choose a radial limit point $\pp\in\Rad_k(\mu)$ which has an approximate tangent $k$-plane. Then use a ``zooming argument'' to finish the proof.
\end{itemize}
We remark that Step 1 does not depend on finite-dimensionality and in fact the result of this step will be used in the proof of Theorem \ref{theoremrigidity}($\dimH = \infty$).

In the sequel we will often refer to the following well-known fact:

\begin{fact}
\label{factlipschitz}
Let $X$ and $Y$ be metric spaces, and let $f:X\to Y$ be a $K$-Lipschitz function, i.e. a function for which
\[
\dist(f(x),f(y)) \leq K \dist(x,y) \all x,y\in X.
\]
Then for all $A\subset X$,
\[
\int_Y \#(f^{-1}(x)\cap A) \;\dee\scrH^k(x) \leq K^k \scrH^k(A).
\]
In particular, if $f$ is $1$-Lipschitz then $\scrH^k(f(A)) \leq \scrH^k(A)$ for all $A\subset X$.
\end{fact}

{\bf Step 1: Without loss of generality, $\mu = \scrH^k\given_\SK$.}
In this step, we will prove that:
\begin{itemize}
\item[(i)] $\nu := \scrH^k\given_\SK$ is finite and absolutely continuous to $\mu$;
\item[(ii)] $\Rad_k(\mu) \subset \Rad_k(\nu)$;
\item[(iii)] $\Supp(\nu) = \SK = \Supp(\mu)$.
\end{itemize}
From these three facts, it follows that if we can prove Theorem \ref{theoremrigidity} in the case $\mu = \nu$, then it holds in the general case as well. We remark that in fact, for the applications of Theorem \ref{theoremrigidity} used in \6\ref{sectionstatements}, it is possible to prove that $\mu = \alpha\nu$ for some $\alpha > 0$ using an ergodicity argument. But for a cleaner formulation and for additional generality, we did not include an ergodicity-type hypothesis in Theorem \ref{theoremrigidity}.

The main tool for this step is the Rogers--Taylor density theorem, which in some circumstances allows one to estimate the Hausdorff measure of a set from the \emph{upper density function}
\begin{equations}
\overline D_\mu^\delta(x) &= \limsup_{r\searrow 0} \frac{1}{r^\delta}\mu(B(x,r))
\end{equations}
of a finite measure $\mu$. Rogers and Taylor originally proved their result for finite-dimensional Euclidean spaces \cite{RogersTaylor}, but it has subsequently been extended to arbitrary metric spaces.
\begin{theorem}[Rogers--Taylor density theorem, {\cite[Theorem 8.2]{MSU}}, {\cite[\62.10.19]{Federer}}]
\label{theoremRTT}
Let $X$ be a metric space, let $\mu$ be a measure on $X$, and fix $\delta > 0$. Then for all $A\subset X$,
\[
\mu(A) \inf_{x\in A}\frac{1}{\overline D_\mu^\delta(x)} \lesssim_\times \scrH^\delta(A) \lesssim_\times \mu(X) \sup_{x\in A}\frac{1}{\overline D_\mu^\delta(x)}\cdot
\]
In particular, if $\overline D_\mu^\delta$ is bounded from below on a set $A\subset X$, then $\scrH^\delta\given_A \leq C\mu$ for some constant $C > 0$. If $\overline D_\mu^\delta$ is strictly positive on $A$, then $\scrH^\delta\given_A$ is $\sigma$-finite and absolutely continuous to $\mu$.
\end{theorem}
In our setting we take $X = \HH$, $\delta = k$, and
\[
A = \Rad_k(C,\mu) := \{\pp\in\Rad_k(\mu) : \text{the implied constants in Definition \ref{definitionradial} are all $< C$}\},
\]
where $C > 0$ is fixed. Fix $\pp\in A$, and let $(g_n)_1^\infty$ be as in Definition \ref{definitionradial}. Then for all $n$,
\[
\frac{1}{(C_\pp \|g_n'\|)^k} \mu\big(B(\pp,C_\pp \|g_n'\|)\big) \geq \frac{1}{(C_\pp \|g_n'\|)^k} \mu(g_n(U)) \gtrsim_{\times,C} \frac{1}{\|g_n'\|^k} \|g_n'\|^k \mu(U) = \mu(U) > 0.
\]
So $\overline D_\mu^k(\pp) > 0$, and thus by Theorem \ref{theoremRTT},
\begin{equation}
\label{RTT1}
\text{$\nu\given_{\Rad_k(C,\mu)}$ is finite and absolutely continuous to $\mu$.}
\end{equation}
%
%
Let $\scrG$ be the collection of maps $g:U\to\HH$ of class $\AA$ such that for all $A\subset U$, \eqref{muconformal} holds.
\begin{claim}
\label{claimergodic}
If $\mu(A) > 0$ for some $A\subset U$, then $\nu(\SK\butnot \scrG(A)) = 0$, where
\[
\scrG(A) = \bigcup_{g\in\scrG} g(A).
\]
\end{claim}
\begin{proof} 
Let $\gamma = \mu\given_{\scrG(A)}$. Fix $\pp\in\Rad_k(\mu)$, and let $(g_n)_1^\infty$ be as in Definition \ref{definitionradial}. Then for all $n$,
\[
\frac{1}{(C_\pp\|g_n'\|)^k} \gamma\big(B(\pp,C_\pp\|g_n'\|)\big) \geq \frac{1}{(C_\pp\|g_n'\|)^k} \mu(g_n(A)) \gtrsim_{\times,\pp} \mu(A) > 0.
\]
So $\overline D_\mu^k(\pp) > 0$, and thus by Theorem \ref{theoremRTT}, $\nu = \scrH^k\given_{\Rad_k(\mu)}$ is $\sigma$-finite and absolutely continuous to $\gamma$. So since $\gamma(\SK\butnot \scrG(A)) = 0$, we have $\nu(\SK\butnot \scrG(A)) = 0$.
\end{proof}

Let $A = U$ in Claim \ref{claimergodic}. Since $\mu(U) > 0$ by hypothesis, we get $\nu(\SK\butnot \scrG(U)) = 0$. On the other hand, by Theorem \ref{theoremsz} (or Theorem \ref{theoremszprime} if $k = \dem(\LS)$), $\nu(\SK) > 0$, so $\nu(\scrG(U)) > 0$, and so by Fact \ref{factlipschitz} $\nu(U) > 0$. We now demonstrate (i)-(iii):
\begin{itemize}
\item[(i)] Since by assumption $\nu(\SK\butnot\Rad_k(\mu)) = 0$, we have $\nu(\Rad_k(\mu)\cap U) > 0$, so for some $C > 0$, we have $\nu(\Rad_k(C,\mu)\cap U) > 0$. In particular, by \eqref{RTT1}, $\mu(\Rad_k(C,\mu)\cap U) > 0$. Letting $A = \Rad_k(C,\mu)\cap U$ in Claim \ref{claimergodic}, we get $\nu(\SK\butnot \scrG(\Rad_k(C,\mu)\cap U)) = 0$. But since the elements of $\scrG$ are conformal, it follows that $\Rad_k(C,\mu)$ is $\scrG$-invariant in the sense that $\scrG(\Rad_k(C,\mu)\cap U) \subset \Rad_k(C,\mu)$. So $\nu(\SK\butnot \Rad_k(C,\mu)) = 0$. Using \eqref{RTT1} again, we see that $\nu$ is finite and absolutely continuous to $\mu$.
\item[(ii)] Suppose that $\pp\in\Rad_k(\mu)$, and let $(g_n)_1^\infty$ be as in Definition \ref{definitionradial}. Fix $n$ and let $A = g_n^{-1}(\HH\butnot\SK)\subset U$. Since $\mu(g_n(A)) = 0$, \eqref{muconformal} shows that $\mu(A) = 0$, i.e. $A\cap\SK = \emptyset$. It follows that $g_n(\SK\cap U)\subset\SK$, and then Fact \ref{factlipschitz} and (a2) yield that \eqref{muconformal} holds with $\mu$ replaced by $\nu$. Since $n$ was arbitrary, we get $\pp\in\Rad_k(\nu)$.
\item[(iii)] Since $\nu(U) > 0$, it is obvious that $\Rad_k(\nu)\subset\Supp(\nu)$, so this follows from (ii) together with the hypothesis that $\Rad_k(\mu)$ is dense in $\SK$.
\end{itemize}

%
%
%
%

{\bf Step 2: Existence of approximate tangent planes.}
Let us recall the definition of an approximate tangent plane:

\begin{definition}[Cf. {\cite[Definition 15.17]{Mattila}}]
\label{definitiontangentplane}
Let $\SK$ be a subset of $\HH$, and let $\pp\in \SK$ satisfy
\[
\limsup_{r\searrow 0}\frac{1}{r^k}\scrH^k\big(\SK\cap B(\pp,r)\big) > 0.
\]
A $k$-dimensional linear subspace $L_0\leq\HH$ is called an \emph{approximate tangent $k$-plane for $\SK$ at $\pp$} if for all $\epsilon > 0$,
\begin{equation}
\label{tangentplane}
\lim_{r\searrow 0} \frac{1}{r^k}\scrH^k\big(\SK\cap B(\pp,r)\butnot (\pp + \NN_{\mathrm{proj}}(L_0,\epsilon))\big) = 0.
\end{equation}
Here $\NN_{\mathrm{proj}}(L_0,\epsilon)$ denotes the \emph{projective $\epsilon$-thickening}
\[
\NN_{\mathrm{proj}}(L_0,\epsilon) = \{\xx\in\HH : \dist(\xx,L_0) \leq \epsilon\|\xx\|\}.
\]
\end{definition}


The existence of approximate tangent planes for a large measure set of points on suitably non-fractal\Footnote{Here we use Mandelbrot's terminology \cite[p.38]{Mandelbrot}, in which a set $\SK$ is called a fractal if and only if $\HD(\SK) > \TD(\SK)$.} subsets of $\HH$ is a combination of some well-known geometric measure theory results:


\begin{lemma}[Non-fractality implies approximate tangent planes]
\label{lemmafederermattila}
Suppose that $\dimH < \infty$, and let $\SK\subset\HH$ be a set such that $\scrH^k(\SK) < \infty$, where either $k = \TD(\SK)$ or $k = \dem(\SK)$. Then the set of points in $\SK$ which have an approximate tangent $k$-plane has positive $\scrH^k$ measure.
\end{lemma}
\NPC{Proof}
\begin{proof} 
The hypothesis implies that $\scrH^k(\pi_V(\SK)) > 0$ for a positive measure set of $k$-dimensional linear subspaces $V\leq\HH$, where $\pi_V$ denotes orthogonal projection onto $V$: if $k = \TD(\SK)$, then this follows from \cite[\69]{Federer2}, and if $k = \dem(\SK)$, then this follows from the proof of Theorem \ref{theoremszprime}. Thus by the Besicovitch--Federer projection theorem (Theorem \ref{theorembesicovitchfederer} below), $\SK$ is not purely $k$-unrectifiable (cf. Definition \ref{definitionunrectifiable} below). The conclusion now follows from \cite[Corollary 15.20]{Mattila}.
\end{proof}
\NPC{Proof}

Since in our context we get either $k = \TD(\SK)$ or $k = \dem(\SK)$ by hypothesis, and the $\sigma$-finiteness of $\scrH^k\given_\SK$ was proven in Step 1, we get that the set of points in $\SK$ which have an approximate tangent $k$-plane has positive $\scrH^k$ measure. In particular, there exists a point $\pp\in\Rad_k(\mu)$ which has an approximate tangent $k$-plane.

\

{\bf Step 3: Zooming argument.}
Let $\pp$ be as above. Let $L_0\leq\HH$ be an approximate tangent $k$-plane for $\SK$ at $\pp$, and let $(g_n)_1^\infty$ be as in Definition \ref{definitionradial}. Then for all $\epsilon > 0$, by \eqref{tangentplane} we have
\[
\frac{1}{\|g_n'\|^k}\mu\Big(B(\pp,C_\pp\|g_n'\|)\butnot (\pp + \NN_{\mathrm{proj}}(L_0,\epsilon))\Big) \tendsto n 0.
\]
Now by (a3),
\[
\mu\left(U\butnot g_n^{-1}\big(\pp + \NN_{\mathrm{proj}}(L_0,\epsilon)\big)\right) \lesssim_\times \frac{1}{\|g_n'\|^k} \mu\big(g_n(U)\butnot (\pp + \NN_{\mathrm{proj}}(L_0,\epsilon))\big) \tendsto n 0.
\]
After extracting a subsequence along which this convergence is geometrically fast, we get $\mu(U\butnot S_\epsilon) = 0$ and thus $\SK\cap U\subset S_\epsilon$, where
\[
S_\epsilon = \bigcup_{N\in\N} \bigcap_{n\geq N} g_n^{-1}\big(\pp + \NN_{\mathrm{proj}}(L_0,\epsilon)\big).
\]
Now write $T_n(\xx) = \pp + \|g_n'\|\xx$ and $h_n = T_n^{-1}\circ g_n\in\AA$; then
\[
S_\epsilon = \bigcup_{N\in\N} \bigcap_{n\geq N} h_n^{-1}\big(\NN_{\mathrm{proj}}(L_0,\epsilon)\big).
\]
But by (a1), $|h_n'(\xx)| \asymp_\times 1$ for all $\xx\in U$, and by \eqref{Cpdef}, $h_n(U)\subset B(\0,C_\pp)$ for all $n$. In particular, by the Arzela--Ascoli theorem, $(h_n)_1^\infty$ is a normal family. After extracting a subsequence, let $h$ be the limit. Then
\begin{align*}
S_\epsilon &\subset h^{-1}\big(\NN_{\mathrm{proj}}(L_0,\epsilon)\big),\\
\SK\cap U \subset \bigcap_{\epsilon > 0} S_\epsilon &\subset h^{-1}\left(\bigcap_{\epsilon > 0}\NN_{\mathrm{proj}}(L_0,\epsilon)\right) = h^{-1}(L_0).
\end{align*}
Since the class $\AA$ is closed under normal limits, we have $h\in\AA$, which completes the proof.

\begin{remark}
\label{remarkpoincareconjecture}
By slightly modifying the above argument, one can show: If $\mu$ is a finite measure on $\HH$, $\SK = \Supp(\mu)$, and $U\subset\HH$ is an open set, then either $\SK\cap U$ is contained in a $k$-dimensional real-analytic manifold, or else no point of $\Rad_\delta(\mu)$ can have a tangent $k$-plane for $\SK$. Here a $k$-dimensional linear subspace $L_0\leq\HH$ is called a \emph{tangent $k$-plane} for $\SK$ at $\pp$ if for all $\epsilon > 0$ there exists $r > 0$ such that $\SK\cap B(\pp,r)\subset \pp + \NN_\Proj(L_0,\epsilon)$. In particular, this proves Poincar\'e's conjecture \cite[Section VIII]{Poincare1} mentioned in the introduction, corresponding to the special case where $\SK$ is the limit set of a certain Kleinian group.
\end{remark}

\draftnewpage
\section{Proof of Meta-rigidity Theorem \ref{theoremrigidity}($\dimH = \infty$, $k = 1$)}
\label{sectionk1}
An analysis of the argument in Section \ref{sectionF} shows that the finite-dimensionality assumption was only used in two places:
\begin{itemize}
\item[1.] in the use of Lemma \ref{lemmafederermattila} (the ``key step'') in Step 2, and
\item[2.] in the use of the Arzela--Ascoli theorem in Step 3.
\end{itemize}

It turns out that when $k = 1$, we can find ``patches'' for these arguments which work in infinite dimensions, keeping the overall proof structure of Theorem \ref{theoremrigidity} the same. But when $k > 1$, a different proof structure will be needed. However, a substantial portion of this section (namely Theorem \ref{theoremlebesguedifferentiation} and Corollary \ref{corollarylebesguedifferentiation}) will also be used in the proof of Theorem \ref{theoremrigidity}($\dimH = \infty$, $k > 1$).

We begin with a lemma which is the analogue of the first two-thirds of the proof of Lemma \ref{lemmafederermattila}:
\begin{lemma}
\label{lemmapatch1}
Let $\SK\subset\HH$ be a compact set such that $\scrH^1(\SK) < \infty$ and $\TD(\SK) = 1$. Then $\SK$ contains a rectifiable curve.
\end{lemma}
\begin{proof}
Since $\SK$ is a compact set with $\TD(\SK) = 1$, there exists a connected component $C\subset\SK$ with at least two points \cite[Theorem 1.4.5]{Engelking}. By \cite[Lemmas 3.5 and 3.7]{Schul}, $C$ is equal to a rectifiable curve.
\end{proof}
\begin{remark}
\label{remarkpatch1}
This lemma is the only place in this section where the assumption $k = 1$ is used in a crucial manner. (This fact will be important in the proof of Theorem \ref{theoremrigidity}($\dimH = \infty$, $k > 1$).) It is not clear whether the analogue of Lemma \ref{lemmapatch1} for $k > 1$ holds or not.
\end{remark}

Next, we would like to say that every rectifiable curve contains approximate tangent planes, but the problem is that the usual version of the Lebesgue differentiation theorem does not hold in infinite dimensions. Instead, we need a variant of this theorem, which we state and prove below. When we apply this theorem to the problem of finding approximate tangent planes to rectifiable curves, it turns out that we need to replace the limit \eqref{tangentplane} with the limit taken along a certain sequence $r_n\searrow 0$. The points that we get do not necessarily have ``approximate tangent planes'' in the sense of Definition \ref{definitiontangentplane}, but they are good enough to make the argument of Step 3 work.

\begin{theorem}[Lebesgue differentiation theorem for general metric spaces]
\label{theoremlebesguedifferentiation}
Let $X$ and $Y$ be separable metric spaces, let $\nu$ be a locally finite measure on $X$, and let $f:X\to Y$ be a bounded function. For $\nu$-a.e. $p\in X$, if $(B_n)_1^\infty$ is a sequence of sets such that
\begin{itemize}
\item[(I)] $\diam(B_n)\to 0$,
\item[(II)] $p\in B_n$ for all $n$, and
\item[(III)] $\nu(4B_n)\asymp_\times\nu(B_n)$ for all $n$, where
\[
4B := \thickvar B{1.5\diam(B)},
\]
\end{itemize}
then
\begin{equation}
\label{lebesguedifferentiation}
\lim_{n\to\infty} \frac{1}{\nu(B_n)} \int_{B_n} \dist_Y\big(f(p),f(x)\big) \;\dee\nu(x) = 0.
\end{equation}
\end{theorem}
\NPC{Proof}
\begin{proof}
This theorem follows from \cite[\6\62.8-2.9]{Federer} but we write out the proof for clarity. For each $\delta,C > 0$, let $\AA_\delta$ be a countable partition of $Y$ such that $\diam(A)\leq\delta$ for all $A\in\AA_\delta$, and let
\begin{align*}
\CC_{C,\delta} &= \left\{B\cap f^{-1}(A) : \begin{array}{l}B\subset X, \; \nu(4B)\leq C\nu(B), \; A\in\AA_\delta,\\ \nu(B\butnot f^{-1}(A)) \geq \delta\nu(B)\end{array}\right\}\\
X_{C,\delta} &= \Big\{x\in X : \inf\{\diam(E) : x\in E\in\CC_{C,\delta}\} = 0\Big\}.
\end{align*}
\begin{claim}
Fix $p\in X$ and a sequence $(B_n)_1^\infty$ such that \text{(I)-(III)} hold, and let $C$ be the implied constant of \text{(III)}. If $p\notin X_{C,\delta}$ for all $\delta > 0$, then \eqref{lebesguedifferentiation} holds.
\end{claim}
\begin{subproof}
Fix $\delta > 0$. Let $A\in\AA_\delta$ be the element which contains $f(p)$. Since $p\notin X_{C,\delta}$, for all sufficiently large $n$ we have $B_n\cap f^{-1}(A)\notin \CC_{C,\delta}$, so
\[
\nu(B_n\butnot f^{-1}(A)) < \delta\nu(B_n).
\]
In particular
\begin{align*}
\int_{B_n} \dist_Y\big(f(p),f(x)\big) \;\dee\nu(x) &\leq \int_{B_n\cap f^{-1}(A)} \dist_Y\big(f(p),f(x)\big) \;\dee\nu(x) + \int_{B_n\butnot f^{-1}(A)} \dist_Y\big(f(p),f(x)\big) \;\dee\nu(x)\\
&\leq \delta\nu(B_n) + \diam_Y(f(X)) \nu(B_n\butnot f^{-1}(A)) \leq \delta\big(1 + \diam_Y(f(X))\big) \nu(B_n).
\end{align*}
Taking the limit as $n\to\infty$ gives
\[
\limsup_{n\to\infty} \frac{1}{\nu(B_n)} \int_{B_n} \dist_Y\big(f(p),f(x)\big) \;\dee\nu(x) \leq \delta\big(1 + \diam_Y(f(X))\big),
\]
and letting $\delta\to 0$ finishes the proof.
\end{subproof}
\noindent So to complete the proof, it suffices to show that $\nu(X_{C,\delta}) = 0$ for all $\delta,C > 0$.

By contradiction, suppose that $\nu(X_{C,\delta}) > 0$ for some $\delta,C > 0$. Then there exists $A\in\AA_\delta$ such that $\nu(X_{C,\delta}\cap f^{-1}(A)) > 0$, and by the inner regularity of $\nu$, there exists a compact set $K\subset X_{C,\delta}\cap f^{-1}(A)$ such that $\nu(K) > 0$. Let $U\supset K$ be an arbitrary open set. Then by the definition of $X_{C,\delta}$,
\[
\CC_{C,\delta,A}(U) := \{B\subset U : \nu(4B)\leq C\nu(B), \;\; B\cap f^{-1}(A)\in\CC_{C,\delta}\}
\]
is a cover of $K$. So by the $4r$-covering lemma (e.g. \cite[Theorem 8.1]{MSU}), there exists a disjoint subcollection $\DD\subset\CC_{C,\delta,A}(U)$ such that $\{4B:B\in\DD\}$ is a cover of $K$. But then
\[
\nu(K) \leq \sum_{B\in\DD} \nu(4B) \asymp_{\times,C} \sum_{B\in\DD} \nu(B) \lesssim_{\times,\delta} \sum_{B\in\DD} \nu(B\butnot f^{-1}(A)) \leq \nu(U\butnot f^{-1}(A)) \leq \nu(U\butnot K).
\]
Since $U$ was arbitrary and $\nu(K) > 0$, this contradicts the outer regularity of $\nu$.
\end{proof}

Applying Theorem \ref{theoremlebesguedifferentiation} to our circumstance, we get:

\begin{corollary}
\label{corollarylebesguedifferentiation}
Let $X = \SK$ and $\nu = \mu = \scrH^k\given_\SK$ be as in Theorem \ref{theoremrigidity}, and let $Y$ and $f$ be as in Theorem \ref{theoremlebesguedifferentiation}. For $\mu$-a.e. $\pp\in\Rad_k(\mu)$, if $(g_n)_1^\infty$ is as in Definition \ref{definitionradial}, then
\begin{equation}
\label{LD2}
\lim_{n\to\infty} \frac{1}{\|g_n'\|^k} \int_{g_n(U)} \dist_Y\big(f(\pp),f(\xx)\big) \;\dee\mu(\xx) = 0.
\end{equation}
\end{corollary}
\begin{proof}
Let $\pp\in\Rad_k(\mu)$ be a point for which
\begin{equation}
\label{Dmu}
\overline D_\mu^k(\pp) = \limsup_{r\searrow 0} \frac{1}{r^k}\mu(B(\pp,r)) < \infty;
\end{equation}
by Theorem \ref{theoremRTT}, \eqref{Dmu} holds for $\mu$-a.e. $\pp\in\Rad_k(\mu)$. Let $(g_n)_1^\infty$ be as in Definition \ref{definitionradial}, and let $B_n = B(\pp,C_\pp\|g_n'\|)$. Then $\mu(4B_n) \lesssim_{\times,\pp} \|g_n'\|^k$ by \eqref{Dmu}, but $\mu(B_n)\geq \mu(g_n(U)) \gtrsim_{\times,\pp} \|g_n'\|^k$ by \eqref{Cpdef} and \eqref{muconformal}. By (b) of Definition \ref{definitionradial}, $\diam(B_n) \leq 2C_\pp \|g_n'\| \to 0$, and by construction $\pp\in B_n$ for all $n$. So by Theorem \ref{theoremlebesguedifferentiation}, \eqref{lebesguedifferentiation} holds for the sequence $(B_n)_1^\infty$ (assuming $\pp$ is not in the $\mu$-nullset where Theorem \ref{theoremlebesguedifferentiation} fails). Since $\mu(B_n) \asymp_{\times,\pp} \|g_n'\|^k$, \eqref{muconformal} finishes the proof.
\end{proof}

We can now prove an analogue of Lemma \ref{lemmafederermattila}:
\begin{lemma}
\label{lemmapatch2}
Let $\SK$ and $\mu = \scrH^k\given_\SK$ be as in Theorem \ref{theoremrigidity}, with $k = 1$. Then there exists a positive $\mu$-measure set of points $\pp\in\Rad_k(\mu)$ such that if $(g_n)_1^\infty$ are as in Definition \ref{definitionradial}, then there exists a $k$-dimensional linear subspace $L_0\leq\HH$ such that for all $\epsilon > 0$,
\begin{equation}
\label{patch2}
\lim_{n\to\infty} \frac1{\|g_n'\|^k} \mu\Big(g_n(U)\butnot (\pp + \NN_{\mathrm{proj}}(L_0,\epsilon))\Big) = 0.
\end{equation}
\end{lemma}
\begin{proof}
By Lemma \ref{lemmapatch1}, $\SK$ contains a rectifiable curve $C$. We recall that this means that there exists a Lipschitz surjection $\gamma:[0,1]\to C$. By the infinite-dimensional version of Rademacher's theorem \cite[Theorem 3.5]{AmbrosioKirchheim}, $\gamma$ is differentiable Lebesgue-a.e. Let $N$ be the set of points where $\gamma$ is either not differentiable or has a zero derivative. By the metric change-of-variables formula \cite[Theorem 5.1]{AmbrosioKirchheim}, $\scrH^1(\gamma(N)) = 0$.

Let $f:C\to[0,1]$ be a measurable partial inverse of $\gamma$, so that $\gamma\circ f(\xx) = \xx$ for all $\xx\in C$. Extend $f$ to $\SK$ by setting $f \equiv 2$ on $\SK\butnot C$. Let $\pp\in C\cap\Rad_k(\mu)\butnot\gamma(N)$ be a point satisfying the conclusion of Corollary \ref{corollarylebesguedifferentiation}; we claim that $\pp$ satisfies the conclusion of the lemma. Let $(g_n)_1^\infty$ be as in Definition \ref{definitionradial}. Since $\pp\in C\butnot\gamma(N)$, $\gamma$ is differentiable at $t := f(\pp)$ and $\gamma'(t)\neq\0$. Let $L_0 = \R\gamma'(t)$.

Fix $\epsilon > 0$, and let $\delta > 0$ be small enough so that for all $s\in [t - \delta,t + \delta]$,
\[
\|\gamma(s) - \gamma(t) - \gamma'(t)(s - t)\| \leq \sigma \|\gamma'(t)\|\cdot|s - t|,
\]
where $\epsilon = \frac{\sigma}{1 - \sigma}$. Then
\[
\gamma([t - \delta,t + \delta]) \subset \pp + \NN_{\mathrm{proj}}(L_0,\epsilon).
\]
But by \eqref{LD2},
\[
\lim_{n\to\infty} \frac{1}{\|g_n'\|^k} \mu\left(\big\{\xx\in g_n(U) : f(\xx) \notin [t - \delta,t + \delta]\big\}\right) = 0,
\]
which implies \eqref{patch2}.
\end{proof}

Armed with this lemma, no changes are needed for Step 3 of Section \ref{sectionF} until the last paragraph. Instead of using the Arzela-Ascoli theorem, we use Liouville's theorem to extend $h_n$ to a M\"obius tansformation acting on all of $\what\HH$, and then we use the fact that $|h_n'| \asymp_\times 1$ on $U$ to deduce that
\[
h_n^{-1}\big(\pp + \NN_{\mathrm{proj}}(L_0,\epsilon)\big)\cap U \subset \NN(h_n^{-1}(L_0),C\epsilon)
\]
for some large constant $C > 0$. So for every $n\in\N$, there exists a generalized $1$-sphere $S_n$ such that $\SK\cap U\subset \NN(S_n,2^{-n})$. From this it is not hard to see that $\SK\cap U$ is contained in a generalized $1$-sphere, since if the sequence $(S_n)_1^\infty$ does not converge, then $\SK\cap U$ will be contained in a point. (A more detailed version of this argument is given in Step 5 of the proof of Theorem \ref{theoremrigidity}($\dimH = \infty$, $k > 1$) below.)

\draftnewpage
\section{Pseudorectifiability}
\label{sectionpseudo}
We now begin the preliminaries to the proof of infinite-dimensional rigidity, i.e. Theorem \ref{theoremrigidity}($\dimH = \infty$, $k > 1$).

\subsection{Definition and basic properties of pseudorectifiability}
Our main tool in the proof of Theorem \ref{theoremrigidity}($\dimH = \infty$, $k > 1$) will be an analogue of Lemma \ref{lemmafederermattila} based on the notion of \emph{pseudorectifiability}, a generalization of rectifiability. Thus, we begin by recalling the definition of rectifiability, and then introducing the definition of pseudorectifiability:

\begin{definition}
\label{definitionrectifiable}
A set $R\subset\HH$ is called \emph{($k$-)rectifiable} if there exists a countable family of Lipschitz maps $\ff_i:[0,1]^k\to\HH$ such that
\[
\scrH^k\left(R\butnot \bigcup_{i = 1}^\infty \ff_i([0,1]^k)\right) = 0.
\]
\end{definition}
Obviously, any $k$-dimensional $\CC^1$ manifold is rectifiable, and the countable union of rectifiable sets is rectifiable. For more on rectifiable sets, see \cite{Mattila}.

Before defining pseudorectifiability, we need to introduce some notations:

\begin{notation}
\label{notationgrass}
For each $\ell\in\N$, let $\Grass_\ell(\HH)$ denote the set of $\ell$-dimensional linear subspaces of $\HH$, and let
\begin{align*}
\Grass(\HH) &= \bigcup_{\ell = 0}^\infty \Grass_\ell(\HH) = \{V\leq\HH : \text{$V$ is a finite-dimensional vector space}\}.
\end{align*}
For each $V\in\Grass(\HH)$, we let $\pi_V$ denote orthogonal projection onto $V$. Moreover, if $L:V\to\HH$ is any linear map, then $\det(L)$ denotes the metric determinant of $L$, i.e. $\det(L) = |\det_{\mathrm{alg}}(\phi\circ L)|$, where $\phi:L(V)\to V$ is any isometry and $\det_{\text{alg}}$ is the standard (algebraic) determinant.
\end{notation}

\begin{definition}
\label{definitionpseudorectifiable}
A set $R\subset\HH$ such that $\scrH^k\given_R$ is $\sigma$-finite will be called \emph{($k$-)pseudorectifiable} if there exist a function $T_R:R\to \Grass_k(\HH)$ and a $\sigma$-finite measure $\mu_R\asymp\scrH^k\given_R$ such that for every $V\in\Grass_k(\HH)$ and $A\subset R$,
\begin{equation}
\label{pseudorectifiable}
\int \#(\pi_V^{-1}(\yy)\cap A) \; \dee\scrH^k(\yy) = \int_A \det\big(\pi_V\given T_R(\xx)\big) \; \dee\mu_R(\xx).
\end{equation}
The function $T_R$ will be called the \emph{tangent plane function} of $R$ with respect to $\mu_R$.
\end{definition}

\begin{observation}
\label{observationpseudounion}
It follows directly from the definition that any subset of a pseudorectifiable set is pseudorectifiable, and the countable disjoint union of pseudorectifiable sets is pseudorectifiable. Thus, the countable union of pseudorectifiable sets is pseudorectifiable.
\end{observation}

Essentially, Definition \ref{definitionpseudorectifiable} says that a set is said to be pseudorectifiable if the ``change-of-variables theorem'' holds for projections onto $k$-dimensional subspaces, if $T_R(\xx)$ is interpreted to be the ``tangent plane of $R$ at $\xx$'', and $\mu_R$ is interpreted to be an ``idealized''\Footnote{In finite dimensions, $\mu_R$ is actually equal to $\scrH^k\given_R$ (Propositions \ref{propositionrectifiable} and \ref{propositionfindimconverse}), but in infinite dimensions only the inequality $\mu_R\leq\scrH^k\given_R$ holds (Lemma \ref{lemmamuleqH}); a counterexample to equality is given by Example \ref{examplemuR1}(ii).} version of $k$-dimensional Hausdorff measure on $R$. In fact, this interpretation is what allows us to prove that any rectifiable set is pseudorectifiable, as we now show:

\begin{proposition}
\label{propositionrectifiable}
Any rectifiable set $R\subset\HH$ is pseudorectifiable and satisfies $\mu_R = \scrH^k\given_R$. Moreover, if $\ff:\HH\to\R^k$ is any $\CC^1$ map then the the formula
\begin{equation}
\label{rectifiable}
\int \#(\ff^{-1}(\yy)\cap A) \; \dee\scrH^k(\yy) = \int_A \det\big(\ff'(\xx)\given T_R(\xx)\big) \; \dee\mu_R(\xx)
\end{equation}
holds for all $A\subset R$.
\end{proposition}
\NPC{Proof}
\begin{proof}
By Observation \ref{observationpseudounion}, it suffices to consider the case where $R = \gg([0,1]^k)$ for some Lipschitz map $\gg:[0,1]^k\to\HH$. By \cite[Theorem 3.5]{AmbrosioKirchheim}, the derivative $\gg'$ exists $\scrH^k$-a.e., and by \cite[Theorem 5.1]{AmbrosioKirchheim}, the change-of-variables theorem holds, both for $\gg$ and for the maps $\ff\circ\gg$ ($\ff:\HH\to\R^k$). Let $\hh:R\to [0,1]^k$ be a measurable partial inverse of $\gg$, so that $\gg\circ\hh(\yy) = \yy$ for all $\yy\in R$. We then define $T_R:R\to\Grass_k(\HH)$ as follows:
\begin{align*}
T_R(\yy) &= \gg'(\hh(\yy))[\R^k] \;\; (\yy\in R).
\end{align*}
Then for all $\ff:\HH\to\R^k$ and $A\subset R$,
\begin{align*}
&\int \#(\ff^{-1}(\zz)\cap A) \;\dee\scrH^k(\zz)
= \int \#\big((\ff\circ \gg)^{-1}(\zz)\cap \hh(A)\big) \;\dee\scrH^k(\zz)\\
&= \int_{\hh(A)} \det\big((\ff\circ \gg)'(\xx)\given \R^k\big) \; \dee \scrH^k(\xx)
= \int_{\hh(A)} \det\big(\ff'(\gg(\xx))\given \gg'(\xx)[\R^k]\big) \det\big(\gg'(\xx)\given \R^k\big) \; \dee \scrH^k(\xx)\\
&= \int \sum_{\xx\in \gg^{-1}(\yy)} \det\big(\ff'(\gg(\xx))\given \gg'(\xx)[\R^k]\big) \one_{\hh(A)}(\xx) \;\dee\scrH^k(\yy) = \int_A \det\big(\ff'(\yy)\given T_R(\yy)\big) \;\dee\scrH^k(\yy),
\end{align*}
i.e. \eqref{rectifiable} holds with $\mu_R = \scrH^k\given_R$. Since \eqref{pseudorectifiable} is just the special case of \eqref{rectifiable} which holds when $\ff$ is linear, this completes the proof.
\end{proof}

In finite dimensions, the converse to Proposition \ref{propositionrectifiable} holds (Proposition \ref{propositionfindimconverse}), but in infinite dimensions there are pseudorectifiable sets which are not rectifiable (Example \ref{examplemuR1}(ii)). However, we need more background before we can prove these facts.

The next most important fact about pseudorectifiability is the essential uniqueness of the tangent plane function.

\begin{lemma}
\label{lemmaTPunique}
If $R$ is a pseudorectifiable set, then the measure $\mu_R$ and the tangent plane function $T_R$ are unique in the sense that if $(\mu_1,T_1)$ and $(\mu_2,T_2)$ are two different pairs that both satify \eqref{pseudorectifiable}, then $\mu_1 = \mu_2$ and $\mu_1(T_1\neq T_2) = 0$.

More generally, let $\mu_1,\mu_2$ be two $\sigma$-finite measures on a measurable space $X$, and let $T_1,T_2:X\to\Grass_k(\HH)$ be measurable functions. If for all $V\in\Grass_k(\HH)$ and $A\subset X$
\begin{equation}
\label{TPunique}
\int_A \det\big(\pi_V\given T_1(x)\big) \; \dee\mu_1(x) = \int_A \det\big(\pi_V\given T_2(x)\big) \; \dee\mu_2(x),
\end{equation}
then $\mu_1 = \mu_2$ and $\mu_1(T_1\neq T_2) = 0$.
\end{lemma}
\NPC{Proof}
\begin{proof}
If $(\mu_1,T_1)$ and $(\mu_2,T_2)$ are two different pairs that both satify \eqref{pseudorectifiable}, then two applications of \eqref{pseudorectifiable} yield \eqref{TPunique}. So for the remainder of the proof, we just assume that \eqref{TPunique} holds.

Let $\mu = \mu_1 + \mu_2$. Then for all $V\in\Grass_k(\HH)$ and for $\mu$-a.e. $x\in X$,
\begin{equation}
\label{standardanalysis}
\alpha_1(x) \det\big(\pi_V\given T_1(x)\big) = \alpha_2(x) \det\big(\pi_V\given T_2(x)\big),
\end{equation}
where $\alpha_1,\alpha_2:X\to[0,1]$ are Radon--Nikodym derivatives of $\mu_1,\mu_2$ with respect to $\mu$ satisfying $\alpha_1 + \alpha_2 = 1$. Let $\QQ$ be a countable dense subset of $\Grass_k(\HH)$, and fix $x\in X$ such that \eqref{standardanalysis} holds for all $V\in\QQ$. Then by continuity, \eqref{standardanalysis} holds for all $V\in\Grass_k(\HH)$. Plugging in $V = T_1(x)$ yields
\begin{equation}
\label{alphacomparison}
\alpha_1(x) = \alpha_2(x) \det\big(\pi_{T_1(x)}\given T_2(x)\big) \leq \alpha_2(x);
\end{equation}
similarly, plugging in $V = T_2(x)$ yields $\alpha_2(x) \leq \alpha_1(x)$, so $\alpha_1(x) = \alpha_2(x) = 1/2$. Since equality holds in \eqref{alphacomparison}, we must have $T_1(x) = T_2(x)$.
\end{proof}

\begin{remark}
\label{remarkTPunique}
It follows from Lemma \ref{lemmaTPunique} that if $(\mu_1,T_1)$ and $(\mu_2,T_2)$ are any pairs that satisfy \eqref{pseudorectifiable} on a dense subset of $\Grass_k(\HH)$, then $\mu_1 = \mu_2$ and $\mu_1(T_1\neq T_2) = 0$. Indeed, if \eqref{pseudorectifiable} holds on a dense subset, then so does \eqref{TPunique}, but the equation \eqref{TPunique} (unlike \eqref{pseudorectifiable}) is continuous with respect to $V$, so it must hold everywhere.
\end{remark}

The next result shows that when we are proving that a set $R\subset\HH$ is pseudorectifiable, we don't need to check the absolute continuity $\mu_R\lessless \scrH^k\given_R$ which appears in Definition \ref{definitionpseudorectifiable}. However, we note that the reverse direction $\scrH^k\given_R\lessless\mu_R$ does need to be checked, since there are sets which would be pseudorectifiable except for that condition (Example \ref{examplemuR1}(iii)).

\begin{lemma}
\label{lemmamuleqH}
Let $R\subset\HH$ be a set, and suppose that \eqref{pseudorectifiable} holds for some pair $(\mu_R,T_R)$. Then $\mu_R\leq\scrH^k\given_R$.
\end{lemma}
\begin{proof}
For each $V\in\Grass_k(\HH)$, Fact \ref{factlipschitz} implies that for all $A\subset R$,
\[
\int_A \det\big(\pi_V\given T_R(\xx)\big)\dee\mu_R(\xx) = \int \#(\pi_V^{-1}(\yy)\cap A) \; \dee\scrH^k(\yy) \leq \scrH^k(A).
\]
Let $\nu$ be the absolutely continuous component of $\scrH^k\given_R$ with respect to $\mu_R$. Then
\begin{equation}
\label{muleqH}
\frac{\dee\nu}{\dee\mu_R}(\xx) \geq \det\big(\pi_V\given T_R(\xx)\big)
\end{equation}
for $\mu_R$-a.e. $\xx\in R$. Now fix $\xx\in R$ such that \eqref{muleqH} holds for a dense set of $V\in\Grass_k(\HH)$, and hence for all $V$. Letting $V = T_R(\xx)$ gives $\frac{\dee\nu}{\dee\mu_R}(\xx)\geq 1$. Thus $\mu_R\leq\nu\leq\scrH^k\given_R$.
\end{proof}

The next lemma is the main invariance property of pseudorectifiability. It is important for the proof of Theorem \ref{theoremrigidity}($\dimH = \infty$, $k > 1$), since it implies the equivariance of the tangent plane function, allowing us to generalize the ``zooming argument'' of Section \ref{sectionF}.

\begin{lemma}
\label{lemmapseudolinear}
Let $\HH_1,\HH_2$ be separable Hilbert spaces, let $R_1\subset\HH_1$ be a pseudorectifiable set, and let $L:\HH_1\to\HH_2$ be a bounded linear transformation. Let $R_2 = L[R_1]$.
\begin{itemize}
\item[(i)] Suppose that $R_2$ is pseudorectifiable. Then if $T_i:R_i\to\Grass_k(\HH_i)$ are tangent plane functions with respect to the measures $\mu_i$ ($i = 1,2$), then
\begin{align} \label{pseudolinear1}
&\int \#(L^{-1}(\yy)\cap A) \;\dee\mu_2(\yy) = \int_A \det(L\given T_1(\xx)) \;\dee\mu_1(\xx) \text{ for all $A\subset R_1$};\\ \label{pseudolinear2}
&T_2(L(\xx)) = L[T_1(\xx)] \text{ for $\mu_1$-a.e. $\xx\in R_1$ such that $\det(L\given T_1(\xx)) > 0$}.
\end{align}
\item[(ii)] If
\begin{equation}
\label{pseudolinear3}
\det\big(L\given T_1(\xx)\big) > 0 \text{ for $\mu_1$-a.e. $\xx\in R_1$}
\end{equation}
(e.g. if $L$ is invertible), then $R_2$ is pseudorectifiable (and in particular part (i) applies).
\end{itemize}
\end{lemma}
\begin{remark*}
If the assumption \eqref{pseudolinear3} is omitted then (ii) is false; see Example \ref{examplemuR1}(iii,v). However, this assumption can be replaced by the assumption that $\dim(\HH_2) < \infty$; this will follow from Claim \ref{claimrectifiable} below.
\end{remark*}
Before we begin the proof, we note that combining \eqref{pseudorectifiable} with the ordinary change-of-variables formula for Lebesgue measure yields that if $R\subset\HH$ is a pseudorectifiable set and if $\pi:\HH\to V$ is a linear map to a $k$-dimensional Euclidean space $V$, then
\[
\int \#(\pi^{-1}(\yy)\cap A) \; \dee\scrH^k(\yy) = \int_A \det\big(\pi\given T_R(\xx)\big) \; \dee\mu_R(\xx).
\]
In particular, the case $\pi = \pi_V\circ L$ will occur twice in the proof below.

\NPC{Proof}
\begin{proof}[Proof of Lemma \ref{lemmapseudolinear}(i)]
Let
\begin{align*}
\w\mu_1(A) &= \int_A \det\big(L\given T_1(\xx)\big) \;\dee\mu_1(\xx),&
\w T_1(\xx) &= L[T_1(\xx)];\\
\w\mu_2(A) &= \int \#(L^{-1}(\yy)\cap A) \; \dee\mu_2(\xx), &
\w T_2(\xx) &= T_2(L(\xx)).
\end{align*}
Then for all $V\in\Grass_k(\HH_2)$ and $A\subset R_1$,
\begin{align*}
&\int_A \det\big(\pi_V\given \w T_1(\xx)\big) \; \dee\w \mu_1(\xx)
= \int_A \det\big(\pi_V\circ L\given T_1(\xx)\big) \; \dee \mu_1(\xx)\\
&= \int \#\big((\pi_V\circ L)^{-1}(\zz)\cap A\big) \;\dee\scrH^k(\zz)
= \int \sum_{\yy\in \pi_V^{-1}(\zz)} \#(L^{-1}(\yy)\cap A) \;\dee\scrH^k(\zz)\\
&= \int \det\big(\pi_V\given T_2(\yy)\big) \#(L^{-1}(\yy)\cap A) \dee\mu_2(\xx)
= \int_A \det\big(\pi_V\given \w T_2(\xx)\big) \; \dee\w\mu_2(\xx).
\end{align*}
So by Lemma \ref{lemmaTPunique}, $\w\mu_1 = \w\mu_2$ and $\w\mu_1(\w T_1\neq \w T_2) = 0$. These assertions are encoded respectively in \eqref{pseudolinear1} and \eqref{pseudolinear2}.
\end{proof}
\begin{proof}[Proof of Lemma \ref{lemmapseudolinear}(ii)]
This proof is similar to the proof of Proposition \ref{propositionrectifiable}. Let $\gg:R_2\to R_1$ be a measurable partial inverse of $L$, so that $L\circ \gg(\yy) = \yy$ for all $\yy\in R_2$. We then define $\mu_2$ and $T_2$ as follows:
\begin{align*}
\mu_2(A) &= \int_{\gg(A)} \det\big(L\given T_1(\xx)\big) \;\dee\mu_1(\xx) \;\; (A\subset R_2), &
T_2(\yy) &= L[T_1(\gg(\yy))] \;\; (\yy\in R_2).
\end{align*}
Then for all $V\in\Grass_k(\HH_2)$ and $A\subset R_2$,
\begin{align*}
&\int \#(\pi_V^{-1}(\zz)\cap A) \;\dee\scrH^k(\zz)
= \int \#\big((\pi_V\circ L)^{-1}(\zz)\cap \gg(A)\big) \;\dee\scrH^k(\zz)\\
&= \int_{\gg(A)} \det\big(\pi_V\circ L\given T_1(\xx)\big) \; \dee \mu_1(\xx)
= \int_{\gg(A)} \det\big(\pi_V\given L[T_1(\xx)]\big) \det\big(L\given T_1(\xx)\big) \; \dee \mu_1(\xx)\\
&= \int \sum_{\xx\in L^{-1}(\yy)} \det\big(\pi_V\given L[T_1(\xx)]\big) \one_{g(A)}(\xx) \;\dee\mu_2(\yy) = \int_A \det\big(\pi_V\given T_2(\yy)\big) \;\dee\mu_2(\yy),
\end{align*}
i.e. \eqref{pseudorectifiable} holds. Now if $\AA$ is any countable partition of $\HH_1$ such that $\mu_1(A) < \infty$ for all $A\in\AA$, then $\gg^{-1}(\AA)$ is a countable partition of $\HH_2$ such that $\mu_2(A) < \infty$ for all $A\in \gg^{-1}(\AA)$. Thus $\mu_2$ is $\sigma$-finite.

By Lemma \ref{lemmamuleqH}, $\mu_2\leq\scrH^k\given_{R_2}$. To show that $\scrH^k\given_{R_2}\lessless\mu_2$, fix $A\subset R_2$ such that $\scrH^k(A) > 0$. By Fact \ref{factlipschitz}, $\scrH^k(\gg(A)) > 0$, so since $\mu_1\asymp\scrH^k\given_{R_1}$, we get $\mu_1(\gg(A)) > 0$. By \eqref{pseudolinear3}, this implies that $\mu_2(A) > 0$.
\end{proof}

\begin{remark}
\label{remarkextendedTP}
Fix $S\subset\HH$ such that $\scrH^k\given_S$ is $\sigma$-finite. By Observation \ref{observationpseudounion}, there exists a pseudorectifiable set $R\subset S$ which $\scrH^k$-almost contains any other pseudorectifiable subset of $S$. It is sometimes convenient to define the \emph{extended tangent plane function} $T_S:S\to\what\Grass_k(\HH) := \Grass_k(\HH)\cup\{\token\}$ via the formula
\[
T_S(\xx) = \begin{cases}
T_R(\xx) & \xx\in R\\
\token & \xx\in S\butnot R
\end{cases}.
\]
Then if $L:\HH_1\to\HH_2$ is an invertible bounded linear map between Hilbert spaces, $S_1\subset\HH_1$, and $S_2 = L[S_1]$, it follows from Lemma \ref{lemmapseudolinear} that
\[
T_{S_2}(L(\xx)) = L[T_{S_1}(\xx)] \text{ for $\scrH^k$-a.e. $\xx\in S_1$},
\]
where by convention $L(\token) = \token$.
\end{remark}

\subsection{Pseudorectifiability and pseudounrectifiability}
The counterpart to the classical notion of rectifiability is \emph{pure unrectifiability}, defined as follows:

\begin{definition}
\label{definitionunrectifiable}
A set $U\subset\HH$ is called \emph{purely ($k$-)unrectifiable} if for every $k$-rectifiable set $R\subset\HH$,
\[
\scrH^k(R\cap U) = 0.
\]
\end{definition}
It is easy to prove directly from the definitions that
\begin{proposition}[{\cite[Theorem 15.6]{Mattila}}]
\label{propositionpartsv1}
Every set $S\subset\HH$ such that $\scrH^k\given_S$ is $\sigma$-finite can be written as the disjoint union of a rectifiable set $R_S$ and a purely unrectifiable set $U_S$. This decomposition is unique up to $\scrH^k$-nullset.
\end{proposition}
The sets $R_S$ and $U_S$ in Proposition \ref{propositionpartsv1} are called the \emph{rectifiable part} and \emph{purely unrectifiable part} of $S$, respectively.

The main theorem we will need regarding purely unrectifiable sets is the Besicovitch--Federer projection theorem:

\begin{theorem}[{\cite[Theorem 18.1]{Mattila}}]
\label{theorembesicovitchfederer}
Suppose $\dimH < \infty$, and let $U\subset\HH$ be a purely unrectifiable such that $\scrH^k\given_U$ is $\sigma$-finite.\Footnote{Although \cite[Theorem 18.1]{Mattila} is stated for the case $\scrH^k(U) < \infty$, the generalization to the case where $\scrH^k\given_U$ is $\sigma$-finite is trivial.} Then for Lebesgue-a.e. $V\in\Grass_k(U)$,
\[
\scrH^k(\pi_V(U)) = 0.
\]
\end{theorem}

Using this theorem, we can give a short proof of a result mentioned above, namely that in finite dimensions, pseudorectifiability is equivalent to rectifiability:

\begin{proposition}
\label{propositionfindimconverse}
Every pseudorectifiable subset of a finite-dimensional vector space is rectifiable.
\end{proposition}
\NPC{Proof}
\begin{proof}
By Proposition \ref{propositionpartsv1}, it suffices to show that any pseudorectifiable purely unrectifiable set in a finite-dimensional vector space has $\scrH^k$ measure zero. Indeed, suppose $R\subset\HH$ is such a set; then by Theorem \ref{theorembesicovitchfederer}, $\eqref{pseudorectifiable}_{\mu_R = 0}$ holds for Lebesgue-a.e. $V\in\Grass_k(\HH)$. By Remark \ref{remarkTPunique}, this is sufficient to conclude that $\mu_R = 0$. Since $\mu_R\asymp\scrH^k\given_R$, it follows that $\scrH^k(R) = 0$.
\end{proof}

We now turn to the definition of pseudounrectifiabilty. Rather than mimicking Definition \ref{definitionunrectifiable} and defining pseudounrectifiability in terms of pseudorectifiability, it turns out to be more fruitful to give an independent definition of pseudounrectifiability, and then to prove the analogue of Proposition \ref{propositionpartsv1} as the limit of its finite-dimensional version.

\begin{definition}
\label{definitionpseudounrectifiable}
A set $U\subset\HH$ will be called \emph{purely ($k$-)pseudounrectifiable} if it can be written in the form $U = \bigcup_n U_n$, where for each $n$ there exists $V_n\in\Grass(\HH)$ such that for all $V\in\Grass(\HH)$ with $V\geq V_n$, the set $\pi_V(U_n)$ is purely $k$-unrectifiable.
\end{definition}
The following lemma demonstrates the ``orthogonality'' of the notions of pseudorectifiability and pure pseudounrectifiability, which followed directly from the definition in the classical setup:

\begin{lemma}
\label{lemmanullintersection}
The intersection of a pseudorectifiable set and a purely pseudounrectifiable set is an $\scrH^k$-nullset.
\end{lemma}
\NPC{Proof}
\begin{proof}
By contradiction suppose that $R$ is a pseudorectifiable set and that $U$ is a purely pseudounrectifiable set such that $\scrH^k(R\cap U) > 0$. Let $U_n$ and $V_n$ be as in Definition \ref{definitionpseudounrectifiable}; then $\scrH^k(R\cap U_n) > 0$ for some $n$. Let $A = R\cap U_n$; then for all $V\in\Grass(\HH)$ with $V\geq V_n$, $\pi_V(A)$ is purely unrectifiable. It follows from Theorem \ref{theorembesicovitchfederer} that for a dense set of $V\in\Grass_k(\HH)$, $\scrH^k(\pi_V(A)) = 0$. So by \eqref{pseudorectifiable}, we have
\begin{equation}
\label{mu0}
\int_A \det\big(\pi_V \given T_R(\xx)\big) \;\dee\mu_R(\xx) = 0
\end{equation}
for a dense set of $V\in\Grass_k(\HH)$. By continuity, \eqref{mu0} holds for all $V\in\Grass_k(\HH)$, and then by Lemma \ref{lemmaTPunique}, $\mu_R(A) = 0$. But since $\scrH^k(A) > 0$, this contradicts that $\mu_R\asymp \scrH^k\given_R$.
\end{proof}

We are now ready to prove an analogue of Proposition \ref{propositionpartsv1} for the notions of pseudorectifiability and pseudounrectifability:

\begin{proposition}
\label{propositionpartsv2}
Let $S\subset\HH$ be a set such that $\scrH^k\given_S$ is $\sigma$-finite. Then there exists a pseudorectifiable set $R\subset S$ such that $S\butnot R$ is purely pseudounrectifiable. Moreover, this set $R$ is unique up to an $\scrH^k$-nullset.
\end{proposition}
\NPC{Proof}
\begin{proof}
The uniqueness assertion follows immediately from Lemma \ref{lemmanullintersection}, so let us demonstrate existence. It follows from Definition \ref{definitionpseudounrectifiable} that there exists a purely pseudounrectifiable set $U\subset S$ which $\scrH^k$-almost contains every other purely pseudounrectifiable set. To complete the proof, we just need to show that $R = S\butnot U$ is purely pseudorectifiable. We begin with the following:
\begin{claim}
\label{claimrectifiable}
For every $V\in\Grass(\HH)$, $\pi_V(R)$ is rectifiable.
\end{claim}
\begin{subproof}
Fix $V\in\Grass(\HH)$, and let $U_V$ be the purely unrectifiable part of $\pi_V(R)$ (cf. Proposition \ref{propositionpartsv1}); we aim to show that $\scrH^k(U_V) = 0$. Without loss of generality, we can assume that $\pi_V(\scrH^k\given_R)\given_{U_V}\lessless \scrH^k$, as otherwise we can achieve this by decreasing $U_V$ by an $\scrH^k$-nullset.

We claim that the set $U_R := \pi_V^{-1}(U_V)\cap R$ is purely pseudounrectifiable. Indeed, fix $W\in\Grass(\HH)$ such that $W\geq V$, and let $A\subset \pi_W(U_R)$ be a rectifiable set. Then $\pi_V(A)$ is a rectifiable subset of the purely unrectifiable set $U_V$, so $\scrH^k(\pi_V(A)) = 0$, and thus by our absolute continuity assumption,
\[
0 = \scrH^k\given_R\circ \pi_V^{-1}(\pi_V(A)) = \scrH^k(\pi_V^{-1}(\pi_V(A))\cap R) \geq \scrH^k(\pi_W^{-1}(A)\cap R).
\]
So by Fact \ref{factlipschitz}, $\scrH^k(A) = \scrH^k(\pi_W(\pi_W^{-1}(A)\cap R)) = 0$. Since $A$ was arbitrary, $\pi_W(U_R)$ is purely unrectifiable, and since $W\geq V$ was arbitrary, $U_R$ is purely pseudounrectifiable.

By the definitions of of $U$ and $R$, we get $\scrH^k(U_R) = 0$, so by Fact \ref{factlipschitz}, $\scrH^k(U_V) = \scrH^k(\pi_V(U_R)) = 0$, i.e. $\pi_V(R)$ is rectifiable.
\end{subproof}
\begin{remark}
By Lemma \ref{lemmanullintersection}, the result of Claim \ref{claimrectifiable} holds for every pseudorectifiable set $R\subset\HH$.
\end{remark}

Let $(V_n)_1^\infty$ be an increasing sequence in $\Grass(\HH)$ whose union is dense in $\HH$, and for each $n$ write $\pi_n = \pi_{V_n}$ and $R_n = \pi_n(R)$. By Proposition \ref{propositionrectifiable}, we may let $T_n:R_n\to\Grass_k(V_n)$ be the tangent plane function with respect to the measure $\mu_n = \scrH^k\given_{R_n}$. For each $n$ let
\begin{align*}
\wbar\mu_n(A) &= \int \#(\pi_n^{-1}(\yy)\cap A) \; \dee\mu_n(\yy) \;\;(A\subset R), &
\wbar T_n(\xx) &= T_n(\pi_n(\xx)) \;\;(\xx\in R).
\end{align*}
Then for all $A\subset R$ and $n < m$,
\begin{align*}
\wbar\mu_n(A) &= \int \#(\pi_n^{-1}(\zz)\cap A) \; \dee\mu_n(\zz)
= \int \sum_{\yy\in\pi_n^{-1}(\zz)} \#(\pi_m^{-1}(\yy)\cap A) \dee\mu_n(\zz) \noreason\\
&= \int \det\left(\pi_n\given T_m(\yy)\right) \#(\pi_m^{-1}(\yy)\cap A) \dee\mu_m(\yy) \by{Lemma \ref{lemmapseudolinear}}\\
&= \int_A \det\left(\pi_n\given \wbar T_m(\xx)\right) \dee\wbar\mu_m(\xx).
\end{align*}
Writing this in the notation of Radon--Nikodym derivatives yields
\begin{equation}
\label{radonnikodym}
\frac{\dee\wbar\mu_n}{\dee\wbar\mu_m}(\xx) = \det\left(\pi_n\given \wbar T_m(\xx)\right).
\end{equation}
In particular, $\wbar\mu_n\leq \wbar\mu_m$ for all $n < m$, i.e. $(\wbar\mu_n)_1^\infty$ is an increasing sequence. On the other hand, Fact \ref{factlipschitz} automatically implies that $\wbar\mu_n\leq\scrH^k\given_R$. It follows that the measure
\[
\mu = \lim_{n\to\infty}\wbar\mu_n
\]
is finite and satisfies $\mu\leq\scrH^k\given_R$.

Define the function $T:R\to\Grass_k(\HH)$ by the formula
\[
T(\xx) = \lim_{n\to\infty} T_n(\xx).
\]
\begin{claim}
\label{claimwelldefined}
$T$ is well-defined $\mu$-a.e., and for each $n$, $\pi_n[T(\xx)] = \wbar T_n(\xx)$ for $\mu$-a.e. $\xx\in R$ such that $\det\big(\pi_n\given T(\xx)\big) > 0$.
\end{claim}
\begin{subproof}
For $\mu$-a.e. $\xx\in R$, we have
\[
1 = \frac{\dee\mu}{\dee\mu}(\xx) = \lim_{n\to\infty}\lim_{m\to\infty} \frac{\dee\wbar\mu_n}{\dee\wbar\mu_m}(\xx) = \lim_{n\to\infty} \lim_{m\to\infty} \det\big(\pi_n\given \wbar T_m(\xx)\big)
\]
and thus there exist infinitely many $n$ such that
\begin{equation}
\label{cnx}
c_n(\xx) := \lim_{m\to\infty} \det\big(\pi_n\given \wbar T_m(\xx)\big) > 0.
\end{equation}
Moreover, by Lemma \ref{lemmapseudolinear}, for each $n$ and for $\wbar\mu_n$-a.e. $\xx\in R$,
\begin{equation}
\label{tangentprojection}
\pi_m[\wbar T_\ell(\xx)] = \wbar T_m(\xx) \text{ for all $n < m < \ell$}.
\end{equation}
It follows that for $\mu$-a.e. $\xx\in R$, there exists $n$ such that both \eqref{cnx} and \eqref{tangentprojection} hold. Fix such an $\xx$. Then
\[
\det\big(\pi_{\wbar T_m(\xx)}\given \wbar T_\ell(\xx)\big) = \det\big(\pi_m \given \wbar T_\ell(\xx)\big) = \frac{\det\left(\pi_n \given \wbar T_\ell(\xx)\right)}{\det\left(\pi_n \given \wbar T_m(\xx)\right)} \tendsto{\ell,m\to\infty} \frac{c_n(\xx)}{c_n(\xx)} = 1.
\]
It then follows from an elementary geometric calculation that $(\wbar T_m(\xx))_1^\infty$ is a Cauchy sequence, i.e. $T(\xx)$ is well-defined. Moreover, taking the limit of \eqref{tangentprojection} as $\ell\to\infty$ shows that $\pi_m[T(\xx)] = \wbar T_m(\xx)$ for all $m > n$. By Lemma \ref{lemmapseudolinear}, this implies that $\pi_{n'}[T(\xx)] = \wbar T_{n'}(\xx)$ for all $n'\leq m$ and for $\mu$-a.e. $\xx\in R$ such that $\det\big(\pi_{n'}\given T(\xx)\big) > 0$. Since $m$ was arbitrary, this equality holds for all $n'$, which completes the proof.
\end{subproof}
Now we want to show that $T$ is a tangent plane function of $R$ with respect to $\mu$. Taking the limit of \eqref{radonnikodym} as $m\to\infty$ gives us
\[
\frac{\dee\wbar\mu_n}{\dee\mu}(\xx) = \det\big(\pi_n\given T(\xx)\big),
\]
i.e. the change-of-variables formula holds for the projections $\pi_n$. The equation $\pi_n[T(\xx)] = \wbar T_n(\xx)$ then allows us to verify that \eqref{pseudorectifiable} holds for all $V\in \bigcup_n \Grass_k(V_n)$.

Fix $W\in\Grass_k(\HH)$, and for each $n$ let $\w V_n = W + V_n$. Then the above argument yields a measure $\w\mu$ and a function $\w T:R\to\Grass_k(\HH)$ such that \eqref{pseudorectifiable} holds for all $V\in\bigcup_n \Grass_k(\w V_n)$. By Remark \ref{remarkTPunique}, $\w\mu = \mu$ and $\w T \equiv T$. But since $W\in\Grass_k(\w V_0)$, \eqref{pseudorectifiable} holds when we plug in $V = W$. Since $W$ was arbitrary, \eqref{pseudorectifiable} holds for all $V\in\Grass(\HH)$.

%

To finish the proof, we just need to show that $\mu\asymp\scrH^k\given_R$. We already know that $\mu\leq\scrH^k\given_R$, so it remains to show that $\scrH^k\given_R\lessless\mu$. Fix $A\subset R$ such that $\mu(A) = 0$. Then by \eqref{pseudorectifiable}, $\scrH^k(\pi_V(A)) = 0$ for all $V\in\Grass(\HH)$. It follows that $A$ is purely unpseudorectifiable, so from the definitions of $U$ and $R$ we have $\scrH^k(A) = 0$.
\end{proof}

\subsection{Pseudorectifiability and topological dimension}
We now prove the main result of this section, an infinite-dimensional analogue of Lemma \ref{lemmafederermattila} regarding pseudorectifiability.

\begin{proposition}
\label{propositionfederermattilainfdim}
Let $\SK\subset\HH$ be a compact set, and suppose that $\scrH^k(\SK) < \infty$, where $k = \TD(\SK)$. Then $\SK$ is not purely pseudounrectifiable. In particular, by Proposition \ref{propositionpartsv2}, $\SK$ contains a pseudorectifiable subset of positive $\scrH^k$ measure.
\end{proposition}
\NPC{Proof}
\begin{proof}
We use the following characterization of topological dimension (cf. \cite[Corollary of Theorem II.8]{Nagata} or \cite[Theorem 1.7.9]{Engelking}): $\TD(\SK)\geq k$ if and only if there exist closed sets $F_1,\ldots,F_k\subset \SK$ and relatively open sets $U_1,\ldots,U_k\subset \SK$ such that:
\begin{itemize}
\item[(i)] $F_i\subset U_i$ for all $i$,
\item[(ii)] If $W_1,\ldots,W_k\subset \SK$ are relatively open sets such that $F_i\subset W_i\subset U_i\all i$, then $\bigcap_{i = 1}^k \del W_i\neq \emptyset$, where $\del W_i$ denotes the boundary of $W_i$.
\end{itemize}
For each $i = 1,\ldots,k$ let $F_i' = \SK\butnot U_i$, and let
\[
3\epsilon_1 = \min_{i = 1}^k \dist(F_i,F_i') > 0.
\]
Since $\SK$ is compact, there exists a finite set $A\subset \SK$ such that $\SK\subset\thickvar A{\epsilon_1}$. Let $V_0\in\Grass(\HH)$ be the linear span of $A$, and let $\pi_0 = \pi_{V_0}$. Then
\begin{equation}
\label{pi0epsilon}
\|\pi_0(\xx) - \xx\| \leq \epsilon_1 \text{ for all $\xx\in \SK$}.
\end{equation}
For each $i = 1,\ldots,k$ let $\wbar{F_i} = \pi_0(F_i)$ and $\wbar{F_i'} = \pi_0(F_i')$; then by \eqref{pi0epsilon}
\[
\dist(\wbar{F_i},\wbar{F_i'}) \geq \dist(F_i,F_i') - 2\epsilon_1 \geq \epsilon_1 > 0.
\]
So by the smooth version of Urysohn's lemma, 
there exists a smooth function $f_i:V_0\to\R$ such that $f_i = -1$ on $\wbar{F_i}$ and $f_i = 1$ on $\wbar{F_i'}$. Let $\ff = (f_1,\ldots,f_k):V_0\to\R^k$.
\begin{claim}[Cf. {\cite[Proof of Theorem III.1]{Nagata}}]
\label{claimstablevalue}
If $\gg:\SK\to\R^k$ is a continuous function satisfying
\begin{equation}
\label{condg}
\max_\SK \|\ff\circ\pi_0 - \gg\| < 1/2,
\end{equation}
then
\[
[-1/2,1/2]^k \subset \gg(\SK).
\]
\end{claim}
\begin{subproof}
Given $\yy\in [-1/2,1/2]^k$, for each $i = 1,\ldots,k$ define $W_i = \{\xx\in \SK : g_i(\xx) < y_i\}$.  The condition \eqref{condg} implies that $F_i\subset W_i\subset U_i$, so we have $\bigcap_{i = 1}^k \del W_i\neq \emptyset$. After choosing $\xx\in\bigcap_{i = 1}^k \del W_i$, we have $\xx\in \SK$ and $\gg(\xx) = \yy$, so $\yy\in\gg(\SK)$.
\end{subproof}

Now fix $V\in\Grass(\HH)$ with $V\geq V_0$, and let $R_V$ and $U_V$ be the rectifiable and purely unrectifiable parts of $\pi_V(\SK)$, respectively (cf. Proposition \ref{propositionpartsv1}). 
Let $M = \max_{\xx\in \SK} \|\xx\|$ and $\epsilon_2 = 1/(2M)$, and let
\[
\SS = \{A \in L(V,\R^k) : \|A\| < \epsilon_2\}.
\]
\begin{claim}
There exists $A\in\SS$ such that
\begin{equation}
\label{ASdef}
\scrH^k\big((\ff\circ\pi_0 + A)(U_V)\big) = 0.
\end{equation}
\end{claim}
\begin{subproof}
Since pure unrectifiability is preserved under bi-Lipschitz maps, the set
\[
\wbar U = \ff\circ\pi_0\oplus I_V(U_V) = \left[\begin{array}{c}\ff\circ\pi_0 \\ I_V\end{array}\right](U_V) \subset \R^k\oplus V
\]
is purely unrectifiable. Moreover, for all $A \in L(V,\R^k)$,
\[
(\ff\circ\pi_0 + A)(U_V) = \left[\begin{array}{cc} I_k & A\end{array}\right](\wbar U) \subset \R^k.
\]
Thus by Theorem \ref{theorembesicovitchfederer}, \eqref{ASdef} holds for Lebesgue-a.e. $A \in L(V,\R^k)$, so in particular \eqref{ASdef} holds for some $A\in\SS$.
\end{subproof}
Let $A\in\SS$ be as in the claim, and let $\gg = \ff\circ\pi_0 + A : V\to\R^k$. Then by Claim \ref{claimstablevalue}, $\gg\circ\pi_V(\SK)\supset [-1/2,1/2]^k$, so
\[
\scrH^k(R_V) \gtrsim_\times \scrH^k\big(\gg(R_V)\big) = \scrH^k\big(\gg\circ \pi_V(\SK)\big) \geq \scrH^k([-1/2,1/2]^k) \asymp_\times 1,
\]
where the first asymptotic comes from the fact that
\[
\max_\SK\|\gg'\| \leq \max_{\pi_0(\SK)}\|\ff'\| + \epsilon_2 \lesssim_\times 1.
\]
So
\[
2\epsilon_3 := \inf_{\substack{V\in\Grass(\HH) \\ V\geq V_0}} \scrH^k(R_V) > 0.
\]
By contradiction, suppose that $\SK$ is purely pseudounrectifiable, and let $U_n$ and $V_n$ be as in Definition \ref{definitionpseudounrectifiable}. By choosing $N$ large enough, we get $\scrH^k\big(\bigcup_{n\leq N} U_n\big) \geq \scrH^k(\SK) - \epsilon_3$. Letting $V = V_0 + \sum_1^N V_n$ and $A = \pi_V\big(\bigcup_{n\leq N}U_n\big)$, we have
\[
\scrH^k(A\cap R_V) \geq \scrH^k(R_V) - \scrH^k(\pi_V(\SK)\butnot A) \geq 2\epsilon_3 - \epsilon_3 > 0,
\]
contradicting Definitions \ref{definitionunrectifiable} and \ref{definitionpseudounrectifiable}.
\end{proof}
\begin{remark*}
In the last step of this proof we used the hypothesis $\scrH^k(\SK) < \infty$ in a crucial way, namely to guarantee the existence of $N$ such that $\scrH^k\big(\bigcup_{n\leq N} U_n\big) \geq \scrH^k(\SK) - \epsilon_3$.
\end{remark*}

\subsection{More invariance properties of pseudorectifiability}
In the proof of Theorem \ref{theoremrigidity}($d = \infty$, $k > 1$), we will need the class of pseudorectifiable sets to be closed under certain operations more general than just images under linear maps (i.e. Lemma \ref{lemmapseudolinear}): specifically, we need to be able to do a cone construction, and we need to be able to say something about subsets of the graph of a differentiable function. Note that more general properties such as the closure of pseudorectifiability under general differentiable maps, while not entirely implausible, seem to be difficult to prove.


\begin{lemma}
\label{lemmaM}
Let $L\subset\HH$ be an affine subspace such that $\0\notin L$, let $R_1\subset L$, and let $R_2 = \R R_1$. Then $R_2$ is $(k + 1)$-pseudorectifiable if and only if $R_1$ is $k$-pseudorectifiable. Moreover, in this case
\begin{align*}
\mu_2(A) &= \int_{M^{-1}(A)} \det\big(M'(\xx,t)\given T_1(\xx)\oplus\R\big) \;\dee(\mu_1\times\lambda)(\xx,t),\\
T_2(t\xx) &= \R\xx + T_1(\xx),
\end{align*}
where $M(\xx,t) = t\xx$. Here $\lambda$ denotes Lebesgue measure on $\R$.
\end{lemma}
\begin{proof}
First suppose that $R_1$ is $k$-pseudorectifiable. Fix $V\in\Grass_{k + 1}(\HH)$. By Claim \ref{claimrectifiable}, $R_3 = \pi_V(R_1)$ is $k$-rectifiable. It follows that $R_4 = R_3\times\R$ is $(k + 1)$-rectifiable and $T_4(\zz,t) = T_3(\zz)\oplus\R$. Fix $A\subset R_2$. Then
\begin{align*}
&\int \#(\pi_V^{-1}(\yy)\cap A) \;\dee\scrH^{k + 1}(\yy)\\
&= \int \sum_{(\zz,t)\in M^{-1}(\yy)} \#\big(t(\pi_V^{-1}(\zz)\cap R_1)\cap A\big) \;\dee\scrH^{k + 1}(\yy) \since{$R_1\subset L$}\\
&= \int \det\big(M'(\zz,t)\given T_4(\zz,t)\big) \#\big(t(\pi_V^{-1}(\zz)\cap R_1)\cap A\big) \;\dee\scrH^{k + 1}(\zz,t) \by{\eqref{rectifiable}}\\
&= \int\int \det\big(tI\oplus\zz\given T_3(\zz)\oplus\R\big) \#\big(\pi_V^{-1}(\zz)\cap (t^{-1}A\cap R_1)\big) \;\dee\scrH^k(\zz)\;\dee\lambda(t)\noreason\\
&= \int\int_{t^{-1}A\cap R_1} \det\big(tI\oplus\pi_V(\xx)\given T_3(\pi_V(\xx))\oplus\R\big) \det\big(\pi_V\given T_1(\xx)\big) \;\dee\mu_1(\xx) \;\dee\lambda(t)\noreason\\
&= \int\int_{t^{-1}A\cap R_1} \det\big(t\pi_V\oplus\pi_V(\xx)\given T_1(\xx)\oplus\R\big) \;\dee\mu_1(\xx) \;\dee\lambda(t)\\
&= \int\int_{t^{-1}A\cap R_1} \det\big(\pi_V\given T_1(\xx) + \R\xx\big) \det\big(tI\oplus\xx\given T_1(\xx)\oplus\R\big) \;\dee\mu_1(\xx) \;\dee\lambda(t)\noreason\\
&= \int_A \det\big(\pi_V\given T_2(\ww)\big) \;\dee\mu_2(\ww).
\end{align*}
It is clear that $\mu_2$ is $\sigma$-finite, and by Lemma \ref{lemmamuleqH}, $\mu_2\leq\scrH^{k + 1}\given_{R_2}$. An elementary covering argument (cf. \cite[Product formula 7.3]{Falconer_book}) shows that $\scrH^{k + 1}\given_{R_1\times\R^*} \lessless \scrH^k\given_{R_1}\times\lambda$. Since $M$ is locally Lipschitz, Fact \ref{factlipschitz} shows that $\scrH^{k + 1}\given_{R_2}\lessless M(\mu_1\times\lambda)$. Finally, since $R_1\subset L$, the determinant $\det\big(M'(\xx,t)\given T_1(\xx)\oplus\R\big)$ is always nonzero, so $\scrH^{k + 1}\given_{R_2}\lessless \mu_2$, completing the proof that $R_2$ is $(k + 1)$-pseudorectifiable.

On the other hand, suppose that $R_1$ is not $k$-pseudorectifiable. Then by Proposition \ref{propositionpartsv2}, there exist $U_1\subset R_1$ and $V_1\in\Grass(\HH)$ such that $\scrH^k(U_1) > 0$ and such that for all $V\geq V_1$, $\pi_V(U_1)$ is purely $k$-unrectifiable. Let $U_2 = \R U_1$, and fix $V\geq V_1$. By Theorem \ref{theorembesicovitchfederer}, we have $\scrH^k(\pi_W(U_1)) = 0$ for Lebesgue-a.e. $W\in\Grass_k(V)$. Since $\pi_W(U_2) = \R\pi_W(U_1)$ and $\scrH^{k + 1}\lessless M(\scrH^k\times\lambda)$ as proven above, we have $\scrH^{k + 1}(\pi_W(U_2)) = 0$ for Lebesgue-a.e. $W\in\Grass_k(V)$. Thus $\pi_V(U_2)$ is purely $(k + 1)$-unrectifiable, and since $V$ was arbitrary, $U_2$ is purely $(k + 1)$-pseudounrectifiable. On the other hand, an elementary covering argument based on weighted Hausdorff measures shows that $\scrH^{k + 1}(U_2) > 0$ (cf. \cite[Lemma 8.16]{Mattila}). Thus, by Lemma \ref{lemmanullintersection} $R_2$ is not $(k + 1)$-pseudorectifiable.
\end{proof}

\begin{lemma}
\label{lemmagraph}
Let $f:\HH\to\R$ be a differentiable function, let $F(\xx) = (\xx,f(\xx))\in\HH\oplus\R$, fix $R_1\subset\HH$, and let $R_2 = F(R_1)\subset\HH\oplus\R$. Then if $R_2$ is pseudorectifiable, then $R_1$ is pseudorectifiable, and
\begin{align*}
\mu_2(A) &= \int_{\pi_\HH(A)} \det\big(F'(\xx)\given\HH\big) \;\dee\mu_1(\xx),&
T_2(F(\xx)) &= F'(\xx)[T_1(\xx)].
\end{align*}
Here $\det\big(F'(\xx)\given\HH\big)$ is interpreted in the obvious way, i.e.
\[
\det\big(F'(\xx)\given\HH\big) = \sqrt{1 + \|f'(\xx)\|^2}.
\]
\end{lemma}
\begin{remark}
It seems difficult to prove the converse (that if $R_1$ is pseudorectifiable, then so is $R_2$). Fortunately, it is not necessary for the proof of Theorem \ref{theoremrigidity}($d = \infty$, $k > 1$).
\end{remark}
\begin{proof}
Fix $V\in\Grass_k(\HH\oplus\R)$ and $A\subset R_1$. Then
\begin{align*}
\int_{F(A)} \det\big(\pi_V\given T_2(\xx)\big) \;\dee\mu_2(\xx)
&= \int \#\big((\pi_V\circ F)^{-1}(\yy)\cap A\big)\;\dee\scrH^k(\yy) \by{\eqref{pseudorectifiable}}\\
&\leq \int_A \det\big((\pi_V\circ F)'(\xx)\big)\;\dee\scrH^k(\xx) \note{see below}\\
&= \int_A \det\big(\pi_V\given F'(\xx)[\HH]\big)\det\big(F'(\xx)\given\HH\big)\;\dee\scrH^k(\xx).\noreason
\end{align*}
To justify the inequality, let $\AA$ be a partition of $A$ into small pieces, and for each $B\in\AA$, apply Fact \ref{factlipschitz} to the map $L_B\circ\pi_V\circ F$, where $L_B:V\to\R^k$ is a linear map chosen so that $L_B\circ\pi_V\circ F$ is $\lambda$-Lipschitz on $B$ for some $\lambda\sim 1$, but $\det(L_B) \sim 1/\det\big((\pi_V\circ F)'(\xx)\big)$ for all $\xx\in B$. (For example, we could let $L_B = [(\pi_V\circ F)'(\xx)]^{-1}$ for any $\xx\in B$.)

Let $\nu_1 = \pi_\HH(\mu_2)$. Then we can write the above inequality in terms of Radon--Nikodym derivatives as
\begin{equation}
\label{piVT1}
\det\big(\pi_V\given T_2(F(\xx))\big) \leq \det\big(\pi_V\given F'(\xx)[\HH]\big)\det\big(F'(\xx)\given\HH\big) \frac{\dee\scrH^k\given_{R_1}}{\dee\nu_1}(\xx)
\end{equation}
for $\nu_1$-a.e. $\xx\in R_1$. Note that the Radon--Nikodym derivative is well-defined since by Fact \ref{factlipschitz}, $\scrH^k\given_{R_1}\lessless \nu_1$. Since \eqref{piVT1} is continuous with respect to $V$, for $\nu_1$-a.e. $\xx\in R_1$, it holds for all $V\in\Grass_k(\LL)$. Fix such an $\xx$. Then for all $V\in\Grass_k(\LL)$ such that $\det\big(\pi_V\given F'(\xx)[\HH]\big) = 0$, we have $\det\big(\pi_V\given T_2(F(\xx))\big) = 0$. It follows that $T_2(F(\xx))\subset F'(\xx)[\HH]$.

In particular, $\det\big(\pi_\HH\given T_2(\xx)\big) > 0$ for $\mu_2$-a.e. $\xx\in R_2$. So by Lemma \ref{lemmapseudolinear}, $R_1$ is pseudorectifiable and satisfies
\begin{align*}
\mu_1(A) &= \int_{F(A)} \det\big(\pi_\HH\given T_2(\xx)\big) \;\dee\mu_2(\xx),&
T_1(\xx) &= \pi_\HH[T_2(F(\xx))].
\end{align*}
Since $T_2(F(\xx))\subset F'(\xx)[\HH]$, we must have $T_2(F(\xx)) = F'(\xx)[T_1(\xx)]$.
\end{proof}

\draftnewpage

\section{Proof of Meta-rigidity Theorem \ref{theoremrigidity}($\dimH = \infty$, $k > 1$)}
\label{sectionI}
We just need a few more preliminaries before we can begin the proof. Since pseudorectifiability is defined in terms of linear maps, it is not obvious what its equivariance properties are with respect to arbitrary conformal maps. We resolve this issue by conjugating our set $\SK$ into the projective model of conformal geometry. Since conformal maps are linear after conjugation to the projective model, this will allow us to use the tools of the previous section.

We therefore recall the definition of the projective model of conformal geometry, see e.g. \cite{AkivisGoldberg, Hertrich}. 
Let $\LL$ be a Hilbert space and let $\QQ$ be a nondegenerate quadratic form on $\LL$ of signature 1. For concreteness, we let $\LL = \R^2\oplus\HH$, and we define the quadratic form $\QQ$ by the formula
\[
\QQ(t_0,t_1,\xx) = 2t_0 t_1 + \|\xx\|^2.
\]
Let $[\LL]$ denote projective space over $\LL$, and for each $\yy\in\LL\butnot\{\0\}$, let $[\yy]$ denote the corresponding point in projective space, so that $[\yy] = [t\yy]$ for all $t\in\R\butnot\{0\}$. Let $L_\QQ = \{\xx\in\LL : \QQ(\xx) = 0\}$. Then $[L_\QQ]$ is a conformal manifold isomorphic to $\what\HH$.\Footnote{To be precise, we should give the conformal structure on $[L_\QQ]$: for each function $\alpha:L_\QQ\to(0,\infty)$ such that $\alpha(t\xx) = |t|\alpha(\xx)$, the Riemannian metric $g_\xx(\yy) = \alpha^{-2}(\xx)\QQ(\yy)$ on $L_\QQ$ is invariant under scaling, and so it induces a Riemannian metric on $[L_\QQ]$; the conformal structure on $[L_\QQ]$ is the class of such metrics.} Concretely, let
\[
\iota(\xx) = (1,-\|\xx\|^2/2,\xx) \;\; (\xx\in\HH).
\]
Then the map $\phi:\what\HH\to [L_\QQ]$ defined by the equation
\[
\phi(\xx) = \begin{cases}
[\iota(\xx)] & \xx\neq\infty\\
[(0,1,\0)] & \xx = \infty
\end{cases}
\]
is a conformal isomorphism between $\what\HH$ and $[L_\QQ]$. The fact that $\phi$ conjugates conformal maps to linear maps is expressed by the following well-known result:

\begin{proposition}[Cf. {\cite[Theorem 1.3.14]{Hertrich}} or {\cite[\61.1]{AkivisGoldberg}}]
\label{propositionPhi}
For every M\"obius transformation $g:\what\HH\to\what\HH$, there exists a $\QQ$-preserving bounded invertible linear transformation $g_\Proj:\LL\to\LL$ such that for all $\xx\in\what\HH$,
\[
\phi(g(\xx)) = [g_\Proj](\phi(\xx)).
\]
The transformation $g_\Proj$ is unique up to a factor of $-1$.
\end{proposition}
Although the cited references assume that $\dimH < \infty$, there is no essential difference to the case $\dimH = \infty$; cf. \cite[\62]{DSU}.

We are now ready to begin the proof of Theorem \ref{theoremrigidity}($d = \infty$, $k > 1$).

{\bf Step 1: Without loss of generality, $\mu = \scrH^k\given_\SK$.} This step proceeds in the same way as in the proof of Theorem \ref{theoremrigidity}($d < \infty$).

{\bf Step 2: Comparing the extended tangent plane functions in the two models.}
Let
\begin{align*}
\SK_\Euc &= \SK \subset\HH,&
\SK_\para &= \iota(\SK) \subset \{1\}\times\R\times\HH,&
\SK_\Proj &= \R\SK_\para \subset\LL,
\end{align*}
and let $T_\Euc:\SK_\Euc\to\what\Grass_k(\HH)$, $T_\para:\SK_\para\to\what\Grass_k(\LL)$, and $T_\Proj:\SK_\Proj\to\what\Grass_{k + 1}(\LL)$ denote the extended tangent plane functions (cf. Remark \ref{remarkextendedTP}). 
Let $\mu_\Euc = \mu = \scrH^k\given_{\SK_\Euc}$, $\mu_\para = \scrH^k\given_{\SK_\para}$, and $\mu_\Proj = \scrH^{k + 1}\given_{\SK_\Proj}$. Then $\mu_\Euc$, $\mu_\para$, and $\mu_\Proj$ are $\sigma$-finite and satisfy $\iota(\mu_\Euc)\asymp\mu_\para$, $\mu_\Proj \asymp M(\mu_\para\times\lambda)$, where $M$ is as in Lemma \ref{lemmaM} and $\lambda$ is Lebesgue measure. 
By Lemmas \ref{lemmaM} and \ref{lemmagraph},
\begin{align} \label{M}
T_\Proj(t\zz) &= T_\para(\zz) + \R\zz \;\; (\zz\in\SK_\para,\; t\in\R^*),\\ \label{graph}
T_\para(\iota(\xx)) &= \text{$\token$ or  $\iota'(\xx)[T_\Euc(\xx)]$} \;\; (\xx\in\HH),
\end{align}
with the understanding in \eqref{M} that $T_\Proj(t\zz) = \token$ if $T_\para(\zz) = \token$. The ambiguity in $T_\para(\iota(\xx))$ comes from the fact that we don't have a converse to Lemma \ref{lemmagraph}, so it may happen that $T_\para(\iota(\xx)) = \token$ while $T_\Euc(\xx) \neq\token$. The solution is to apply Proposition \ref{propositionfederermattilainfdim} to $\SK_\para$ instead of to $\SK_\Euc$, so we get $\mu_\para(T_\para\neq\token) > 0$. For convenience of notation, let
\[
\w T_\Euc(\xx) = \begin{cases}
T_\Euc(\xx) & T_\para(\iota(\xx))\neq\token\\
\token & \text{otherwise}
\end{cases},
\]
so that we can write
\begin{equation}
\label{graphv2}
T_\para(\iota(\xx)) = \iota'(\xx)[\w T_\Euc(\xx)] \;\; (\xx\in\HH)
\end{equation}
and $\mu_\Euc(\w T_\Euc(\xx)\neq\token) > 0$.

{\bf Step 3. A limit equation for $\w T_\Euc$.}
Apply Corollary \ref{corollarylebesguedifferentiation} to the function $f = \w T_\Euc : \SK\to Y = \what\Grass_k(\HH)$ to get a point $\pp\in\Rad_k(\mu)$ such that if $(g_n)_1^\infty$ is as in Definition \ref{definitionradial}, then
\[
\lim_{n\to\infty} \frac{1}{\|g_n'\|^k} \int_{g_n(U)} \dist_\Grass\big(\w T_\Euc(\xx),L_0\big) \;\dee\mu(\xx) = 0,
\]
where $L_0 = \w T_\Euc(\pp) \neq \token$. Then by \eqref{muconformal},
\[
\lim_{n\to\infty} \int_U \dist_\Grass\big(\w T_\Euc(g_n(\xx)),L_0\big) \;\dee\mu(\xx) = 0.
\]
Passing to a subsequence along which the convergence is at least geometrically fast and using Markov's inequality, we get
\begin{equation}
\label{limitA}
\lim_{n\to\infty} \w T_\Euc(g_n(\xx)) = L_0 \text{ for $\mu$-a.e. $\xx\in \SK\cap U$}.
\end{equation}
We now rewrite \eqref{limitA} using equivariance of the extended tangent plane function. For each $n$, let $g_n^\Proj:\LL\to\LL$ be as in Proposition \ref{propositionPhi}. Then by Lemma \ref{lemmapseudolinear},
\[
g_n^\Proj[T_n^\Proj(\xx)] = T_n^\Proj(g_n^\Proj(\zz)) \text{ for $\mu_\Proj$-a.e. $\zz\in \SK_\Proj\cap U_\Proj$},
\]
where $U_\Proj = \R\iota(U)$. So by \eqref{M} and \eqref{graphv2},
\begin{equation}
\label{equivariance}
g_n'(\xx)[\w T_\Euc(\xx)] = \w T_\Euc(g_n(\xx)) \text{ for $\mu$-a.e. $\xx\in \SK\cap U$ and for all $n$}.
\end{equation}
Fix $\xx\in \SK\cap U$ such that both \eqref{limitA} and \eqref{equivariance} hold. Since $g_n$ is conformal, $g_n'(\xx)$ is a similarity, so
\[
\dist_\Grass\big(\w T_\Euc(\xx),g_n'(\xx)^{-1}(L_0)\big) = \dist_\Grass\big(g_n'(\xx)[\w T_\Euc(\xx)],L_0\big) = \dist_\Grass\big(\w T_\Euc(g_n(\xx)),L_0\big) \tendsto n 0.
\]
Letting
\[
L_n^\Euc(\xx) = g_n'(\xx)^{-1}(L_0),
\]
we have
\begin{equation}
\label{limitB}
\dist_\Grass(L_n^\Euc(\xx),\w T_\Euc(\xx))\to 0 \text{ for $\mu$-a.e. $\xx\in \SK\cap U$}.
\end{equation}
In particular, $\w T_\Euc(\xx)\neq\token$ for $\mu$-a.e. $\xx\in \SK\cap U$.

{\bf Step 4: Extracting a convergent subsequence.}
To continue the proof, we switch back to the projective model. Let
\[
V_0 = \lb 0\rb\times\R\times L_0 \subset\LL,
\]
and for each $n$ let
\begin{align}
\label{Vndef}
V_n &= (g_n^\Proj)^{-1}[V_0],&
L_n^\Proj(\zz) &= \R\zz + V_n\cap\zz^\perp.
\end{align}
Here $\zz^\perp$ denotes the $\QQ$-orthogonal complement of $\zz$. It is readily verified that
\[
L_n^\Proj(\iota(\xx)) = \iota'(\xx)[L_n^\Euc(\xx)] + \R\iota(\xx) \all\xx\in\HH,
\]
so by \eqref{M} and \eqref{graphv2}, the pointwise convergence \eqref{limitB} implies that $L_n^\Proj(\zz)\to T_\Proj(\zz)$ for $\mu_\Proj$-a.e. $\zz\in U_\Proj$. Equivalently,
\[
\dist_\Grass(T_\Proj(\zz),\R\zz + V_n\cap\zz^\perp) \to 0 \all^{\mu_\Proj} \zz\in \SK_\Proj\cap U_\Proj.
\]
Let
\[
\dist_\subset(V,W) = \inf\{\epsilon > 0 : V\subset \NN_\Proj(W,\epsilon)\}.
\]
Then we may weaken our state of knowledge as follows:
\begin{align}
\label{limit}
\exists X\in\Grass(\LL) &\all^{\mu_\Proj} \zz\in U_\Proj\butnot X\\ \label{limitinternal}
& \dist_\subset(T_\Proj(\zz),\R\zz + V_n) \to 0.
\end{align}
(We will use the equation number \eqref{limit} to refer to both lines of this formula, while \eqref{limitinternal} refers only to the second line.)

The idea now is that if $(V_n)_1^\infty$ is not a Cauchy sequence, then we should be able to find another sequence $(V_n)_1^\infty$ satisfying \eqref{limit}, but with smaller dimension. The new sequence will be constructed as a sequence of ``almost-intersections'' of members of the old sequence. Specifically, we will use the following lemma:

\begin{lemma}
\label{lemmaepsilonintersection}
Fix $\epsilon\in (0,1)$ and $V_1,V_2\in\Grass(\LL)$. There exists $V\in\Grass(\LL)$ with $\dist_\subset(V,V_i)\leq\epsilon$ such that for all $\ww\in\LL$,
\[
\dist(\ww,V) \leq (3/\epsilon)\max_{i = 1}^2 \dist(\ww,V_i).
\]
\end{lemma}
The set $V$ will be called an \emph{$\epsilon$-intersection} of $V_1$ and $V_2$. Of course, it is not unique.

\NPC{Proof}
\begin{proof}
Let $\pi_i = \pi_{V_i}$ ($i = 1,2$), and consider the positive-definite quadratic form $\RR(\vv) = \|\vv\|^2 - \|\pi_2(\vv)\|^2 = \|\vv - \pi_2(\vv)\|^2$ on $V_1$. Choose an orthonormal basis $\ff_1,\ldots,\ff_n$ of $V_1$ with respect to which $\RR$ is diagonal, say
\[
\RR\left(\sum_{i = 1}^n v_i \ff_i\right) = \sum_{i = 1}^n c_i v_i^2
\]
for some $c_1,\ldots,c_n \in [0,1]$. Without loss of generality suppose that $c_1 \leq \cdots \leq c_m \leq \epsilon^2 < c_{m + 1} \leq \cdots \leq c_n$. (The cases $m = 0$ and $m = n$ are allowed.) Let $V = \sum_1^m \R\ff_i \leq V_1$. Then for all $\vv\in V$, $\|\vv - \pi_2(\vv)\|^2 = \RR(\vv) \leq \epsilon^2 \|\vv\|^2$, so $\|\vv - \pi_2(\vv)\| \leq \epsilon\|\vv\|$. Thus $\dist_\subset(V,V_2)\leq \epsilon$.

Now fix $\ww\in\LL$, and let $\vv_i = \pi_i(\ww)$ ($i = 1,2$) and $\delta = \max_{i = 1}^2 \dist(\ww,V_i)$. Then
\begin{align*}
\|\ww - \vv_i\| &\leq \delta\\
\|\vv_1 - \pi_2(\vv_1)\| &\leq \|\vv_1 - \vv_2\| \leq 2\delta\\
4\delta^2 &\geq \|\vv_1 - \pi_2(\vv_1)\|^2 = \RR(\vv_1) = \sum_{i = 1}^n c_i v_{1,i}^2\\
& \geq \epsilon^2 \sum_{i = m + 1}^n v_{1,i}^2 = \epsilon^2 \dist^2(\vv_1,V)\\
\dist(\vv_1,V) &\leq 2\delta/\epsilon\\
\dist(\ww,V) &\leq 2\delta/\epsilon + \delta \leq 3\delta/\epsilon.
\qedhere\end{align*}
\end{proof}

\begin{claim}
\label{claimintersection}
Let $(V_n^{(1)})_1^\infty$ and $(V_n^{(2)})_1^\infty$ be two sequences of bounded dimension satisfying \eqref{limit}. Fix $\epsilon > 0$, and for each $n$ let $V_n$ be an $\epsilon$-intersection of $V_n^{(1)}$ and $V_n^{(2)}$. Then the sequence $(V_n)_1^\infty$ satisfies \eqref{limit}.
\end{claim}
\begin{proof}
Fix $\zz\in U_\Proj$ such that $\eqref{limitinternal}_{V_n = V_n^{(i)}}$ holds, and without loss of generality suppose $\|\zz\| = 1$. For each $n\in\N$ let
\[
\delta_n = \max_{i = 1}^2 \dist_\subset(T_\Proj(\zz),\R\zz + V_n^{(i)}),
\]
so that $\delta_n\to 0$. Fix $n\in\N$ and $\vv\in T_\Proj(\zz)$, and for $i = 1,2$ find $a_i\in\R$ and $\vv_i\in V_n^{(i)}$ such that
\[
\|\vv - (a_i\zz + \vv_i)\| \leq \delta_n \|\vv\|.
\]
Subtracting gives
\[
\|(a_2 - a_1)\zz + (\vv_2 - \vv_1)\| \leq 2\delta_n \|\vv\|
\]
and thus
\[
|a_2 - a_1| \leq \frac{2\delta_n\|\vv\|}{\dist(\zz,V_n^{(1)} + V_n^{(2)})}\cdot
\]
Since
\begin{align*}
\dist(\vv_1,V_n^{(2)}) \leq \dist\big((a_2 - a_1)\zz + \vv_2,V_n^{(2)}\big) + 2\delta_n \|\vv\| \leq |a_2 - a_1| + 2\delta_n \|\vv\|,
\end{align*}
by Lemma \ref{lemmaepsilonintersection} we get
\[
\dist(\vv_1,V_n) \leq (3/\epsilon)(|a_2 - a_1| + 2\delta_n \|\vv\|) \lesssim_\times \frac{\delta_n\|\vv\|}{\dist(\zz,V_n^{(1)} + V_n^{(2)})}
\]
and thus
\[
\dist(\vv,\R\zz + V_n) \leq \|\vv - (a_1\zz + \vv_1)\| + \dist(\vv_1,V_n) \lesssim_\times \frac{\delta_n\|\vv\|}{\dist(\zz,V_n^{(1)} + V_n^{(2)})}\cdot
\]
Since $\vv$ was arbitrary,
\[
\dist_\subset(T_\Proj(\zz),\R\zz + V_n) \lesssim_\times \frac{\delta_n}{\dist(\zz,V_n^{(1)} + V_n^{(2)})},
\]
so \eqref{limitinternal} holds for all $\zz\in U_\Proj$ which satisfy both $\eqref{limitinternal}_{V_n = V_n^{(i)}}$ and
\begin{equation}
\label{liminfy}
\liminf_{n\to\infty} \dist(\zz,V_n^{(1)} + V_n^{(2)}) > 0.
\end{equation}
Now let
\[
W = \left\{\zz\in\LL : \dist(\zz,V_n^{(1)} + V_n^{(2)}) \tendsto n 0\right\}.
\]
Evidently, $W$ is a vector space of dimension at most $2\ell$, where $\ell$ is the maximum dimension of the vector spaces $V_n^{(i)}$ ($n\in\N$, $i = 1,2$). But if there exists any $\zz\in\LL\butnot W$ such that \eqref{liminfy} fails, then by passing to a subsequence we may add $\zz$ to $W$, thus increasing its dimension; thus by passing to a subsequence which maximizes $\dim(W)$, we guarantee \eqref{liminfy} for all $\zz\in\LL\butnot W$. Letting $X_i$ be as in $\eqref{limit}_{V_n = V_n^{(i)}}$, we get \eqref{limitinternal} for $\mu_\Proj$-a.e. $\zz\in U_\Proj\butnot (W + X_1 + X_2)$, i.e. \eqref{limit} holds.
\end{proof}

Now let $\ell$ be the smallest number such that there exists a sequence $(V_n)_1^\infty$ of $\ell$-dimensional subspaces which satisfies \eqref{limit}, and let $(V_n)_1^\infty$ be such a sequence.
\begin{claim}
$(V_n)_1^\infty$ is a Cauchy sequence.
\end{claim}
\begin{proof}
Suppose not; then there exist $\epsilon > 0$ and sequences $n_j,m_j\to\infty$ such that for all $j$, $\dist_\Grass(V_{n_j},V_{m_j}) \geq 3\epsilon$. Write $W_j^{(1)} = W_{n_j}$ and $W_j^{(2)} = W_{m_j}$, and let $W_j$ be an $\epsilon$-intersection of $W_j^{(1)}$ and $W_j^{(2)}$. Then by Claim \ref{claimintersection}, the sequence $(W_j)_1^\infty$ satisfies \eqref{limit}. But then the minimality of $\ell$ implies that $\dim(W_j) = \ell$ for all sufficiently large $j$. For such $j$,
\[
\dist_\Grass(V_{n_j},V_{m_j}) \leq \dist_\subset(W_j,W_j^{(1)}) + \dist_\subset(W_j,W_j^{(2)}) \leq 2\epsilon,
\]
a contradiction.
\end{proof}
So write $V_n\to V$ for some $V\in\Grass_\ell(\LL)$. Note that the sequence $(V)_1^\infty$ also satisfies \eqref{limit}, i.e.
\begin{equation}
\label{limit2}
\exists X\in\Grass(\LL) \all^{\mu_\Proj} \zz\in U_\Proj\butnot X \;\; T_\Proj(\zz) \subset \R\zz + V.
\end{equation}
{\bf Step 5. Finishing the proof.}
We now break up into cases:

{\bf Case 1:} $\mu_\Proj(U_\Proj\butnot X) = 0$. In this case, \eqref{limit2} becomes trivial but we can get the conclusion anyway. Namely, translating back to Euclidean space we get
\[
\mu(U\butnot X_\Euc) = 0
\]
for some finite-dimensional generalized sphere $X_\Euc$. After increasing the dimension of $X_\Euc$ by one, we can assume that $X_\Euc$ is a linear subspace of $\HH$. But then the theorem follows from applying Theorem \ref{theoremrigidity}($\dimH < \infty$) to
\begin{align*}
\w\HH &= X_\Euc,&
\w\mu &= \mu\given_{U\cap X_\Euc},&
\w U &= U\cap X_\Euc.
\end{align*}
{\bf Case 2.} $\mu_\Proj(U_\Proj\butnot X) > 0$, $\ell < k + 1$. In this case, for $\mu_\Proj$-a.e. $\zz\in U_\Proj\butnot X$, we have
\begin{align*}
T_\Proj(\zz) &\subset \R\zz + V\\
\dim(T_\Proj(\zz)) = k + 1 &\geq \ell + 1 = \dim(\R\zz + V),\\
T_\Proj(\zz) &= \R\zz + V,\\
V &\subset T_\Proj(\zz) \subset \zz^\perp,\\
\zz &\in V^\perp.
\end{align*}
Hence $\SK_\Proj\cap U_\Proj\butnot X\subset V^\perp$, whence it follows (e.g. from Lemma \ref{lemmapseudolinear}) that $T_\Proj(\zz)\subset V^\perp$ for $\mu_\Proj$-a.e. $\zz\in\SK_\Proj\cap U_\Proj\butnot X$. But then from the above we have $V\subset V^\perp$, i.e. $V$ is $\QQ$-isotropic. Since $\dim(V) = \ell = k \geq 1$, this implies that $V^\perp\cap L_\QQ\subset V$, so $\SK_\Proj\cap U_\Proj\butnot X\subset V$. This contradicts the hypothesis that $\mu_\Proj(U_\Proj\butnot X) > 0$, since $\mu_\Proj = \scrH^{k + 1}\given_{\SK_\Proj}$.

{\bf Case 3.} $\mu_\Proj(U_\Proj\butnot X) > 0$, $\ell = k + 1$. Note that this is the last case, since $\ell\leq k + 1$ because our original sequence $(V_n)_1^\infty$ was of dimension $k + 1$. In fact, this means that when $\ell = k + 1$, then we did not need to extract a subsequence, and $(V_n)_1^\infty$ is our original sequence defined by \eqref{Vndef}. But each element in this sequence satisfied $\#([V_n\cap L_\QQ]) = 1$, so the limiting subspace $V\in\Grass_{k + 1}(\LL)$ must also satisfy $\#([V\cap L_\QQ]) = 1$. Also, $X = \emptyset$, since the sequence defined by \eqref{Vndef} did not require an $X$.

After conjugating, we can without loss of generality suppose that $V = \lb 0\rb\times\R\times L_0$ for some $L_0\in\Grass_k(\HH)$. Then for $\mu$-a.e. $\xx\in U$,
\[
\R\iota(\xx) + V \supset T_\Proj(\iota(\xx)) = \R\iota(\xx) + \iota'(\xx)[T_\Euc(\xx)],
\]
which implies that $T_\Euc \equiv L_0$ on $U$. On the other hand, by the hypothesis of Theorem \ref{theoremrigidity}, we have $\TD(\SK\cap U) = k$. By the sum theorem for topological dimension \cite[Theorem 1.5.3]{Engelking}, we have $\TD(K) = k$ for some compact $K\subset \SK\cap U\butnot\{\infty\}$. So to finish the proof, we need the following lemma:

\begin{lemma}
\label{lemmaTPconstant}
Let $K\subset\HH$ be a compact pseudorectifiable set with tangent plane function $T_K\equiv L_0$ such that $\TD(K) = k$. Then there exists an affine $k$-plane $\zz + L_0$ such that $\scrH^k(K\cap (\zz + L_0)) > 0$.
\end{lemma}
The proof of this lemma will be given in the next section. Once we have it, then the argument of Section \ref{sectionk1} finishes the proof, since $\SK$ contains a rectifiable set. (Cf. Remark \ref{remarkpatch1}).

\begin{remark}
To prove Theorem \ref{theorem1} for the case of weakly discrete Kleinian groups (cf. \cite[Definition 5.2.1]{DSU} for the definition of weak discreteness), Lemma \ref{lemmaTPconstant} is not necessary. Indeed, by Claim \ref{claimintersection}, the intersection of any two subspaces satisfying \eqref{limit2} also satisfies \eqref{limit2}, so the minimality of $\ell$ implies that $V$ is the only $(k + 1)$-dimensional subspace satisfying \eqref{limit2}. As in Step 1 let $\scrG = \{g\in G : g(U)\subset U\}$. Then for all $g\in \scrG$, $g_\Proj^{-1}[V]$ also satisfies the criterion \eqref{limit2}, so $g_\Proj[V] = V$. In particular, $g_\Proj([V\cap L_\QQ]) = [V\cap L_\QQ]$, i.e. $g(\infty) = \infty$. Thus $\scrG$ consists of similarities. But a weakly discrete Kleinian group cannot contain two similarities with distinct fixed points \cite[Proposition 6.4.1]{DSU}, so $\SK\cap U\butnot\{\infty\}$ is a singleton, a contradiction.

On the other hand, the need for an additional lemma to complete the proof of Theorem \ref{theorem2} is demonstrated by Example \ref{examplemuR1}(ii,iv), an example of a compact pseudorectifiable IFS limit set $\LS$ satisfying $T_\LS \equiv L_0$. This example shows that the assumption $\TD(K) = k$ in Lemma \ref{lemmaTPconstant} is necessary (and cannot be replaced by the assumption that $K$ is the limit set of an IFS).
\end{remark}

\draftnewpage
\section{Proof of Lemma \ref{lemmaTPconstant} (Eliminating the exceptional case)}
\label{sectionlemmaTPconstant}

We will need the following characterization of topological dimension for compact metric spaces:

\begin{lemma}
\label{lemmatopology2}
For a compact metric space $X$, the following are equivalent:
\begin{itemize}
\item[(A)] $\TD(X) \geq k$.
\item[(B)] There exist closed sets $F_1,\ldots,F_k\subset X$ and open sets $U_1,\ldots,U_k\subset X$ such that for each $i$, $F_i\subset U_i$, and such that if $F_i\subset W_i\subset U_i$ for some open sets $W_1,\ldots,W_k$, then $\bigcap_1^k \del W_i \neq \emptyset$.
\item[(C)] There exist closed sets $F_1,\ldots,F_k\subset X$ and $F_1',\ldots,F_k'\subset X$ such that for each $i$, $F_i\cap F_i' = \emptyset$, and such that if $G_i\supset F_i$ and $G_i'\supset F_i'$ are closed sets such that $G_i\cap G_i' = \emptyset$, then $\bigcup_1^k (G_i\cup G_i') \propersubset X$.
\end{itemize}
\end{lemma}
The equivalence of (A) and (B) is proven in \cite[Corollary of Theorem II.8]{Nagata} (cf. \cite[Theorem 1.7.9]{Engelking}).
\begin{proof}[Proof of \text{(B) \implies (C)}]
Let $F_i' = X\butnot U_i$ for each $i$, and let $G_i\supset F_i$ and $G_i'\supset F_i'$ be closed sets such that $G_i\cap G_i' = \emptyset$. For each $i$, let $W_i$ be an open set which contains $G_i$ such that $\cl{W_i}\cap G_i' = \emptyset$. By definition, $\bigcap_1^k \del W_i \neq\emptyset$. But by construction, $\bigcap_1^k \del W_i \subset X\butnot \bigcup_1^k (G_i\cup G_i')$, so $\bigcup_1^k (G_i\cup G_i') \propersubset X$.
\end{proof}
\begin{proof}[Proof of \text{(C) \implies (B)}]
Let $U_i = X\butnot F_i'$ for each $i$, and let $W_1,\ldots,W_k$ be open sets such that $F_i\subset W_i\subset U_i$. By contradiction suppose that $\bigcap_1^k \del W_i = \emptyset$. Let $N_1,\ldots,N_k$ be open neighborhoods of $\del W_1,\ldots,\del W_k$ such that $\bigcap_1^k N_i = \emptyset$. Letting $G_i = F_i\cup (W_i\butnot N_i)$ and $G_i' = X\butnot W_i$ finishes the proof.
\end{proof}

Now we begin the proof of Lemma \ref{lemmaTPconstant}. Let $F_1,\ldots,F_k\subset K$ and $F_1',\ldots,F_k'\subset K$ be as in Lemma \ref{lemmatopology2}. As in the proof of Proposition \ref{propositionfederermattilainfdim}, we let $V_1\in\Grass(\HH)$ be a subspace large enough so that if $\pi_1 = \pi_{V_1}$, then for all $i$, $\pi_1(F_i)\cap\pi_1(F_i') = \emptyset$. For each $i = 1,\ldots,k$, we let $g_i:V_1\to\R$ be as in the smooth version of Urysohn's lemma, so that $g_i = -1$ on $F_i$ and $g_i = 1$ on $F_i'$. Let $\gg = (g_1,\ldots,g_k):\HH\to\R^k$, and let $\hh = \gg\circ\pi_1:\HH\to [-1,1]^k$.
\begin{claim}
\label{claimK4}
Fix $V_2\in\Grass(L_0^\perp)$, let $\pi_2 = \pi_{V_2}$, and let $\ff:V_1\to\R^{k - 1}$ be a $\CC^1$ map. Then
\[
\scrH^k\big([(\ff\circ\pi_1)\oplus \pi_2](K)\big) = 0.
\]
In particular, by Theorem \ref{theoremsz}
\[
\TD\big([(\ff\circ\pi_1)\oplus \pi_2](K)\big) < k.
\]
\end{claim}
\begin{proof}
Let $V_3 = L_0 + V_1 + V_2$, and let $\pi_3 = \pi_{V_3}$. Then by Lemma \ref{lemmapseudolinear}, $K_3 := \pi_3(K)$ is pseudorectifiable with tangent plane function $T_{K_3} \equiv L_0$. By Proposition \ref{propositionfindimconverse}, $K_3$ is rectifiable, so \eqref{rectifiable} is valid for sets $A\subset K_3$. But
\[
[(\ff\circ\pi_1)\oplus \pi_2](K) = [(\ff\circ\pi_1) \oplus \pi_2](K_3),
\]
so
\[
\scrH^k\big([(\ff\circ\pi_1)\oplus \pi_2](K)\big) \leq \int_{K_3} \det\big((\ff'(\pi_1(\xx))\circ\pi_1)\oplus \pi_2\given L_0\big) \;\dee\scrH^k(\xx) = 0,
\]
where the last equality is because for each $\xx\in K_3$,
\begin{align*}
\rank\big((\ff'(\pi_1(\xx))\circ\pi_1)\oplus \pi_2\given L_0\big)
&= \dim\big([(\ff'(\pi_1(\xx))\circ\pi_1)\oplus \0](L_0)\big)\\
&\leq \dim(\R^{k - 1}) = k - 1 < k = \dim(L_0).
\qedhere\end{align*}
\end{proof}

\begin{claim}
The set
\begin{equation}
\label{bigintersection}
\bigcap_{\yy\in (-1,1)^k} \pi_{L_0^\perp}(\hh^{-1}(\yy)\cap K)
\end{equation}
is nonempty.
\end{claim}
\begin{proof}
By contradiction suppose not. Then by the intersection property of compact sets, there exists a finite set $F\subset (-1,1)^k$ such that
\begin{equation}
\label{contradictionhypothesis}
\bigcap_{\yy\in F} \pi_{L_0^\perp}(\hh^{-1}(\yy)\cap K) = \emptyset.
\end{equation}
By replacing $\hh$ with $\Phi\circ\hh$, where $\Phi:[-1,1]^k\to [-1,1]^k$ is a diffeomorphism such that $\Phi\given\del [-1,1]^k = \id$ and $\Phi(F) \subset \R\ee_k$, we may without loss of generality assume that $F\subset\R\ee_k$. Write $F = \{t_1\ee_k,\ldots,t_r\ee_k\}$, where $-1 < t_1 < \ldots < t_r < 1$. For each $p = 1,\ldots,r$ let
\begin{align*}
A_p &= \pi_{L_0^\perp}(\hh^{-1}(t_p\ee_k)\cap K),
\end{align*}
so that $\bigcap_{p = 1}^r A_p = \emptyset$. Then there exists $V_2\in\Grass(L_0^\perp)$ such that $\bigcap_{p = 1}^r \pi_2(A_p) = \emptyset$, where $\pi_2 = \pi_{V_2}$. Let $\ff = (\pi_*\circ\hh)\oplus\pi_2$, where $\pi_*:\R^k\to\R^{k - 1}$ is the natural projection sending $\ee_k$ to $\0$; then $\bigcap_{p = 1}^r \ff(A_p) = \emptyset$. Let $B_1,\ldots,B_r$ be closed neighborhoods of $\ff(A_1),\ldots,\ff(A_r)$ such that $\bigcap_{p = 1}^r B_p = \emptyset$. Then since $K$ is compact, there exist neighborhoods $U_1,\ldots,U_r$ of $t_1\ee_k,\ldots,t_r\ee_k$ such that for all $p$,
\[
\ff(\hh^{-1}(U_p)\cap K)\subset B_p.
\]
Let $\delta > 0$ be small enough such that for all $p = 1,\ldots,r$,
\[
t_p\ee_k + [-\delta,\delta]^k \subset U_p.
\]
Let $K' = \ff(K)$; then by Claim \ref{claimK4}, $\TD(K') < k$. We proceed to use this fact to recursively construct a sequence of $2k$-tuples $((F_{i,p},F_{i,p}')_{i = 1}^k)_{p = 0}^r$ of closed subsets of $K'$. The first $2k$-tuple is defined by the equations
\begin{equations}
F_{i,0} &= \{(\yy,\zz)\in K' : y_i \leq -\delta\} \;\; (i = 1,\ldots,k - 1),&
F_{k,0} &= K',\\
F_{i,0}' &= \{(\yy,\zz)\in K' : y_i \geq \delta\}, \;\; (i = 1,\ldots,k - 1),&
F_{k,0}' &= \bigcap_{p = 1}^r B_p = \emptyset.
\end{equations}
Now suppose that the $2k$-tuple $(F_{i,p},F_{i,p}')_{i = 1}^k$ satisfies $F_{i,p}\cap F_{i,p}' = \emptyset$ for all $i = 1,\ldots,k$. Since $\TD(K') < k$, by Lemma \ref{lemmatopology2} there exist closed sets $G_{i,p}\supset F_{i,p}$ and $G_{i,p}'\supset F_{i,p}'$ such that $G_{i,p}\cap G_{i,p}' = \emptyset$ and
\[
\bigcup_{i = 1}^k (G_{i,p}\cup G_{i,p}') = K'.
\]
Then we let
\begin{equations}
F_{i,p + 1} &= G_{i,p} \;\; (i = 1,\ldots,k - 1),&
F_{k,p + 1} &= G_{k,p}\cap B_p,\\
F_{i,p + 1}' &= G_{i,p}' \;\; (i = 1,\ldots,k - 1),&
F_{k,p + 1}' &= G_{k,p}'\cup\bigcap_{q > p + 1} B_q.
\end{equations}
We claim that $F_{i,p + 1}\cap F_{i,p + 1}' = \emptyset$ for all $i$. This is obvious for $i < k$, and
\[
F_{k,p + 1}\cap F_{k,p + 1}' \subset G_{k,p}\cap (G_{k,p}' \cup \bigcap_{q > p} B_q) \subset G_{k,p}\cap (G_{k,p}' \cup F_{k,p}') = \emptyset.
\]
This guarantees that the recursion can continue.

Note that for $i = 1,\ldots,k - 1$, we have $F_{i,0}\subset F_{i,r}$ and $F_{i,0}'\subset F_{i,r}'$. Also note that $F_{k,0} = F_{k,r}' = K'$, which implies that $G_{k,0}' = G_{k,r} = \emptyset$.

Now we define the sets $G_i,G_i'\subset K$ ($i = 1,\ldots,k$) to get a contradiction to the fact that $\TD(K) = k$ using Lemma \ref{lemmatopology2}. For each $p = 0,\ldots,r$ let
\begin{equations}
T_p &= \hh^{-1}\big([-\delta,\delta]^{k - 1}\times [t_p,t_{p + 1}]\big)\cap K,
\end{equations}
with the convention that $t_0 = -1$ and $t_{r + 1} = 1$. Then let
\begin{equations}
G_i &= \ff^{-1}(G_{i,r})\cap K \;\; (i = 1,\ldots,k - 1),&
G_k &= F_k\cup\bigcup_{p = 0}^r (T_p \cap \ff^{-1}(G_{k,p})),\\
G_i' &= \ff^{-1}(G_{i,r}')\cap K \;\; (i = 1,\ldots,k - 1),&
G_k' &= F_k'\cup\bigcup_{p = 0}^r (T_p\cap \ff^{-1}(G_{k,p}')).
\end{equations}
Then clearly, $G_i\supset F_i$ and $G_i'\supset F_i'$ for all $i$. To get a contradiction, we must demonstrate the following relations:
\begin{align} \label{contra1}
G_i\cap G_i' &= \emptyset \;\; (i = 1,\ldots,k),\\ \label{contra2}
\bigcup_{i = 1}^k (G_i\cup G_i') &= K.
\end{align}
Now, \eqref{contra1} is obviously satisfied for $i = 1,\ldots,k - 1$. By contradiction, suppose that $\xx\in G_k\cap G_k'$. First suppose that $\xx\in F_k$, i.e. $h_k(\xx) = -1$. Since $F_k\cap F_k' = \emptyset$, we have $\xx\in T_p\cap \ff^{-1}(G_{k,p}')$ for some $p = 0,\ldots,r$. Since $\xx\in T_p$, we have $p = 0$, so $\ff(\xx)\in G_{k,0}'$. But this contradicts that $G_{k,0}' = \emptyset$, as we observed above. A similar argument rules out the case $\xx\in F_k'$.

Alternatively, suppose that $\xx\in T_p\cap \ff^{-1}(G_{k,p})\cap T_q\cap \ff^{-1}(G_{k,q}')$ for some $p,q = 0,\ldots,r$. Since $G_{k,p}\cap G_{k,p}' = \emptyset$, we get $p\neq q$, and since $T_p\cap T_q \neq \emptyset$, we get $|p - q| = 1$. Without loss of generality suppose that $q = p + 1$.  Then $\hh(\xx)\in T_p\cap T_q \subset U_p$ and thus $\ff(\xx)\in B_p$. Since
\begin{align*}
G_{k,p} \cap B_p &\subset G_{k,p + 1},&
G_{k,p}' &\subset G_{k,p + 1}',
\end{align*}
it follows that $\ff(\xx)$ is in at most one of the sets $G_{k,p}\cup G_{k,p + 1}$, $G_{k,p}'\cup G_{k,p + 1}'$. This is again a contradiction, so we get \eqref{contra1}.

Similarly, fix $\xx\in K$. If $\hh(\xx)\notin [-\delta,\delta]^{k - 1}\times[-1,1]^k$, then there exists $i = 1,\ldots,k - 1$ such that $\hh(\xx)\in F_{i,0}\cup F_{i,0}'$ and thus $\xx\in G_i\cup G_i'$. So suppose $\hh(\xx)\in [-\delta,\delta]^{k - 1}\times[-1,1]^k$; then $\xx\in T_p$ for some $p = 0,\ldots,r$. Thus
\[
\xx\in T_p\cap \ff^{-1}\left(\bigcup_{i = 1}^k (G_{i,p}\cup G_{i,p}')\right) \subset \bigcup_{i = 1}^k (G_i\cup G_i').
\]
So \eqref{contra2} holds, contradicting that $\TD(K) = k$.
\end{proof}

Let $\zz$ be a member of \eqref{bigintersection}. Then for all $\yy\in (-1,1)^k$, we have $\hh^{-1}(\yy)\cap K\cap (\zz + L_0)\neq\emptyset$. Equivalently,
\[
(-1,1)^k \subset \hh(K\cap (\zz + L_0)).
\]
So by Fact \ref{factlipschitz}, $\scrH^k(K\cap (\zz + L_0)) > 0$. This completes the proof of Lemma \ref{lemmaTPconstant}.
%
%
%

\appendix
\section{An example regarding pseudorectifiability}
The following example conveniently combines counterexamples to several natural hypotheses about pseudorectifiable sets:

\begin{example}
\label{examplemuR1}
There exist sets $R_1,R_2\subset\HH$ such that
\begin{itemize}
\item[(i)] Both $R_1$ and $R_2$ satisfy $\eqref{pseudorectifiable}_{k = 1}$ for some pairs $(\mu_1,T_1)$ and $(\mu_2,T_2)$.
\item[(ii)] $\mu_1\asymp\scrH^1\given_{R_1}$, but $\mu_1\neq\scrH^1\given_{R_1}$. Thus $R_1$ is pseudorectifiable but not rectifiable.
\item[(iii)] $\scrH^1(R_2) < \infty$, but $\scrH^1\given_{R_2}\not\lessless \mu_2$. Thus $R_2$ is not pseudorectifiable despite satisfying \eqref{pseudorectifiable}.
\item[(iv)] Both $R_1$ and $R_2$ are limit sets of similarity IFSes satisfying the strong separation condition.
\item[(v)] There exists a bounded linear map $L:\HH\to\HH$ such that $R_2 = L[R_1]$.
\end{itemize}
\end{example}
Of course, $k = 1$ is just for convenience and the examples can easily be extended to higher dimensions.
\NPC{Proof}
\begin{proof}
Let $\HH_* = \ell^2(\N^2)$ and $\HH = \R\oplus\HH_*$, and let $\pi_* = \pi_{\HH_*}$. Define the function $f:[0,1]\to\HH_*$ as follows:
\[
f(t) = \sum_{k = 1}^\infty \frac{1}{2^k} \ee_{k,\lfloor 2^k t\rfloor}.
\]
Fix $\alpha > \sqrt 2$ and let $F(t) = (t,\alpha f(t))$. We may now define the sets $R_1,R_2\subset\HH$ and their corresponding pairs $(\mu_1,T_1)$ and $(\mu_2,T_2)$ as follows:
\begin{align*}
R_1 &= \cl{F([0,1])},& \mu_1 &= F[\lambda],& T_1 &\equiv L_0 := \R\times\lb\0\rb\\
R_2 &= \cl{\alpha f([0,1])},& \mu_2 &= 0,& T_2 &\equiv 0.
\end{align*}
To demonstrate that \eqref{pseudorectifiable} holds, we first prove the following result:

\begin{claim}
\label{claimpiecewiseepsilon}
Fix $V\in\Grass(\HH_*)$ and $\epsilon > 0$. Then there exists $K = K_V\subset [0,1]$ such that $\pi_V\circ f\given_K$ is piecewise $\epsilon$-Lipschitz, and $\lambda([0,1]\butnot K) \leq \epsilon$.
\end{claim}
In what follows we will use the notation
\[
I(k,i) = \left(\frac{i}{2^k},\frac{i + 1}{2^k}\right).
\]
\begin{subproof}
We recall the trace formula for $\dim(V)$:
\[
\infty > \dim(V) = \Tr[\pi_V] = \Tr[\pi_V^2] = \sum_{k = 1}^\infty \sum_{i = 0}^{2^k - 1} \|\pi_V(\ee_{k,i})\|^2.
\]
In particular, for each $k$,
\[
\#\left\{i = 0,\ldots,2^k - 1 : \|\pi_V(\ee_{k,i})\|^2 \geq \frac{1}{2^{3k/4}}\right\} \leq 2^{3k/4} \dim(V).
\]
Letting
\begin{align*}
B_k &= \left\{t\in[0,1] : \|\pi_V(\ee_{k,\lfloor 2^k t\rfloor})\|^2 \geq 2^{-3k/4}\right\}\\
\w B_k &= \left\{t\in[0,1] : \dist(2^k t,\Z) \leq 2^{-k/4}\right\}
\end{align*}
we have
\[
\lambda(B_k) \leq 2^{-k/4}\dim(V), \;\; \lambda(\w B_k) \leq 2\cdot 2^{-k/4}.
\]
Letting $K = [0,1]\butnot \bigcup_{k\geq k_0} (B_k\cup\w B_k)$ for some sufficiently large $k_0$, we get $\lambda([0,1]\butnot K) \leq \epsilon$.

We claim that for each interval of the form $I = I(k_0,i)$, $\pi_V\circ f\given_{K\cap I}$ is $\epsilon$-Lipschitz. Indeed, fix distinct $t_1,t_2\in K\cap I$, and write
\begin{align*}
\ell_k^{(i)} &= \lfloor 2^k t_i\rfloor \;\; (i = 1,2, \; k\in\N)\\
k_1 &= \min\{k : \ell_k^{(1)}\neq \ell_k^{(2)}\} > k_0.
\end{align*}
Since $t_1,t_2\notin \w B_{k_1}$, we have
\[
|t_2 - t_1| \geq 2^{-5k_1/4}.
\]
On the other hand, since $t_1,t_2\notin B_k$ for all $k\geq k_1$, we have
\begin{align*}
&\|\pi_V\circ f(t_2) - \pi_V\circ f(t_1)\|
= \left\|\sum_{k = k_1}^\infty \frac{1}{2^k} \pi_V(\ee_{k,\ell_k^{(2)}} - \ee_{k,\ell_k^{(1)}})\right\|\\
&\leq \sum_{k = k_1}^\infty \frac{2}{2^k} \max_{i = 1}^2 \|\pi_V(\ee_{k,\ell_k^{(i)}})\|
\leq \sum_{k = k_1}^\infty \frac{2}{2^k} \sqrt{2^{-3k/4}}
\leq \frac{\epsilon}{2^{5k_1/4}}
\leq \epsilon|t_2 - t_1|,
\end{align*}
where the second-to-last inequality holds assuming $k_0$ is sufficiently large.
\end{subproof}
Next, note that for each $k\in\N$ and $i = 0,\ldots,2^k - 1$,
\[
\diam\big(F\big(I(k,i)\big)\big) \leq \sum_{\ell = k + 1}^\infty \frac{1 + \alpha}{2^\ell} \asymp_\times \frac{1}{2^k} = \lambda\big(I(k,i)\big) = \mu_1\big(F\big(I(k,i)\big)\big).
\]
So by Theorem \ref{theoremRTT}, $\scrH^1\given_{R_1}\leq C\mu_1$ for some constant $C > 0$.

To show that \eqref{pseudorectifiable} holds, fix $V\in\Grass(\HH)$. Then $V\subset W := \R\oplus U$ for some $U\in\Grass(\HH_*)$. Fix $\epsilon > 0$, and let $K = K_U\subset[0,1]$ be as in Claim \ref{claimpiecewiseepsilon}. Since $\pi_U\circ f\given_K$ is piecewise $\epsilon$-Lipschitz, $\pi_W\circ F\given_K$ is piecewise $(1 + \alpha\epsilon)$-Lipschitz, so by Fact \ref{factlipschitz},
\begin{align*}
\scrH^1(\pi_W(R_1)) &= \scrH^1(\pi_W\circ F([0,1])) \leq \scrH^1(\pi_W\circ F(K)) + \scrH^1\big(F([0,1]\butnot K)\big)\\
&\leq (1 + \alpha\epsilon)\scrH^1(K) + C\mu_1\big(F([0,1]\butnot K)\big) \leq (1 + \alpha\epsilon) + C\epsilon.
\end{align*}
Similarly,
\begin{align*}
\scrH^1(\pi_W(R_2)) &= \alpha\scrH^1(\pi_W\circ f([0,1])) \leq \alpha\scrH^1(\pi_W\circ f(K)) + \alpha\scrH^1\big(F([0,1]\butnot K)\big)\\
&\leq \alpha\epsilon\scrH^1(K) + C\mu_1\big(F([0,1]\butnot K)\big) \leq \alpha\epsilon + C\epsilon.
\end{align*}
Since $\epsilon$ was arbitrary, $\scrH^1(\pi_W(R_1)) \leq 1$ and $\scrH^1(\pi_W(R_2)) = 0$. In particular, \eqref{pseudorectifiable} holds for $R_2$.

Since $\scrH^k\big(\pi_{L_0}\circ\pi_W(R_1)\big) = \scrH^k([0,1]) = 1$, another application of Fact \ref{factlipschitz} shows that $\scrH^1(\pi_W(R_1)) = 1$. On the other hand, Claim \ref{claimpiecewiseepsilon} implies that $\pi_W(R_1) = \pi_W\circ F([0,1])$ is rectifiable, and thus pseudorectifiable by Proposition \ref{propositionrectifiable}. Applying \eqref{pseudorectifiable} with $V = L_0$ shows that the tangent plane function of $\pi_W(R_1)$ is identically equal to $L_0$. Thus for all $A\subset R_1$
\begin{align*}
\int \#(\pi_V^{-1}(\zz)\cap A) \; \dee\scrH^1\given_V(\zz) &= \int \det(\pi_V\given L_0) \#(\pi_W^{-1}(\yy)\cap A) \;\dee\scrH^1\given_W(\yy)\\
&= \int_A \det(\pi_V\given T_1(\xx)) \;\dee\mu_1(\xx),
\end{align*}
i.e. \eqref{pseudorectifiable} holds for $R_1$.

Having proven (i), we proceed to (ii)-(v). For (v), we observe that $R_2 = \pi_*(R_1)$. To see (iv), note that $R_1$ and $R_2$ are the limit sets of the iterated function systems
\begin{align*}
u_a^{(1)}(\xx) &= \frac 12(a\ee_0 + \alpha\ee_{1,a}) + \frac12\left(x_0\ee_0 + \sum_{k = 1}^\infty \sum_{i = 0}^{2^k - 1} x_{k,i} \ee_{k + 1,2^k a + i}\right) \;\; (a = 0,1)
\end{align*}
\begin{align*}
u_a^{(2)}(\xx) &= \frac \alpha2\ee_{1,a} + \frac12\sum_{k = 1}^\infty \sum_{i = 0}^{2^k - 1} x_{k,i} \ee_{k + 1,2^k a + i} \;\; (a = 0,1)
\end{align*}
respectively. The strong separation condition is proven by observing that if $W = (\R\ee_{1,0} + \R\ee_{1,1})^\perp$, then $u_a^{(i)}(R_i) \subset \frac \alpha2\ee_{1,a} + W$.

Next, since both these IFSes have Bowen parameter $\delta = 1$, by \cite[Theorems 11.2 and 11.3]{MSU}, the Hausdorff 1-measures $\scrH^1(R_i)$ are positive and finite. This demonstrates (iii). In fact, an elementary calculation of the Hausdorff 1-measure using the mass distribution principle shows that $\scrH^1(R_1) \geq \alpha/\sqrt 2 > 1 = \mu_1(R_1)$, so $\mu_1\neq\scrH^1\given_{R_1}$. Since we already showed that $\scrH^1\given_{R_1}\lessless\mu_1$, we get that $R_1$ is pseudorectifiable, but by Proposition \ref{propositionrectifiable}, $R_1$ is not rectifiable. This demonstrates (ii) and completes the proof.
\end{proof}

\bibliographystyle{amsplain}

\bibliography{bibliography}

\providecommand{\bysame}{\leavevmode\hbox to3em{\hrulefill}\thinspace}
\providecommand{\MR}{\relax\ifhmode\unskip\space\fi MR }
\providecommand{\MRhref}[2]{%
  \href{http://www.ams.org/mathscinet-getitem?mr=#1}{#2}
}
\providecommand{\href}[2]{#2}
\begin{thebibliography}{10}

\bibitem{Ahlfors}
L.~V. Ahlfors, \emph{{M}\"obius transformations in several dimensions}, Ordway
  Professorship Lectures in Mathematics, University of Minnesota, School of
  Mathematics, Minneapolis, Minn., 1981.

\bibitem{AkivisGoldberg}
M.~A. Akivis and V.~V. Goldberg, \emph{Conformal differential geometry and its
  generalizations}, Pure and Applied Mathematics (New York), John Wiley \&
  Sons, Inc., New York, 1996.

\bibitem{AmbrosioKirchheim}
L.~Ambrosio and B.~Kirchheim, \emph{Rectifiable sets in metric and {B}anach
  spaces}, Math. Ann. \textbf{318} (2000), no.~3, 527--555.

\bibitem{Antoine}
L.~Antoine, \emph{Sur l'hom\'eomorphisme de deux figures et de leurs
  voisinages}, J. Math. Pures Appl. \textbf{4} (1921), 221--325 (French).

\bibitem{AstalaZinsmeister}
K.~Astala and M.~Zinsmeister, \emph{Mostow rigidity and {F}uchsian groups}, C.
  R. Acad. Sci. Paris S\'er. I Math. \textbf{311} (1990), no.~6, 301--306.

\bibitem{BKZ}
K.~Bara\'nski, B.~Karpi\'nska, and A.~Zdunik, \emph{Bowen's formula for
  meromorphic functions}, Ergodic Theory Dynam. Systems \textbf{32} (2012), no.
  4, 1165--1189.

\bibitem{BergweilerEremenko}
W.~Bergweiler and A.~E. Eremenko, \emph{Meromorphic functions with linearly
  distributed values and {J}ulia sets of rational functions}, Proc. Amer. Math.
  Soc. \textbf{137} (2009), no. 7, 2329--2333.

\bibitem{BestvinaCooper}
M.~Bestvina and D.~Cooper, \emph{A wild {C}antor set as the limit set of a
  conformal group action on {$S^3$}}, Proc. Amer. Math. Soc. \textbf{99}
  (1987), no.~4, 623--626.

\bibitem{Bishop3}
C.~J. Bishop, \emph{Divergence groups have the {B}owen property}, Ann. of Math.
  (2) \textbf{154} (2001), no.~1, 205--217.

\bibitem{Bishop2}
\bysame, \emph{Non-rectifiable limit sets of dimension one}, Rev. Mat.
  Iberoamericana \textbf{18} (2002), no.~3, 653--684.

\bibitem{Bishop_non_rigidity}
\bysame, \emph{A transcendental {J}ulia set of dimension 1},
  \url{http://www.math.sunysb.edu/~bishop/papers/dim=1.pdf}, 2012.

\bibitem{BishopJones}
C.~J. Bishop and P.~W. Jones, \emph{{H}ausdorff dimension and {K}leinian
  groups}, Acta Math. \textbf{179} (1997), no. 1, 1--39.

\bibitem{Blankinship}
W.~A. Blankinship, \emph{Generalization of a construction of {A}ntoine}, Ann.
  of Math. (2) \textbf{53} (1951), 276--297.

\bibitem{Bothe}
H.~G. Bothe, \emph{Ein eindimensionales {K}ompaktum im {$E^3$}, das sich nicht
  lagetreu in die {M}engersche {U}niversalkurve einbetten l{\"a}sst}, Fund.
  Math. \textbf{54} (1964), 251--258 (German).

\bibitem{Bowditch_geometrical_finiteness}
B.~H. Bowditch, \emph{Geometrical finiteness for hyperbolic groups}, J. Funct.
  Anal. \textbf{113} (1993), no. 2, 245--317.

\bibitem{Bowen}
R.~Bowen, \emph{Hausdorff dimension of quasicircles}, Inst. Hautes \'Etudes
  Sci. Publ. Math. \textbf{50} (1979), 11--25.

\bibitem{Braam}
P.~J. Braam, \emph{A {K}aluza-{K}lein approach to hyperbolic three-manifolds},
  Enseign. Math. (2) \textbf{34} (1988), no.~3-4, 275--311.

\bibitem{CanaryTaylor}
R.~D. Canary and E.~C. Taylor, \emph{Kleinian groups with small limit sets},
  Duke Math. J. \textbf{73} (1994), no.~2, 371--381.

\bibitem{CarlesonGamelin}
L.~Carleson and T.~W. Gamelin, \emph{Complex dynamics}, Universitext: Tracts in
  Mathematics, Springer-Verlag, New York, 1993.

\bibitem{Coornaert}
M.~Coornaert, \emph{Mesures de {P}atterson-{S}ullivan sur le bord d'un espace
  hyperbolique au sens de {G}romov ({P}atterson-{S}ullivan measures on the
  boundary of a hyperbolic space in the sense of {G}romov)}, Pacific J. Math.
  \textbf{159} (1993), no. 2, 241--270 (French).

\bibitem{DSU}
Tushar Das, David Simmons, and Mariusz Urba{\'n}ski, \emph{Geometry and
  dynamics in {G}romov hyperbolic metric spaces: with an emphasis on non-proper
  settings}, \url{http://arxiv.org/abs/1409.2155}, preprint 2014, to appear in
  Math. Surveys Monogr.

\bibitem{DavermanVenema}
R.~J. Daverman and G.~A. Venema, \emph{Embeddings in manifolds}, Graduate
  Studies in Mathematics, 106, American Mathematical Society, Providence, RI,
  2009.

\bibitem{DouadyHubbard}
A.~Douady and J.~H. Hubbard, \emph{{\'E}tude dynamique des polyn{\^o}mes
  complexes. {P}artie {I}. ({D}ynamical study of complex polynomials. {P}art
  {I})}, Publications Math\'ematiques d'Orsay, 84-2, Universit\'e de Paris-Sud,
  D\'epartement de Math\'ematiques, Orsay, 1984 (French).

\bibitem{Edwards}
R.~D. Edwards, \emph{Demension theory. {I}.}, Geometric topology (Proc. Conf.,
  Park City, Utah, 1974), Lecture Notes in Math., Vol. 438, Springer, Berlin,
  1975, pp.~195--211.

\bibitem{Engelking}
R.~Engelking, \emph{Dimension theory. {T}ranslated from the {P}olish and
  revised by the author}, North-Holland Mathematical Library, 19, North-Holland
  Publishing Co., Amsterdam-Oxford-New York; PWN---Polish Scientific
  Publishers, Warsaw, 1978.

\bibitem{EremenkoVanstrien}
A.~E. Eremenko and S.~J. van Strien, \emph{Rational maps with real
  multipliers}, Trans. Amer. Math. Soc. \textbf{363} (2011), no. 12,
  6453--6463.

\bibitem{Falconer_book}
Kenneth Falconer, \emph{Fractal geometry: {M}athematical foundations and
  applications}, John Wiley \& Sons, Ltd., Chichester, 1990.

\bibitem{Fatou1}
P.~Fatou, \emph{Sur les {\'e}quations fonctionnelles}, Bull. Soc. Math. France
  \textbf{48} (1920), 208--314 (French).

\bibitem{Federer2}
H.~Federer, \emph{Dimension and measure}, Trans. Amer. Math. Soc. \textbf{62}
  (1947), 536--547.

\bibitem{Federer}
\bysame, \emph{Geometric measure theory}, Die Grundlehren der mathematischen
  Wissenschaften, Band, vol. 153, Springer-Verlag New York Inc., New York,
  1969.

\bibitem{Fricke}
R.~Fricke, \emph{Die {K}reisbogenvierseite und das {P}rincip der {S}ymmetrie},
  Math. Ann. \textbf{44} (1894), no.~4, 565--599 (German).

\bibitem{FrickeKlein}
R.~Fricke and F.~Klein, \emph{Vorlesungen \"uber die {T}heorie der automorphen
  {F}unktionen ({L}ectures on the theory of automorphic functions)},
  Bibliotheca Mathematica Teubneriana, Johnson Reprint Corp., New York; B. G.
  Teubner Verlagsgesellschaft, Stuttgart, 1965 (German).

\bibitem{Gusevskii}
N.~A. Gusevski\u\i, \emph{Strange actions of free {K}leinian groups}, Sibirsk.
  Mat. Zh. \textbf{37} (1996), no.~1, 90--107.

\bibitem{Hamilton2}
D.~H. Hamilton, \emph{Length of {J}ulia curves}, Pacific J. Math. \textbf{169}
  (1995), no.~1, 75--93.

\bibitem{Hamilton}
\bysame, \emph{Rectifiable {J}ulia curves}, J. London Math. Soc. (2)
  \textbf{54} (1996), no.~3, 530--540.

\bibitem{Hertrich}
Udo Hertrich-Jeromin, \emph{Introduction to {M}\"obius differential geometry.},
  London Mathematical Society Lecture Note Series, 300, Cambridge University
  Press, Cambridge, 2003.

\bibitem{HurewiczWallman}
W.~Hurewicz and H.~Wallman, \emph{Dimension theory}, Princeton Mathematical
  Series, v.4, Princeton University Press, Princeton, N. J., 1941.

\bibitem{Kapovich2}
M.~Kapovich, \emph{Homological dimension and critical exponent of {K}leinian
  groups}, Geom. Funct. Anal. \textbf{18} (2009), no. 6, 2017--2054.

\bibitem{KotusUrbanski2}
J.~Kotus and Mariusz Urba{\'n}ski, \emph{Fractal measures and ergodic theory of
  transcendental meromorphic functions}, Transcendental dynamics and complex
  analysis, London Math. Soc. Lecture Note Ser., 348, Cambridge Univ. Press,
  Cambridge, 2008, pp.~251--316.

\bibitem{LuukkainenVaisala}
J.~Luukkainen and J.~V\"ais\"al\"a, \emph{Elements of {L}ipschitz topology},
  Ann. Acad. Sci. Fenn. Ser. A I Math. \textbf{3} (1977), no.~1, 85--122.

\bibitem{Mandelbrot}
B.~B. Mandelbrot, \emph{Fractals and chaos: {T}he {M}andelbrot set and beyond.
  {S}electa {V}olume {C}. {W}ith a foreword by {P}. {W}. {J}ones and texts
  co-authored by {C}. {J}. {G}. {E}vertsz and {M}. {C}. {G}utzwiller}, Selected
  Works of Benoit B. Mandelbrot, Springer-Verlag, New York, 2004.

\bibitem{Matsumoto}
S.~Matsumoto, \emph{Foundations of flat conformal structure}, Aspects of
  low-dimensional manifolds, Adv. Stud. Pure Math., 20, Kinokuniya, Tokyo,
  1992, pp.~167--261.

\bibitem{Mattila}
P.~Mattila, \emph{Geometry of sets and measures in {E}uclidean spaces: Fractals
  and rectifiability}, Cambridge Studies in Advanced Mathematics, 44, Cambridge
  University Press, Cambridge, 1995.

\bibitem{MSU}
R.~Daniel Mauldin, T.~Szarek, and Mariusz Urba{\'n}ski, \emph{Graph directed
  {M}arkov systems on {H}ilbert spaces}, Math. Proc. Cambridge Philos. Soc.
  \textbf{147} (2009), 455--488.

\bibitem{MauldinUrbanski1}
R.~Daniel Mauldin and Mariusz Urba{\'n}ski, \emph{Dimensions and measures in
  infinite iterated function systems}, Proc. London Math. Soc. (3) \textbf{73}
  (1996), no. 1, 105--154.

\bibitem{MauldinUrbanski3}
\bysame, \emph{Parabolic iterated function systems}, Ergodic Theory Dynam.
  Systems \textbf{20} (2000), no. 5, 1423--1447.

\bibitem{MauldinUrbanski2}
\bysame, \emph{Graph directed {M}arkov systems: Geometry and dynamics of limit
  sets}, Cambridge Tracts in Mathematics, vol. 148, Cambridge University Press,
  Cambridge, 2003.

\bibitem{MayerUrbanski}
V.~Mayer and Mariusz Urba{\'n}ski, \emph{Finer geometric rigidity of limit sets
  of conformal {IFS}}, Proc. Amer. Math. Soc. \textbf{131} (2003), no. 12,
  3695--3702.

\bibitem{MayerUrbanski2}
\bysame, \emph{Thermodynamical formalism and multifractal analysis for
  meromorphic functions of finite order}, Mem. Amer. Math. Soc. \textbf{203}
  (2010), no. 954, vi+107 pp.

\bibitem{McMillanRow}
D.~R. McMillan{, }Jr. and W.~H. Row{, }Jr., \emph{Tangled embeddings of
  one-dimensional continua}, Proc. Amer. Math. Soc. \textbf{22} (1969),
  378--385.

\bibitem{McMullen_classification}
Curt McMullen, \emph{The classification of conformal dynamical systems},
  Current developments in mathematics, 1995 (Cambridge, MA), Int. Press,
  Cambridge, MA, 1994, pp.~323--360.

\bibitem{McMullen_renormalization}
\bysame, \emph{Renormalization and 3-manifolds which fiber over the circle},
  Annals of Mathematics Studies, 142, Princeton University Press, Princeton,
  NJ, 1996.

\bibitem{McMullen_conformal_2}
\bysame, \emph{Hausdorff dimension and conformal dynamics. {II}.
  {G}eometrically finite rational maps}, Comment. Math. Helv. \textbf{75}
  (2000), no. 4, 535--593.

\bibitem{Nagata}
J.~Nagata, \emph{Modern dimension theory. {R}evised edition.}, Sigma Series in
  Pure Mathematics, 2, Heldermann Verlag, Berlin, 1983.

\bibitem{Nevanlinna}
R.~Nevanlinna, \emph{On differentiable mappings}, Analytic Functions, Princeton
  University Press, N. J., 1960.

\bibitem{Patterson2}
S.~J. Patterson, \emph{The limit set of a {F}uchsian group}, Acta Math.
  \textbf{136} (1976), no. 3--4, 241--273.

\bibitem{Patterson3}
\bysame, \emph{Further remarks on the exponent of convergence of {P}oincar\'e
  series}, Tohoku Math. J. (2) \textbf{35} (1983), no. 3, 357--373.

\bibitem{PPS}
F.~Paulin, M.~Pollicott, and B.~Schapira, \emph{Equilibrium states in negative
  curvature}, \url{http://arxiv.org/abs/1211.6242}, preprint 2013.

\bibitem{Poincare1}
H.~Poincar\'e, \emph{M\'emoire sur les groupes klein\'eens}, Acta Math.
  \textbf{3} (1883), no.~1, 49--92 (French).

\bibitem{PoincareKlein}
H.~Poincar\'e and F.~Klein, \emph{Correspondance d'{H}enri {P}oincar\'e et de
  {F}elix {K}lein}, Acta Math. \textbf{39} (1923), no.~1, 94--132 (French).

\bibitem{Przytycki}
Feliks Przytycki, \emph{Hausdorff dimension of harmonic measure on the boundary
  of an attractive basin for a holomorphic map}, Invent. Math. \textbf{80}
  (1985), no.~1, 161--179.

\bibitem{Przytycki2}
\bysame, \emph{Conical limit set and {P}oincar{\'e} exponent for iterations of
  rational functions}, Trans. Amer. Math. Soc. \textbf{351} (1999), no.~5,
  2081--2099.

\bibitem{PrzytyckiRivera}
Feliks Przytycki and J.~Rivera-Letelier, \emph{Statistical properties of
  topological {C}ollet-{E}ckmann maps}, Ann. Sci. \'Ec. Norm. Sup\'er. (4)
  \textbf{40} (2007), no.~1, 135--178.

\bibitem{PRS2}
Feliks Przytycki, J.~Rivera-Letelier, and S.~K. Smirnov, \emph{Equality of
  pressures for rational functions}, Ergodic Theory Dynam. Systems \textbf{24}
  (2004), no.~3, 891--914.

\bibitem{PrzytyckiUrbanski}
Feliks Przytycki and Mariusz Urba{\'n}ski, \emph{Conformal fractals: ergodic
  theory methods}, London Mathematical Society Lecture Note Series, 371,
  Cambridge University Press, Cambridge, 2010.

\bibitem{Rempe}
L.~Rempe-Gillen, \emph{Hyperbolic dimension and radial {J}ulia sets of
  transcendental functions}, Proc. Amer. Math. Soc. \textbf{137} (2009), no. 4,
  1411--1420.

\bibitem{RiveraShen}
J.~Rivera-Letelier and W.~Shen, \emph{Statistical properties of one-dimensional
  maps under weak hyperbolicity assumptions}, Ann. Sci. \'Ec. Norm. Sup\'er.
  (4) \textbf{47} (2014), no.~6, 1027--1083.

\bibitem{RogersTaylor}
C.~A. Rogers and S.~J. Taylor, \emph{The analysis of additive set functions in
  {E}uclidean space}, Acta Math. \textbf{101} (1959), 273--302.

\bibitem{Rushing}
T.~B. Rushing, \emph{Hausdorff dimension of wild fractals}, Trans. Amer. Math.
  Soc. \textbf{334} (1992), no.~2, 597--613.

\bibitem{Schul}
R.~Schul, \emph{Subsets of rectifiable curves in {H}ilbert space-the analyst's
  {TSP}}, Jinpeng Anal. Math. \textbf{103} (2007), 331--375.

\bibitem{SSU}
K.~Simon, B.~Solomyak, and Mariusz Urba{\'n}ski, \emph{Hausdorff dimension of
  limit sets for parabolic {IFS} with overlaps}, Pacific J. Math. \textbf{201}
  (2001), no.~2, 441--478.

\bibitem{Stanko}
M.~A. {\v S}tan{$'$}ko, \emph{The imbedding of compacta in euclidean space},
  Dokl. Akad. Nauk SSSR \textbf{186} (1969), 1269--1272 (Russian).

\bibitem{Sullivan_density_at_infinity}
D.~P. Sullivan, \emph{The density at infinity of a discrete group of hyperbolic
  motions}, Inst. Hautes \'Etudes Sci. Publ. Math. \textbf{50} (1979),
  171--202.

\bibitem{Sullivan_discrete_conformal_groups}
\bysame, \emph{Discrete conformal groups and measurable dynamics}, Bull. Amer.
  Math. Soc. \textbf{6} (1982), no. 1, 57--73.

\bibitem{Sullivan_seminar}
\bysame, \emph{Seminar on conformal and hyperbolic geometry}, 1982, Notes by M.
  Baker and J. Seade, IHES.

\bibitem{Sullivan_conformal_dynamical}
\bysame, \emph{Conformal dynamical systems}, Geometric dynamics (Rio de
  Janeiro, 1981), Lecture Notes in Math., 1007, Springer, Berlin, 1983,
  pp.~725--752.

\bibitem{Sullivan_wandering}
\bysame, \emph{Quasiconformal homeomorphisms and dynamics. {I}. {S}olution of
  the {F}atou-{J}ulia problem on wandering domains}, Ann. of Math. (2)
  \textbf{122} (1985), no.~3, 401--418.

\bibitem{Urbanski7}
Mariusz Urba{\'n}ski, \emph{On the {H}ausdorff dimension of a {J}ulia set with
  a rationally indifferent periodic point}, Studia Math. \textbf{97} (1991),
  no.~3, 167--188.

\bibitem{Urbanski6}
\bysame, \emph{Rational functions with no recurrent critical points}, Ergodic
  Theory Dynam. Systems \textbf{14} (1994), no.~2, 391--414.

\bibitem{Urbanski4}
\bysame, \emph{Measures and dimensions in conformal dynamics}, Bull. Amer.
  Math. Soc. \textbf{40} (2003), no. 3, 281--321.

\bibitem{Vaisala3}
J.~V\"ais\"al\"a, \emph{Demension and measure}, Proc. Amer. Math. Soc.
  \textbf{76} (1979), no.~1, 167--168.

\bibitem{Zdunik}
A.~Zdunik, \emph{Parabolic orbifolds and the dimension of the maximal measure
  for rational maps}, Invent. Math. \textbf{99} (1990), no.~3, 627--649.

\end{thebibliography}

\end{document}